\documentclass[11pt]{amsart}
\usepackage{a4wide}

\usepackage{comment}
\usepackage[alphabetic, initials]{amsrefs}
\usepackage[T1]{fontenc}
\usepackage{amssymb}

\usepackage{enumitem}

\usepackage{tikz}
\usetikzlibrary{decorations.markings, arrows, calc, patterns, arrows.meta, 
bending}

\usetikzlibrary{external}

\usepackage{float}

\usepackage{bbm}
\usepackage{dsfont}
\usepackage{amsbsy}

\usepackage{thmtools}
\usepackage{thm-restate}

\newtheorem{thm}{Theorem}[section]
\newtheorem{cor}[thm]{Corollary}
\newtheorem{lem}[thm]{Lemma}
\newtheorem{prop}[thm]{Proposition}

\theoremstyle{definition}
\newtheorem{defin}[thm]{Definition}

\newtheorem{rem}[thm]{Remark}

\numberwithin{figure}{section}

\newcommand{\lle}{\langle}
\newcommand{\rre}{\rangle}
\newcommand{\lf}{\lfloor}
\newcommand{\rf}{\rfloor}
\newcommand{\eps}{\varepsilon}

\DeclareMathOperator{\Tr}{Tr}

\BibSpec{collection.article}{
    +{}  {\PrintAuthors}                {author}
    +{,} { \textit}                     {title}
    +{.} { }                            {part}
    +{:} { \textit}                     {subtitle}
    +{,} { \PrintContributions}         {contribution}
    +{,} { \PrintConference}            {conference}
    +{}  {\PrintBook}                   {book}
    +{,} { }                            {booktitle}
    +{,} { }                            {publisher}
    +{,} { \PrintDateB}                 {date}
    +{,} { pp.~}                        {pages}
    +{,} { }                            {status}
    +{,} { \PrintDOI}                   {doi}
    +{,} { available at \eprint}        {eprint}
    +{}  { \parenthesize}               {language}
    +{}  { \PrintTranslation}           {translation}
    +{;} { \PrintReprint}               {reprint}
    +{.} { }                            {note}
    +{.} {}                             {transition}
    +{}  {\SentenceSpace \PrintReviews} {review}
}

\title[Random groups in the square model have property $(\textrm{T})$ at densities over $\frac{5}{12}$]{Random groups in the square model \\have property $(\textrm{T})$ at densities over $\frac{5}{12}$}

\author{Damian Orlef}
\address{Institute of~Mathematics of the Polish Academy of~Sciences, ul. \'{S}niadeckich 8, 00-656 Warsaw, Poland}
\email{dorlef@impan.pl}

\author{Tomasz Odrzyg\'{o}\'{z}d\'{z}}
\address{Institute of~Mathematics of the Polish Academy of~Sciences, ul. \'{S}niadeckich 8, 00-656 Warsaw, Poland}
\email{tomeko314@gmail.com}

\begin{document}

\begin{abstract}
A~random group in the square model $\mathcal{Q}(n,d)$ is~given by~a~presentation with~$n$ generators and a~uniformly random set of~$(2n-1)^{4d}$ cyclically reduced relators of~length 4 over them. We~prove that if $d>\frac{5}{12}$, then a random group in $\mathcal{Q}(n,d)$ has Kazhdan's property~$(\textrm{T})$ asymptotically almost surely (a.a.s.). This provides a~new construction of~infinite word-hyperbolic Kazhdan groups when $d\in\left(\frac{5}{12}, \frac{1}{2}\right)$. Moreover, also for $d>\frac{5}{12}$, we show that a~random group in the square model a.a.s.\ has property $(\textrm{F}L^p)$ for an increasing range of values of $p$, and property $(\textrm{F}_X)$
for any fixed uniformly curved Banach space $X$. These results lead to bounds
on the~conformal dimension of the~boundary of a~random group in the square model.

In~the~process we devise a~method of~applying the~spectral link
criteria for~property $(\textrm{T})$, and other fixed point properties, to~the~groups presented by~relations of~length~4. 
In order to use it for the square model, we verify the~required spectral gap hypothesis 
by~the~trace method, in~the~spirit of~the work of Broder--Shamir. 
\end{abstract}

\maketitle

\section{Introduction}

In this article we investigate a~model of random groups, called the~\emph{square model}, and show that it provides (for certain values of 
the~\emph{density} parameter $d$) a~new construction of infinite groups with Kazhdan's property $(\mathrm{T})$, which additionally satisfy various fixed point 
properties for affine isometric actions on Banach spaces. All~groups $\Gamma$ we consider are countable and discrete,
and for such groups we use the~following definitions. Given a~Banach space $X$ (real or complex), we say
that $\Gamma$ has property $(\mathrm{F}_X)$ if every affine 
isometric action of $\Gamma$ on $X$ has a~global fixed point. For $p\geq 1$, we say that $\Gamma$ has property $(\mathrm{F}L^p)$, if it has 
property $(\mathrm{F}_X)$ for every real $L^p$--space $X$. Notably, for any~countable discrete group and any $p\in[1,2]$, 
property $(\textrm{F}L^p)$ is equivalent to property $(\mathrm{T})$ 
(by a~combination of~\cite[Theorems 19.65 and 19.66]{dru18} and \cite[Theorem 2.12.4]{bhv08}).

The~square model of~random groups was first introduced in~\cite{odr16a} and is defined as follows.
Let $d\in(0,1)$ be a~number, called \emph{density}, and let $n\in\mathbb{N}$. 
A~\emph{random group in the square model $\mathcal{Q}(n,d)$} is a~group $\Gamma$, given by (and considered with)
the~presentation $\Gamma=\lle S|R\rre$, where $S$ is a~fixed set of~$n$ (free) generators and $R$ 
is a~uniformly random set of $\left\lfloor (2n-1)^{4d}\right\rfloor$ cyclically reduced words of~length 4
over $S\cup S^{-1}$, i.e.\ words of form $pqrs$, where $p, q, r, s \in S\cup S^{-1}$, $p\neq q^{-1}, q\neq r^{-1}, r\neq s^{-1}$ and 
$s\neq p^{-1}$. Alternatively, we call $\Gamma$ a~\emph{random group in the square model at density $d$ (with $n$ generators)}.

Of usual interest is the case when $d$ is fixed and $n\rightarrow\infty$. Given a~property~$\mathcal{P}$ of groups or~presentations (possibly depending on $n$), we say that a~random group in the square model at density~$d$ (with $n$ generators) \emph{satisfies $\mathcal{P}$ asymptotically almost surely (a.a.s.)} if 

\begin{equation*}
	\lim_{n\rightarrow\infty}
	\mathbb{P}\big(\Gamma=\lle S| R\rre \text{ in }\mathcal{Q}(n,d)\text{ satisfies }\mathcal{P}\big)=1.
\end{equation*}
(More generally, we use the~phrase \emph{asymptotically almost surely} to indicate any situation, in which the~probability of an~event
converges to 1 as $n\rightarrow\infty$.)

It was shown in \cite{odr16a} that the~type of groups obtained in the model varies with the~density as follows. A~random group in the~square model at density $d$ is a.a.s.
\begin{enumerate}[label=(\roman*)]
	\item\label{random_square_p1} non-trivial free if $d<\frac{1}{4}$,
	\item\label{random_square_p2} infinite, word-hyperbolic, torsion-free and of~geometric dimension 2 if ${d\in\left(\frac{1}{4}, \frac{1}{2}\right)}$, 
	\item\label{random_square_p3} isomorphic to $\mathbb{Z}/2\mathbb{Z}$ if $d>\frac{1}{2}$.
\end{enumerate}
Parts~\ref{random_square_p1} and \ref{random_square_p3} are, respectively, \cite[Theorem~4.8]{odr16a} and \cite[Theorem~2.8]{odr16a}.
Part~\ref{random_square_p2} expands on~\cite[Corollary~3.16]{odr16a} by the~following arguments.
The~isoperimetric inequality \cite[Theorem~3.14]{odr16a} implies that, for $d\in\left(\frac{1}{4},\frac{1}{2}\right)$, a~random $\Gamma=\lle S|R\rre$ in $\mathcal{Q}(n,d)$ a.a.s.\ has 
an~aspherical presentation complex and so it is torsion-free, of geometric dimension at most~2. A.a.s.\ its Euler characteristic equals to $1-|S|+|R|=1-n+\lfloor (2n-1)^{4d}\rfloor\rightarrow \infty$, so~in~particular $\Gamma$ is not quasi-isometric to any~free group, including $\mathbb{Z}$, and is hence non-elementary hyperbolic, of geometric dimension exactly 2.

The~square model is the~case of $k=4$ for a~more general construction, which we call the~\emph{$k$-gonal model}, where the~set of generators $S$ has $n$ elements and the~set of relators $R$ is a~random set of $\lfloor (2n-1)^{kd}\rfloor$ cyclically reduced words of fixed length $k$ over $S\cup S^{-1}$. For $k=3$, this is the~\emph{triangular model}, introduced
in \cite{zuk03}. The~case of general $k$ was first studied in \cite{ash20} under the name \emph{standard $k$-angular model}. For each $k$, the $k$-gonal model
can be seen as a~simplification of (or a~different perspective on) the~\emph{Gromov density model}, originating in \cite{gro93}.

Famously, for densities $d\in(\frac{1}{3}, \frac{1}{2})$, random groups in the~triangular model and in the~Gromov density model are a.a.s.\ infinite and have property $(\textrm{T})$
(cf. \cite{zuk03}, \cite{kot13}, \cite{ash23} for the~discussion of~the~proof and \cite{als15} for a~more detailed analysis of the~case $d=\frac{1}{3}$ in~the~triangular model). 
This fact is established by means of a~spectral criterion attributed to \.{Z}uk \cite{zuk03} and Ballmann--\'Swi\k{a}tkowski \cite{bal97}, tracing back to the~ideas of~Garland \cite{gar73}. By~generalising this method, additional properties of form $(\textrm{F}_X)$, for Banach spaces~$X$, (including properties $(\textrm{F}L^p)$) have also been established for random groups in the~triangular model and in the~Gromov density model --- see \cite{now15}, \cite{dru19}, \cite{dLdlS21}, \cite{opp23a}, \cite{opp23b}.

Infinite groups with property $(\textrm{T})$ have been also shown to occur, for various ranges of $d$, in~the~$k$-gonal model for $k=6$ \cite{odr25} and $k\geq 8$ \cite{ash23} (see also \cite{mon22}).

Regarding~the~case of $k=4$, it was proven in \cite{odr25} that a~random group in the square model at density $d<\frac{3}{8}$ a.a.s.\ does not have property $(\mathrm{T})$. Here, for the first time, we show that there exist densities $d<\frac{1}{2}$, such that a~random group in the square model at density $d$ a.a.s.\ has property $(\mathrm{T})$. Our main results are as follows.
\begin{thm}\label{main_thm_T}
	If $d>\frac{5}{12}$, then a~random group in the~square model at density $d$ has property~$(\mathrm{T})$ a.a.s.
\end{thm}
\begin{thm}\label{main_thm_FLp}
	For every $d>\frac{5}{12}$, there exists a~constant $f_d>0$, such that a~random group in the~square model
	at density $d$, with $n$ generators
	a.a.s.\ has property $(\mathrm{F}L^p)$ for every 
	${p\in \left[1, f_d\left(\log{n}\right)^{\frac{1}{2}}\right]}$.
\end{thm}
\begin{thm}\label{main_thm_F_X}
	Suppose $X$ is a~uniformly curved Banach space and $d>\frac{5}{12}$. Then a~random group in the~square model
	at density $d$ has property $\left(\mathrm{F}_X\right)$ a.a.s.
\end{thm}

We note that Theorem~\ref{main_thm_T} is a~direct consequence of Theorem~\ref{main_thm_FLp}, by previous considerations.

\subsection{Application to conformal dimension}\label{intro_confdim_subsec}

Given a~finitely generated hyperbolic group~$\Gamma$, we denote 
by $\partial\Gamma$ its Gromov boundary ---
it is a~topological space with a~natural family of 
quasi-symmetrically equivalent metrics (visual metrics). We define the conformal dimension 
of $\partial\Gamma$, denoted $\operatorname{Confdim}
\left(\partial\Gamma\right)$, as the infimum of~Hausdorff 
dimensions of~metric spaces $Y$ which  
are quasi-symmetrically equivalent to $\partial\Gamma$.
This number is a~quasi-isometry invariant of~$\Gamma$ among
hyperbolic groups \cite[th\'eor\`eme 1.6.4]{bou95}. 
See \cite{mac10} for more details.

Using~Theorem~\ref{main_thm_FLp}, we can repeat the reasoning
of~\cite[Subsection 10.2]{dru19}, designed for random groups in the triangular model, to obtain two-sided bounds on the~conformal dimension of random groups in the square model. Namely, in Section~\ref{bound_conf_dim_sec}, we prove the following.

\begin{thm}\label{confdim_bounds_square_thm}
For every $d\in\left(\frac{5}{12},\frac{1}{2}\right)$, there exist constants $f_d, F_d>0$ with the following property. The~conformal dimension of the~Gromov boundary of~a~random group $\Gamma$ in the~square model at density $d$, with $n$ generators, a.a.s.\ satisfies the bounds:
\begin{equation*}
	f_d(\log{n})^{\frac{1}{2}}\leq\operatorname{Confdim}(\partial\Gamma)\leq F_d\log{n}.
\end{equation*}
\end{thm}
The constants $f_d$ are the same as in Theorem~\ref{main_thm_FLp}.
A~consequence of Theorem~\ref{confdim_bounds_square_thm}
is that the~quasi-isometry types of random groups in the square 
model at given density $d\in\left(\frac{5}{12}, \frac{1}{2}\right)$ keep 
changing as $n\rightarrow\infty$. Precisely speaking, we have:

\begin{cor}
Suppose $d\in\left(\frac{5}{12}, \frac{1}{2}\right)$. Then there exist:\ an increasing sequence of natural numbers
$(n_{k,d})_{k=1}^{\infty}$ and a~sequence $(Q_{n,d})_{n=1}^\infty$ of sets of (isomorphism classes of) finitely generated groups, satisfying the 
following properties. For every $n\geq 1$, a random group in the model $\mathcal{Q}\left(n, d\right)$ a.a.s.\ belongs to $Q_{n,d}$, 
and for all numbers $k\geq 1$, $i\leq k$, and $j\geq n_{k,d}$, 
no group in $Q_{i,d}$ is quasi-isometric to a~group in $Q_{j,d}$.
\end{cor}
\begin{proof}
For $n\geq 1$, let $Q_{n,d}$ be the~set of finitely generated hyperbolic groups $\Gamma$ which satisfy the inequalities of Theorem~\ref{confdim_bounds_square_thm}. Then it suffices to choose any increasing sequence $(n_{k,d})_{k=1}^\infty$, such that
$F_d\log{k} < f_d\left(\log{n_{k,d}}\right)^{\frac{1}{2}}$
for $k \geq 1$.

\end{proof}
 
\subsection{Outline of the article and the main ideas}
\label{outline_subsec}

The~rest of the~article is structured as follows.

In Section~\ref{prelim_sec} we recall 
some basic facts about expander graphs and uniformly curved Banach spaces. There we also 
describe a~few pieces of notation and conventions we use.

Section~\ref{criterion_sec} describes the~central idea of the~article. Proofs of Theorems~\ref{main_thm_T}, \ref{main_thm_F_X} and \ref{main_thm_FLp} rely on the following construction.
Suppose $\Gamma = \lle S|R \rre$ is a~group presentation with a~finite generating set $S$, where $R$ contains
``many'' relators of length 4, and we want to 
show that $\Gamma$ has property $(\textrm{F}_X)$ for some uniformly curved Banach space $X$.
We call a~triple $\tau=(r_1, r_2, r_3)$, of cyclically reduced words of length 4 over $S\cup S^{-1}$, a~\emph{triagram over $S$} if 
the words are of form $r_1=x_1x_2x_3x_4$, $r_2=y_1y_2y_3y_4$ and $r_3=z_1z_2z_3z_4$, where $x_i, 
y_i, z_i \in S\cup S^{-1}$ for $1\leq i\leq 4$ and the following conditions are satisfied:
\begin{itemize}
	\item $x_4=y_1^{-1}$, $y_4=z_1^{-1}$, $z_4=x_1^{-1}$,
	\item $x_3\neq y_2^{-1}$, $y_3\neq z_2^{-1}$, $z_3\neq x_2^{-1}$.
\end{itemize}
These conditions imply that the word $\partial(\tau)=(y_3z_2)(z_3x_2)(x_3y_2)$ is the~boundary word 
of~the~van Kampen diagram shown in the centre of~Figure~\ref{triag_constr1}. Hence, if additionally ${r_1, r_2, r_3\in R}$, then $\partial(\tau)$ is a~word representing 
the~trivial element of $\Gamma$, which can be seen as a~cyclically reduced word of length 3 over $W_2(S) = \{st : s,t\in S\cup S^{-1},\,s\neq t^{-1}\}$.
\begin{figure}[H]
	\centering
	\begin{tikzpicture}[line cap = round, line join = round]

\def\midarr{0.5}
\def\boundarr{0.5}
\tikzset{boundedgestyle/.style={
        thick,
        decoration={
            markings,
            mark=at position \boundarr with {\arrow{angle 45}}
        },
        postaction={decorate}
}}
\tikzset{midedgestyle/.style={
        thick,
        decoration={
            markings,
            mark=at position \midarr with {\arrow{angle 45}}
        },
        postaction={decorate}
}}

\def\l{1}

\def\s{1}
\def\t{1.732}
\def\w{0.577}
\def\txt{0.25}

\def\p{0}
\def\q{0}

\def\r{0.15*\l}

\def\xa{\l*\p}
\def\ya{\l*\q-\l*2*\w*\r}
\draw[midedgestyle] (\xa, \ya) -- (\xa-\l*\t, \ya-\l*\s);
\node at (\xa-\midarr*\l*\t+\txt*0.5*\l*\s, \ya-\midarr*\l*\s-\txt*0.5*\t*\l) {$x_1$};
\draw[boundedgestyle] (\xa-\l*\t, \ya-\l*\s) -- (\xa, \ya-2*\l*\s);
\node at (\xa-\l*\t+\boundarr*\l*\t-\txt*0.5*\l*\s, \ya-\l*\s-\boundarr*\l*\s-\txt*0.5*\t*\l) {$x_2$};
\draw[boundedgestyle] (\xa, \ya-2*\l*\s) -- (\xa+\l*\t, \ya-\l*\s);
\node at (\xa+\boundarr*\l*\t+\txt*0.5*\l*\s, \ya-2*\l*\s+\boundarr*\l*\s-\txt*0.5*\t*\l) {$x_3$};
\draw[midedgestyle] (\xa+\l*\t, \ya-\l*\s) -- (\xa, \ya);
\node at (\xa+\l*\t-\midarr*\l*\t-\txt*0.5*\l*\s, \ya-\l*\s+\midarr*\l*\s-\txt*0.5*\t*\l) {$x_4$};

\def\xb{\l*\p+\l*\r}
\def\yb{\l*\q+\l*\w*\r}
\draw[midedgestyle] (\xb, \yb) -- (\xb+\l*\t, \yb-\l*\s);
\node at (\xb+\midarr*\l*\t+\txt*0.5*\l*\s, \yb-\midarr*\l*\s+\txt*0.5*\t*\l) {$y_1$};
\draw[boundedgestyle] (\xb+\l*\t, \yb-\l*\s) -- (\xb+\l*\t, \yb+\l*\s);
\node at (\xb+\l*\t+\txt*\l*\s, \yb-\l*\s+\boundarr*2*\l*\s) {$y_2$};
\draw[boundedgestyle] (\xb+\l*\t, \yb+\l*\s) -- (\xb, \yb+2*\l*\s);
\node at (\xb+\l*\t-\boundarr*\l*\t+\txt*0.5*\l*\s, \yb+\l*\s+\boundarr*\l*\s+\txt*0.5*\t*\l) {$y_3$};
\draw[midedgestyle] (\xb, \yb+2*\l*\s) -- (\xb, \yb);
\node at (\xb+\txt*\l*\s, \yb+2*\l*\s-\midarr*2*\l*\s) {$y_4$};

\def\xc{\l*\p-\l*\r}
\def\yc{\l*\q+\l*\w*\r}
\draw[midedgestyle] (\xc, \yc) -- (\xc, \yc+2*\l*\s);
\node at (\xc-\txt*\l*\s, \yc+\midarr*2*\l*\s) {$z_1$};
\draw[boundedgestyle] (\xc, \yc+2*\l*\s) -- (\xc-\l*\t, \yc+\l*\s);
\node at (\xc-\boundarr*\l*\t-\txt*0.5*\l*\s, \yc+2*\l*\s-\boundarr*\l*\s+\txt*0.5*\t*\l) {$z_2$};
\draw[boundedgestyle] (\xc-\l*\t, \yc+\l*\s) -- (\xc-\l*\t, \yc-\l*\s);
\node at (\xc-\l*\t-\txt*\l*\s, \yc+\l*\s-\boundarr*2*\l*\s) {$z_3$};
\draw[midedgestyle] (\xc-\l*\t, \yc-\l*\s) -- (\xc, \yc);
\node at (\xc-\l*\t+\midarr*\l*\t-\txt*0.5*\l, \yc-\l*\s+\midarr*\l*\s+\txt*0.5*\t*\l) {$z_4$};

\def\q{0}
\def\p{5}
\def\r{0}

\def\xa{\l*\p}
\def\ya{\l*\q-\l*2*\w*\r}
\draw[midedgestyle] (\xa, \ya) -- (\xa-\l*\t, \ya-\l*\s);
\node at (\xa-\boundarr*\l*\t+\txt*0.5*\l*\s, \ya-\boundarr*\l*\s-\txt*0.5*\t*\l) {$x_1$};
\draw[boundedgestyle] (\xa-\l*\t, \ya-\l*\s) -- (\xa, \ya-2*\l*\s);
\node at (\xa-\l*\t+\boundarr*\l*\t-\txt*0.5*\l*\s, \ya-\l*\s-\boundarr*\l*\s-\txt*0.5*\t*\l) {$x_2$};
\draw[boundedgestyle] (\xa, \ya-2*\l*\s) -- (\xa+\l*\t, \ya-\l*\s);
\node at (\xa+\boundarr*\l*\t+\txt*0.5*\l*\s, \ya-2*\l*\s+\boundarr*\l*\s-\txt*0.5*\t*\l) {$x_3$};

\def\xb{\l*\p+\l*\r}
\def\yb{\l*\q+\l*\w*\r}
\draw[midedgestyle] (\xb, \yb) -- (\xb+\l*\t, \yb-\l*\s);
\node at (\xb+\boundarr*\l*\t+\txt*0.5*\l*\s, \yb-\boundarr*\l*\s+\txt*0.5*\t*\l) {$y_1$};
\draw[boundedgestyle] (\xb+\l*\t, \yb-\l*\s) -- (\xb+\l*\t, \yb+\l*\s);
\node at (\xb+\l*\t+\txt*\l*\s, \yb-\l*\s+\boundarr*2*\l*\s) {$y_2$};
\draw[boundedgestyle] (\xb+\l*\t, \yb+\l*\s) -- (\xb, \yb+2*\l*\s);
\node at (\xb+\l*\t-\boundarr*\l*\t+\txt*0.5*\l*\s, \yb+\l*\s+\boundarr*\l*\s+\txt*0.5*\t*\l) {$y_3$};

\def\xc{\l*\p-\l*\r}
\def\yc{\l*\q+\l*\w*\r}
\draw[midedgestyle] (\xc, \yc) -- (\xc, \yc+2*\l*\s);
\node at (\xc-\txt*\l*\s, \yc+\boundarr*2*\l*\s) {$z_1$};
\draw[boundedgestyle] (\xc, \yc+2*\l*\s) -- (\xc-\l*\t, \yc+\l*\s);
\node at (\xc-\boundarr*\l*\t-\txt*0.5*\l*\s, \yc+2*\l*\s-\boundarr*\l*\s+\txt*0.5*\t*\l) {$z_2$};
\draw[boundedgestyle] (\xc-\l*\t, \yc+\l*\s) -- (\xc-\l*\t, \yc-\l*\s);
\node at (\xc-\l*\t-\txt*\l*\s, \yc+\l*\s-\boundarr*2*\l*\s) {$z_3$};

\def\q{0}
\def\p{10}
\def\r{0}
\def\xa{\l*\p}
\def\ya{\l*\q-\l*2*\w*\r}
\def\xb{\l*\p+\l*\r}
\def\yb{\l*\q+\l*\w*\r}
\def\xc{\l*\p-\l*\r}
\def\yc{\l*\q+\l*\w*\r}
\draw[boundedgestyle] (\xa, \ya-2*\l*\s) -- (\xb+\l*\t, \yb+\l*\s);
\node[rotate=60] at (0.5*\xa+0.5*\xb+0.5*\l*\t+1.25*\txt*0.5*\t*\l, 0.5*\ya-\l*\s+0.5*\yb+0.5*\l*\s-1.25*\txt*0.5*\l) {$x_3y_2$};
\draw[boundedgestyle] (\xb+\l*\t, \yb+\l*\s) -- (\xc-\l*\t, \yc+\l*\s);
\node at (0.5*\xb+0.5*\l*\t+0.5*\xc-0.5*\l*\t, 0.5*\yb+0.5*\l*\s+0.5*\yc+0.5*\l*\s+\txt*\l) {$y_3z_2$};
\draw[boundedgestyle] (\xc-\l*\t, \yc+\l*\s) -- (\xa, \ya-2*\l*\s);
\node[rotate=-60] at (0.5*\xc-0.5*\l*\t+0.5*\xa-1.25*\txt*0.5*\t*\l, 0.5*\yc+0.5*\l*\s+0.5*\ya-\l*\s-1.25*\txt*0.5*\l) {$z_3x_2$};
\end{tikzpicture}
	\vspace{1em}
	\caption{Obtaining relators of length 3 over $W_2(S)$.}
\label{triag_constr1}
\end{figure}
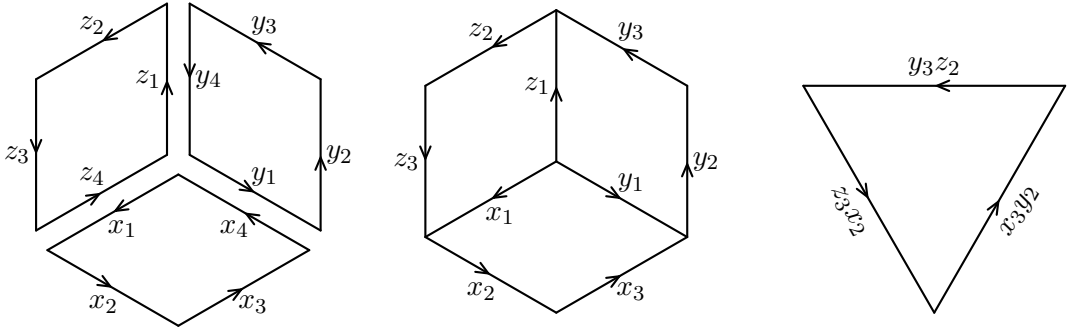
Suppose $\mathcal{T}\subseteq R^3$ is a~set of triagrams. 
Partition $W_2(S)=U\sqcup U^{-1}$ for some~subset $U$ and let $\widetilde{R}$ be the~set of words of form $\partial(\tau)$, where 
$\tau\in\mathcal{T}$, viewed as words of length 3 over $W_2(S)$.
Then the~group $\widetilde{\Gamma}=\lle U | \widetilde{R}\rre$ maps onto a~finite-index subgroup of $\Gamma$ (via a~homomorphism expanding each $u\in U$ 
into the corresponding product of 2 elements of $S\cup S^{-1}$). Hence, to prove that $\Gamma$ has property $(\textrm{F}_X)$, it 
suffices to show that $\widetilde{\Gamma}$ does. In appropriate cases the~latter can be achieved by using \cite[Theorem A]{dLdlS21} --- a~criterion well-suited 
for groups with triangular presentations, which reduces the~task to proving that 
a~certain link graph $L(S,\mathcal{T})$, called the~\emph{triagram link graph} and 
defined in terms of $W_2(S)$ and $\mathcal{T}$, is a~two-sided $\eps$-spectral expander 
for sufficiently small $\eps$. We summarize this reasoning as Theorem~\ref{crit_unif_curved} and formulate a~similar criterion (with a~weaker type of assumption) for property~$(\textrm{T})$, Theorem~\ref{crit_kazhdan}.

In Section~\ref{fix_sec} we prove Theorems~\ref{main_thm_T}, \ref{main_thm_F_X} and \ref{main_thm_FLp} by showing that $L=L(S,\mathcal{T})$ indeed satisfies the~required conditions a.a.s.\ if $\Gamma=\lle S|R\rre$ is a~random group in the~square model $\mathcal{Q}(n,d)$ for $d>\frac{5}{12}$, for some choice of $\mathcal{T}\subseteq R^3$ (Proposition~\ref{full_expander_prop}). To simplify calculations,
in Subsection \ref{binomial_subsec} we observe that it is sufficient to establish this for $\Gamma$ being a~random group in the closely related \emph{binomial square model $\mathcal{Q}(n,\rho)$}.
The~graph $L$ is then a~random graph, which can be naturally divided into 3 parts: $L_1, L_2$ and $L_3$. We show that $L$ is a.a.s.\ a~two-sided $\eps$-spectral expander
for $\eps\rightarrow 0$ by proving that each $L_i$, for $1\leq i \leq 3$, 
has the~same kind of property (Lemma~\ref{part_expander_lem}) and satisfies a~certain asymptotic regularity condition (Lemma~\ref{regularity_lem}). To prove Lemma~\ref{part_expander_lem} (in Subsection~\ref{part_expander_proof_subsec}), we assume Lemma~\ref{regularity_lem} and  
use the~trace method, as described e.g.\ in \cite{bro87}: we consider the~expected value of a~power of the~Markov operator $L_i$, denoted $A_{L_i}$, and we use it to show that there is not enough ``room'' for $A_{L_i}$ to have eigenvalues ``far'' from 0, besides a~simple eigenvalue 1. 

In Section~\ref{asymptotic_regularity_sec} we give the~postponed proof of Lemma~\ref{regularity_lem} by using a~concentration inequality of Kim--Vu \cite{kim00}.

Section~\ref{bound_conf_dim_sec} is dedicated
to showing Theorem~\ref{confdim_bounds_square_thm}, which gives two-sided bounds on 
the~conformal dimension of the~boundary of~a~random group in the~square 
model $\mathcal{Q}(n,d)$ for $d\in\left(\frac{5}{12}, \frac{1}{2}\right)$.

\subsection{Human authorship declaration}
The entirety of this article, including all mathematical arguments and insights, content of the~text and figures, has been produced by the~authors without any~use of generative artificial intelligence tools.

\subsection{Acknowledgements} We would like to thank Piotr Przytycki for introducing us to the~topic of 
this article, for providing us with numerous valuable mathematical insights and editorial 
comments, and for hosting us during a~couple of research visits to the~McGill University.
We would also like to thank John M. Mackay and Piotr Nowak for interesting discussions and useful suggestions.

The first author was partially supported by the European Research Council (ERC) under the European Union’s Horizon 2020 research and innovation programme (grant agreement no. 
677120-INDEX) and by (Polish) Narodowe Centrum Nauki, UMO-2018/30/M/ST1/00668.

\section{Preliminaries and notation}\label{prelim_sec}
\subsection{Expanders}
The graphs that we consider in this article are finite, undirected, possibly with multiple edges, but typically without loops.
Formally, here a~\emph{graph} is a~triple $G=(V,E,\nu)$, where $V\neq \emptyset$ (the \emph{vertex set}), $E$ (the \emph{edge set}) are finite 
disjoint sets and ${\nu : E \longrightarrow \{\{x,y\} : x, y\in V\}}$ is a~map assigning sets of
endpoints to edges. 
We say that an~edge~$e$ is \emph{adjacent} to a vertex $x$ if $x\in \nu(e)$. We call an edge $e$ a~\emph{loop} if $|\nu(e)|=1$. For none of the graphs defined in this article $\nu$ is explicitly stated, as it is always clear from the~context.
For $x,y \in V$ denote by $\omega_G(x,y)$ the number of edges $e\in E$, such that $\nu(e)=\{x,y\}$, i.e.\ ~edges~$e$ with endpoints $x$ and $y$. We call the function $\omega_G : V^2 \longrightarrow \mathbb{N}$ the~\emph{weight function} of~$G$.
The~\emph{degree} of a~vertex $x\in V$ in the~graph $G$ is the number $d_G(x)=\sum_{y\in V} \omega_G(x,y)$, i.e.\ ~the number of edges adjacent to $x$, and we call the resulting function $d_G:V\longrightarrow \mathbb{N}$ the~\emph{degree function} of $G$. We say that a~vertex $x$ is \emph{isolated}
if $d_G(x)=0$.

For the~remainder of this subsection, assume that $G$ is a~graph with no isolated vertices and no loops. Denote by $L^2(V, d_G)$ the~vector space of complex-valued 
functions on $V$, equipped with the scalar product $\lle\cdot,\cdot \rre_G$, given by
\begin{equation*}
\lle f, g \rre_G = \sum_{x\in V} d_G(x)f(x)\overline{g(x)}
\end{equation*}
for $f,g\in L^2(V,d_G)$. In other words, $L^2(V,d_G)$ is the~$L^2$--space of $V$ with the~measure defined by $d_G$. 
The~\emph{Markov operator} of $G$ is the operator $A_G : L^2(V,d_G) \longrightarrow L^2(V,d_G)$, defined by
\begin{equation*}
(A_Gf)(x) = \frac{1}{d_G(x)}\sum_{y\in V} \omega_G(x,y)f(y)
\end{equation*}
for $f\in L^2(V,d_G)$, $x\in V$.
Using the fact that $\omega_G(x,y)=\omega_G(y,x)$, it is easy to check that for any $f,g\in L^2(V,d_G)$ we have
\begin{equation*}
    \lle A_Gf, g\rre_G = \sum_{x, y\in V} \omega_G(x,y)f(x)\overline{g(y)} = \overline{\lle f, A_Gg\rre}_G.
\end{equation*}
$A_G$ is thus self-adjoint, so it has $N=|V|$ real eigenvalues and an orthonormal eigenbasis. We denote eigenvalues of $A_G$ by
$\lambda_1(A_G)\geq \lambda_2(A_G)\geq\ldots\geq\lambda_N(A_G)$. All of them lie in the interval $[-1,1]$ and $\lambda_1(A_G)=1$, which corresponds to the eigenvector $\pmb{1}\in L^2(V, d_G)$, defined by $\pmb{1}(x)=1$ for $x\in V$. Moreover, $G$ is connected if and only if 1 is a~simple eigenvalue of $A_G$, i.e.\ if and only if $\lambda_2(A_G)<1$.

Let $\lambda>0$. Similarly as in \cite{opp23a}, given a~graph $G$ without isolated vertices and loops,
we say that $G$ is a~\emph{one-sided $\lambda$-spectral expander} 
if $\lambda_2(A_G)\leq \lambda$, and we say that $G$ is a~\emph{two-sided $\lambda$-spectral expander}
if $|\lambda_i(A_G)|\leq\lambda$ for $2\leq i\leq N$ (equivalently,
$|\lambda_2(A_G)|, |\lambda_N(A_G)|\leq\lambda$). 

Denote $L^2_0(V, d_G) = \left\{ f\in L^2(V,d_G) : \lle f, \pmb{1} \rre_G = 0 \right\}$.
By considering an orthonormal eigenbasis of $A_G$, containing a~scalar multiple of $\pmb{1}$, we arrive at the~following variational characterisation.
\begin{rem}\label{expand_min_max_rem}
	Suppose $G$ is a graph with no isolated vertices and no loops. Then $G$ is a~two-sided $\lambda$-spectral expander if and only if
	\begin{equation*}
	|\lle A_Gf, f\rre_G| \leq \lambda\|f\|_G^2
	\end{equation*}
	for all $f\in L^2_0(V, d_G)$.
\end{rem}

\subsection{Uniformly curved Banach spaces}
Uniformly curved Banach spaces were introduced in \cite{pis10}. 
For example, all $L^p$--spaces for $1<p<\infty$ are uniformly curved. This is a~class of~spaces, for which we can apply
\cite[Theorem A]{dLdlS21} and beyond that our arguments do not rely on their definition.
The~following form of~the definition is taken directly from~\cite{dLdlS21}.

Suppose $X$ is a Banach space and $T:L^2(\Omega_1,\mu_1)\longrightarrow L^2(\Omega_2,\mu_2)$ is an operator between $L^2$--spaces. 
If $T\otimes \text{id}_X$ extends to a~bounded operator $T_X : L^2(\Omega_1, \mu_1; X) \longrightarrow L^2(\Omega_2, \mu_2; X)$,
then let $\|T_X\|$ be the norm of the extension. If such an~extension does not exist, let ${\|T_X\|=\infty}$. For fixed $X$ and $\eps>0$,
denote by $\Delta_X(\eps)$ the~supremum of~values of~$\|T_X\|$ over any measure spaces $(\Omega_1, \mu_1)$, $(\Omega_2, \mu_2)$
and operators $T:L^2(\Omega_1, \mu_1)\longrightarrow L^2(\Omega_2, \mu_2)$, satisfying conditions $\|T:L^1(\Omega_1, \mu_1)\longrightarrow L^1
(\Omega_2, \mu_2)\|\leq 1$, $\|T:L^\infty(\Omega_1,\mu_1)\longrightarrow L^\infty(\Omega_2,\mu_2)\|\leq 1$ and ${\|T:L^2(\Omega_1,\mu_1)\longrightarrow L^2
(\Omega_2,\mu_2)\|\leq \eps}$.

\begin{defin}
A Banach space $X$ is \emph{uniformly curved} if $\Delta_X(\eps)\rightarrow 0$ as $\eps\rightarrow 0$.
\end{defin}
For further examples and discussion of uniformly curved spaces, see \cite{dLdlS21} and \cite{pis10}.

\subsection{General notation}
Many of the named objects and quantities in this article depend implicitly on the number $n$ of generators. 

We use standard $O(\cdot)$ and $o(\cdot)$ asymptotic notations. Suppose $q$ is a~quantity, possibly depending on $n$ and other parameters. We write $O(q)$ in place of an~unnamed quantity $r$, such that $|r|\leq C|q|$ 
for a~constant $C>0$. Similarly, we denote by $o(q)$ a~quantity $r$, such that $|r|\leq f(n)|q|$, 
where $f : \mathbb{N} \longrightarrow \mathbb{R}$ is a~function satisfying $f(n)\rightarrow 0$ as $n\rightarrow\infty$. If $C$ (resp.\ definition of~$f$) depends only on parameters $p_1,\ldots, p_k$, we may write $O_{p_1,\ldots,p_k}(q)$ (resp.\ $o_{p_1,\ldots,p_k}(q)$) instead of $O(q)$ (resp.\ $o(q)$). If $C$ (resp.\ $f$) depends only on the place in the text where the symbol $O(q)$ (resp.\ $o(q)$) is used,
then we may write $O_\text{abs}(q)$ (resp.\ $o_\text{abs}(q)$) instead.
Given functions $f, g : \mathbb{N} \longrightarrow \mathbb{R}$, we write $f\sim g$ if
$f(n) = (1+o(1))g(n)$ for sufficiently large $n$.

By $\log x$ we denote the~natural logarithm of $x$.
Given a~set $C$, we denote by $|C|$ the number of elements of $C$.
Suppose $A$ is an~event in a~probability space. By $\mathds{1}_A$ we denote the~indicator of~$A$, i.e. the~random variable $\mathds{1}_A$
such that $\mathds{1}_A=1$ if $A$ holds and $\mathds{1}_A=0$, otherwise. By $A^c$
we denote the complement of $A$.  

\section{Spectral criterion for square groups}\label{criterion_sec}
In this section we formulate a~spectral criterion for a~group presented by relators of 
length~4 to have property $(\mathrm{F}_X)$ for a~given uniformly curved Banach space $X$, and
a~similiar criterion for property $(\textrm{T})$. Our statements follow from previously known results of this type for groups with triangular presentations. First we introduce some notation.

Suppose $S=\{s_1,s_2,\ldots,s_n\}$ is a~finite set
and let $F(S)$ be the free group generated by $S$. Denote $S^{-1}=\{s_1^{-1},s_2^{-1},\ldots, s_n^{-1}\}\subseteq F(S)$. For $m\geq 1$, denote by $W_m(S)\subseteq F(S)$ the~set 
of~cyclically reduced words of~length $m$ over $S\cup S^{-1}$. 

\begin{defin}
	We call a~triple $\tau=(r_1, r_2, r_3)$ of words in $W_4(S)$ a~\emph{triagram over $S$}
	if~the~words are of form $r_1=x_1x_2x_3x_4$, $r_2=y_1y_2y_3y_4$, $r_3=z_1z_2z_3z_4$, 
	where $x_i, y_i, z_i \in S\cup S^{-1}$ for $1\leq i\leq 4$, and the following conditions 
	are satisfied:
	\begin{itemize}
		\item $x_4=y_1^{-1}, y_4=z_1^{-1}, z_4=x_1^{-1}$,
		\item $x_3y_2, y_3z_2, z_3x_2\in W_2(S)$, i.e. $x_3\neq y_2^{-1}$, $y_3\neq z_2^{-1}$, $z_3\neq x_2^{-1}$.
	\end{itemize}
	In such case, we define $\partial_1(\tau)=y_3z_2$, $\partial_2(\tau)=z_3x_2$, $\partial_3(\tau)=x_3y_2$ and $\partial(\tau)=\partial_1(\tau)\partial_2(\tau)\partial_3(\tau)$. By $T(S)$ we denote the set of all triagrams over $S$. 
\end{defin}

As discussed in Subsection~\ref{outline_subsec}, if $\Gamma=\lle S|R\rre$ is a~group 
presentation and $\tau\in R^3$ is a~triagram over $S$, then the~word $\partial(\tau)$
represents the~trivial element of $\Gamma$. Alternatively, this fact can be seen by explicitly writing
\begin{equation*}
	\partial(\tau)=(y_3z_2)(z_3x_2)(x_3y_2)=(y_1y_2)^{-1}(r_2r_3r_1)(y_1y_2).
\end{equation*}

For any $\tau\in T(S)$, the~product $\partial(\tau)=\partial_1(\tau)\partial_2(\tau)\partial_3(\tau)$ can be seen as a~word of length 
3 over $W_2(S)$, which is cyclically reduced in the sense that $\partial_2(\tau)\neq \partial_1(\tau)^{-1}$,
$\partial_3(\tau)\neq\partial_2(\tau)^{-1}$ and $\partial_1(\tau) \neq \partial_3(\tau)^{-1}$.

To formulate our spectral criteria, we will be using the~following graph construction. As~we explain 
later, it produces the link graph of (a variant of) the Cayley 
complex of a~presentation, in which $W_2(S)$ is the set of generators and their inverses,
and relators are of form $\partial_1(\tau)\partial_2(\tau)\partial_3(\tau)$, where $\tau$ is a 
triagram.

\begin{defin}\label{def_triagram_link_graph}
Suppose $\mathcal{T}$ is a~set of triagrams over a~finite set of generators $S$. The~\emph{triagram link graph} $L(S,\mathcal{T})$ is the~graph with the~vertex set $W_2(S)$ and the~edge set constructed as follows.

At the~start, let $E_1, E_2, E_3$ be all empty sets of edges between vertices in $W_2(S)$.

Then, for each triagram $\tau\in\mathcal{T}$ in turn, introduce the following edges:
\begin{itemize}
	\item a~new edge in $E_1$ between the~vertices $\partial_2(\tau)^{-1}$ and $\partial_3(\tau)$,
	\item a~new edge in $E_2$ between the~vertices $\partial_3(\tau)^{-1}$ and $\partial_1(\tau)$,
	\item a~new edge in $E_3$ between the~vertices $\partial_1(\tau)^{-1}$ and $\partial_2(\tau)$.
\end{itemize}

Finally, let the edge set of $L(S,\mathcal{T})$ be $E=E_1\sqcup E_2\sqcup E_3$. 

Additionally, for $1\leq i\leq 3,$ denote by $L_i(S,\mathcal{T})$ the graph with the vertex set $W_2(S)$ and the edge set $E_i$.
\end{defin}

The following theorem is our criterion for property $(\mathrm{F}_X)$ of groups presented by 
relations of~length~4, which is the main tool of this article. It is an analogue and, as we will show, a~direct consequence of
\cite[Theorem~A with Remark 5.2(i)]{dLdlS21}.
\begin{thm}\label{crit_unif_curved}
	For every uniformly curved Banach space $X$, there exists a~constant $\eps(X)>0$ with the following property.
	
	Suppose $\Gamma=\lle S|R\rre$ is a~group presentation with a~finite generating set $S$. If there exists 
	a~set of triagrams $\mathcal{T}\subseteq R^3\cap T(S)$,
	such that the~triagram link 
	graph $L(S,\mathcal{T})$ is a~two-sided $\eps$-spectral expander for some $\eps<\eps(X)$, then $\Gamma$ has property $(\mathrm{F}_X)$.
	
	If $X$ is an $L^p$--space with $p\geq 2$, then one can take $\eps(X)=2p^{-\frac{p}{2}}2^{-\frac{p^2}{2}}$.
\end{thm}

A~similar criterion, with just one-sided spectral assumption on $L(S,\mathcal{T})$, can be given for property $(\mathrm{T})$ of $\Gamma$.
Specifically, we have the~following theorem, which can be seen as a~version
of a~classical spectral criterion (see e.g.\ \cite[Proposition 6]{zuk03}). We state and prove it for its own sake, as
it will not be used in this article.

\begin{thm}\label{crit_kazhdan}
	Suppose $\Gamma=\lle S|R\rre$ is a~group presentation with a~finite generating set $S$. If~there exists 
	a~set of triagrams $\mathcal{T}\subseteq R^3\cap T(S)$, such that
	the~triagram link graph $L(S,\mathcal{T})$ is a~one-sided $\eps$-spectral expander for some $\eps<\frac{1}{2}$, then $\Gamma$ has Kazhdan's property~$(\textrm{T})$.
\end{thm}

To prove Theorems~\ref{crit_unif_curved} and \ref{crit_kazhdan}, we are going to construct a~few auxilliary objects. From now on let $\Gamma = \lle S| R\rre$ be a~group presentation with a~finite generating set $S$ and let $\mathcal{T}\subseteq R^3\cap T(S)$ be a~set of triagrams.
As the involution of $W_2(S)$ sending $w$ to $w^{-1}$ has no fixed points, so there exists 
a~partition $W_2(S) = U \sqcup U^{-1}$ for some (non-unique) subset $U\subseteq W_2(S)$. Let $F(U)$ be the 
free group generated by the set $U$ and let $\phi : F(U) \longrightarrow F(S)$ be the~homomorphism obtained by 
expanding elements of $U$ into words of length 2 over $S\cup S^{-1}$. The image of $\phi$ has index 2
in $F(S)$. For any $w\in W_2(S)$, denote by $\widetilde{w}$ the unique element 
of $U \cup U^{-1} \subseteq F(U)$, for which $\phi\left(\widetilde{w}\right)=w$.
Consider any triagram $\tau\in T(S)$ and define $\widetilde{\partial}(\tau)
=\widetilde{\partial_1(\tau)}\widetilde{\partial_2(\tau)}\widetilde{\partial_3(\tau)}\in F(U)$, so in particular 
$\phi\left(\widetilde{\partial}(\tau)\right)=\partial(\tau)$. 
Since, as discussed earlier, $\partial(\tau)$ can be seen as a~cyclically reduced word of length 3 over $W_2(S)$, so
$\widetilde{\partial}(\tau)$ is a~cyclically reduced word of length 3 over $U\cup U^{-1}$.
Denote $\widetilde{\partial}\left(\mathcal{T}\right) =
\left\{\widetilde{\partial}(\tau):\tau\in\mathcal{T}\right\}$ and let $\widetilde{\Gamma}=\lle U|\widetilde
{\partial}(\mathcal{T})\rre$.
For every $\tau\in\mathcal{T}$, the~word $\partial(\tau)$ represents the 
trivial element of $\Gamma$, so the~map $\phi$ defines a homomorphism of quotients $\widetilde{\phi} : 
\widetilde{\Gamma} \longrightarrow \Gamma$, which has the~image of index 1 or 2 in $\Gamma$.
Hence, to show that $\Gamma$ has property $(\mathrm{F}_X)$ for some uniformly curved Banach space $X$, it suffices
to show that $\widetilde{\Gamma}$ does. The same is true for property $(\mathrm{T})$.

We can now see that Theorem~\ref{crit_kazhdan} is a~direct consequence of \cite[Proposition 6]{zuk03}.
\begin{proof}[Proof of Theorem~\ref{crit_kazhdan}]
In the~notation of \cite[Section 7.1]{zuk03}, choose distinct elements $s_i$, so that $\{s_1, s_2, \ldots, s_k\}=U$, and let
$(R_1, R_2, \ldots, R_n)$ be such a~sequence of elements of $\widetilde{\partial}\left(\mathcal{T}\right)$, in which every $t\in \widetilde{\partial}\left(\mathcal{T}\right)$ appears 
the~number of times equal to $|\{\tau\in\mathcal{T}:\widetilde{\partial}(\tau)=t\}|$, so in particular $n=|\mathcal{T}|$. Let there be no relators of form $R_i'$, so
$\Gamma$ in \cite[Section 7.1]{zuk03} is our $\widetilde{\Gamma}=\lle U|\widetilde{\partial}(\mathcal{T})\rre$. Let us, in the~definition of $L'(S)$, replace the phrase 
\begin{center}
\emph{For every $R\in\{R_1,\ldots,R_n\}$, say $R=s_xs_ys_z$, we add to the graph the edges\ldots}
\end{center}
by 
\begin{center}
\emph{For every $1\leq i\leq n$, write $R_i=s_xs_ys_z$ and add to the graph the edges\ldots}
\end{center} 
Then $L'(S)$ is isomorphic to the~triagram link graph $L(S,\mathcal{T})$. 
We also note that, under this change, \cite[Proposition 6]{zuk03} remains valid with the~same proof.
Hence, if $L=L(S,\mathcal{T})$ is connected and its normalized Laplacian $\Delta$
has the~smallest non-zero eigenvalue larger than~$\frac{1}{2}$, then $\widetilde{\Gamma}$ has property~$(\textrm{T})$. By definition, $\Delta=I-A$, where $I, A\in L^2(W_2(S), d_L)$, $I$ is the identity operator and $A$ 
is the~Markov operator of $L$. We also recall that a~one-sided $\lambda$-spectral expander for $\lambda<1$ 
is necessarily connected. Thus, in other words, if $L$ is a~one-sided $\eps$-spectral expander for some 
$\eps<\frac{1}{2}$, then $\widetilde{\Gamma}$ has property~$(\textrm{T})$.
\end{proof}

To prove Theorem~\ref{crit_unif_curved}, we need one more construction. 
Let $c:\widetilde{\partial}(\mathcal{T})\longrightarrow \mathbb{N}$ be any~function (which we will specify later),
and let us consider the~following 2-dimensional combinatorial complex~$M$, which is the Cayley complex of the presentation $\widetilde{\Gamma}=\lle U|\widetilde{\partial}(\mathcal{T})\rre$, with some of its 2-faces included multiple times.

To construct $M$, first declare its set of vertices to be 
$M_0=\widetilde{\Gamma}$. Next, for every pair
$(m,u) \in M_0 \times U$, introduce an undirected edge $e$ with endpoints $m$ and $mu$. This edge has two 
orientations: by $e_+$ (resp.\ $e_-$) we denote $e$, when oriented from $m$ to $mu$ (resp.\ from $mu$ 
to~$m$) and label $e_+$ (resp.\ $e_-$) by $u$ (resp.\ $u^{-1}$).
Let $M_1$ be the set of all labelled edges of form $e_+$ or $e_-$, obtained this way.
Note that they are allowed to be loops, since elements $u\in U$ may represent the~trivial element of $\widetilde{\Gamma}$.
For $e\in M_1$, denote by $e^{-1}$ the~edge $e$ with inverted orientation.
Finally, to construct the~2-faces of $M$, consider every pair $(m,t)\in M_0\times \widetilde{\partial}(\mathcal{T})$ and write $t=u_1u_2u_3$, where $u_i\in U\cup U^{-1}$ for $1\leq i\leq 3$. Let $(e_1, e_2, e_3)\in M_1^3$ be the~cycle defined by the~conditions: $e_1$ begins at $m$ and $e_i$ is labelled by $u_i$ for $1\leq i\leq 3$. With this notation in mind, attach $c(t)$ new distinct triangles to $M$, so that the boundary of~each of~them is attached along the cycle $(e_1, e_2, e_3)$. Denote by $M_2$ the set of all triangles attached in this step (for any pair $(m,t)$). For every $f\in M_2$, choose $(e_1(f), e_2(f), e_3(f))\in M_1^3$ to be any fixed boundary cycle of $f$ and denote $e_4(f)=e_1(f)$.

Given $m\in M_0$, we define the~\emph{link of $m$ in $M$} as the~undirected graph $L(m)$ with the~vertex set $V(m)=\{e\in E_1 : e\text{ begins at }m\}$ and the edge set satisfying the~condition, that the~number of edges between any $p, q\in V(m)$ equals the number of pairs $(f,i) \in M_2 \times \{1, 2, 3\}$ such that $(e_i(f), e_{i+1}(f))=(p^{-1}, q)$ or $(e_i(f), e_{i+1}(f))=(q^{-1}, p)$. Because of symmetries of $M$,
all graphs $L(m)$, for $m\in M_0$, are isomorphic to each other.

We are now in the~position to prove Theorem~\ref{crit_unif_curved}.

\begin{proof}[Proof of Theorem~\ref{crit_unif_curved}]

Let $X$ be a~uniformly curved Banach space. The group $\widetilde{\Gamma}$ acts freely on the~complex $M$, 
which is connected and locally finite, but may not be simplicial. We claim, however, that \cite[Theorem A with Remark 5.2(i)]{dLdlS21} remain valid for this action, with our definition of links in mind.
To verify that this is the~case, it suffices to note two facts about the~proofs of these results, given in \cite{dLdlS21}:
\begin{itemize}
	\item the only place, where the~properties of the action, the complex, or 
	the definition of links play any role is the use of \cite[Theorem 4.1]{dLdlS21},
	\item the~proof of \cite[Theorem 4.1]{dLdlS21} is correct also for the action $\widetilde{\Gamma}\curvearrowright M$.
\end{itemize}

Hence, there exists a~constant $\eps(X)>0$, depending only on $X$ (and $\eps(X)=2p^{-\frac{p}{2}}2^{-\frac{p^2}{2}}$ if $X$ is an $L^p$--space for $p\geq 2$), such that the~following holds: if every link $L(m)$, for $m\in M_0$, is a~two-sided $\eps$-spectral expander for some $\eps<\eps(X)$, then $\widetilde{\Gamma}$ has property $(\mathrm{F}_X)$.
Here we rephrased the~condition
$\|A_L\|_{B(L_0^2(L,\nu))}<\eps(X)$ by using the~fact that the~measure $\nu$ is defined by the~rescaled 
degree function $d_L$ and so $\|A_L\|_{B(L_0^2(L,\nu))}=\max\left(|\lambda_2(A_L)|, |\lambda_N(A_L)|\right)$.

Now let us define $c:\widetilde{\partial}(\mathcal{T}) \longrightarrow \mathbb{N}$ by
setting $c(t)=|\{\tau\in\mathcal{T}:\widetilde{\partial}(\tau)=t\}|$ for $t\in\widetilde{\partial}\left(\mathcal{T}\right)$. To finish the proof of Theorem~\ref{crit_unif_curved}, it suffices to note
that, for this $c$, every link $L(m)$, for $m\in M_0$, is isomorphic to
the~triagram link graph $L(S,\mathcal{T})$. To see that this is the case, consider any $m\in M_0$.
Labels on edges define a~bijection between $V(m)$ and $U\cup U^{-1}=W_2(S)$.
By using it and comparing the~definition of the link $L(m)$ to Definition~\ref{def_triagram_link_graph}, 
it is straightforward to verify that $L(m)$ is indeed isomorphic to $L(S,\mathcal{T})$.

\end{proof}

\section{Fixed point properties for random square groups}\label{fix_sec}

In this section we use the spectral criterion given by Theorem~\ref{crit_unif_curved} to prove property $(\textrm{T})$ 
and properties of form $(\textrm{F}_X)$ for random groups in the square model $\mathcal{Q}(n,d)$ for $d>\frac{5}{12}$ (Theorems \ref{main_thm_T}, \ref{main_thm_FLp} and \ref{main_thm_F_X}).

\subsection{Passing to binomial square model}
\label{binomial_subsec}
As the~first step, we will verify that our main results will follow once we establish their analogues for random groups in the~binomial square model $\mathcal{Q}(n,\rho)$, which differs from $\mathcal{Q}(n,d)$ in the 
way the~random set of relators is obtained.

\begin{defin}
Fix $\rho : \mathbb{N} \longrightarrow [0,1]$, $n\in\mathbb{N}$ and a~set $S$ of $n$ (free) generators.
A random group~$\Gamma$ in the~\emph{binomial square model} $\mathcal{Q}(n,\rho)$ is given by (and considered with) the~presentation $\Gamma=\lle S | R \rre$, where $R\subseteq W_4(S)$ is a~random set of relators, such that every element of $W_4(S)$ is included in $R$ independently with probability $\rho(n)$.
If $\mathcal{P}$ is a~property of groups or presentations, then we say that a~random group
in the model $\mathcal{Q}(n,\rho)$ satisfies $\mathcal{P}$ \emph{asymptotically almost surely (a.a.s.)} if the~probability that a~random group $\Gamma=\lle S|R\rre$ in $\mathcal{Q}(n,\rho)$ satisfies $\mathcal{P}$ tends to 1 as $n$ tends to infinity.
\end{defin}

When $d\in (0,1)$ and $\rho:\mathbb{N}\rightarrow [0,1]$ are chosen so that $\rho(n)|W_4|=\lf (2n-1)^{4d}\rf$, then the random sets $R$ used in the construction of $\mathcal{Q}(n,d)$ and $\mathcal{Q}(n,\rho)$ have the same expected size and we may expect asymptotic properties of $\mathcal{Q}(n,d)$ and $\mathcal{Q}(n,\rho)$ to be similar. This is true in~our case: Theorems~\ref{main_thm_T}, \ref{main_thm_FLp} and \ref{main_thm_F_X} will follow, by a standard result about random subsets, from the~following theorem.

\begin{thm}\label{binomial_model_main_thm}
Suppose $d\in \left(\frac{5}{12}, \frac{1}{2}\right]$ and $\rho:\mathbb{N}\rightarrow [0,1]$ satisfies $\rho\sim (2n)^{4(d-1)}$. For every uniformly curved Banach space~$X$, a~random group in the~model $\mathcal{Q}(n,\rho)$ a.a.s.\ has property~$(\mathrm{F}_X)$. Additionally, there exists a~constant $f_d>0$, depending only on $d$, such that a~random group in the~model $\mathcal{Q}(n,\rho)$ a.a.s.\ has property $(\mathrm{F}L^p)$ for every
$p\in\left[1, f_d \left(\log{n}\right)^{\frac{1}{2}}\right]$.
\end{thm}

\begin{proof}[Proof of Theorems~\ref{main_thm_T}, \ref{main_thm_FLp} and \ref{main_thm_F_X}, by using Theorem~\ref{binomial_model_main_thm}]
Let us just prove Theorem~\ref{main_thm_FLp}, as Theorem~\ref{main_thm_F_X} is obtained fully analogously and
Theorem~\ref{main_thm_T} follows from Theorem~\ref{main_thm_FLp}.

If $d>\frac{1}{2}$, then a~random group in the~model $\mathcal{Q}(n,d)$ is a.a.s.\ finite and the~theorem is clear. Suppose now that $d\in\left(\frac{5}{12}, \frac{1}{2}\right]$ and let $f_d$ be as in the~Theorem~\ref{binomial_model_main_thm}.
For $n\geq 1$, let $S_n$ be a~set of size $n$ and let
\begin{equation*}
\mathcal{P}_n = \left\{ R\subseteq W_4(S_n) : \lle S_n|R \rre\text{ has property }(\mathrm{F}L^p)\text{ for every }p\in \left[1, f_d\left(\log{n}\right)^{\frac{1}{2}}\right]\right\}.
\end{equation*}
A~quotient of~a~group with property $(\mathrm{F}L^p)$ is also a~group with property $(\mathrm{F}L^p)$, so 
the~family $\mathcal{P}_n$ is increasing, i.e.\ if $R_1\in \mathcal{P}_n$ and $R_1\subseteq R_2\subseteq W_4(S_n)$, then also $R_2\in \mathcal{P}_n$.
Define $\rho : \mathbb{N}\rightarrow [0,1]$ by $\rho(n)=\frac{\left\lfloor (2n-1)^{4d}\right\rfloor}{|W_4(S_n)|}$.
Since $|W_4(S_n)|\sim (2n)^4$, so $\rho\sim (2n)^{4(d-1)}$.
Denote by $R$ (resp.\ ~$R'$) a~random set of relators over $S=S_n$, appearing in the~definition of~the~model $\mathcal{Q}(n,\rho)$ (resp.\ ~$\mathcal{Q}(n,d)$). 
Theorem~\ref{binomial_model_main_thm} now says that $\mathbb{P}\left(R \in \mathcal{P}_n\right)\rightarrow 1$ as $n\rightarrow \infty$. What we want to 
show is equivalent to $\mathbb{P}\left(R'\in \mathcal{P}_n\right)\rightarrow 1$ as $n\rightarrow \infty$. This follows directly from 
\cite[Corollary 1.16(i)]{jan00} applied for $\Gamma=W_4$, $M=\left\lfloor (2n-1)^{4d}\right\rfloor$ and $\mathcal{Q}=\mathcal{P}_n$, since
then, in the notation of the~Corollary, $N=|W_4(S_n)|$ and $R$ (resp.\ $R'$) is $\Gamma_{M/N}$ (resp.\ $\Gamma_M$) up to distribution.
\end{proof}

From now on we focus on random groups in the binomial square model and our goal is to prove Theorem~\ref{binomial_model_main_thm}.
For convenience, we are going to use the~following equivalent construction of~the~random set of relators $R$ in the model $\mathcal{Q}\left(n,\rho\right)$.

\begin{rem}\label{three_round_relators_rem}
Given $\rho : \mathbb{N} \rightarrow [0,1]$, define $\rho_3 : \mathbb{N}\rightarrow [0,1]$ by $\rho_3(n)=1-(1-\rho(n))^{\frac{1}{3}}$ for all 
$n\in\mathbb{N}$. 
If $\rho(n)=o(1)$, then $\rho_3(n)\sim \frac{1}{3}\rho(n)$, by the mean value theorem.
Suppose $S$ is a~set of size $n$.
Let $R$ be a~random subset of~$W_4(S)$, such that every $w\in W_4(S)$ 
is included in~$R$ independently with probability $\rho(n)$. Let $R_1, R_2, R_3$ be independent subsets of $W_4(S)$ such that, for every $1\leq i\leq 3$, every $w\in W_4(S)$ is included in $R_i$ independently with probability $\rho_3(n)$. 
Then, for every $w\in W_4(S)$, 
\begin{equation*}
	\mathbb{P}\left(w\notin R_1\cup R_2\cup R_3\right) = (1-\rho_3(n))^3 = 1-\rho(n) = \mathbb{P}\left(w\notin R\right).
\end{equation*}
Hence, the sets $R$ and $R_1 \cup R_2 \cup R_3$ have the~same 
distribution as random subsets of $W_4(S)$ and so the random presentation $\lle S | R_1 \cup R_2 \cup R_3\rre$ has the 
same distribution as the random presentation obtained in the binomial model $\mathcal{Q}(n, \rho)$. 
\end{rem}

\subsection{Random triagram link graph as an expander}
We aim to show Theorem~\ref{binomial_model_main_thm} by using the spectral criterion given by Theorem~\ref{crit_unif_curved}.
To that end, we will show that random groups in the~binomial square model (for appropriate parameters) a.a.s.\ admit triagram link graphs 
that are two-sided $\eps$-spectral expanders for $\eps\rightarrow 0$. Precisely, we will prove the following.

\begin{prop}\label{full_expander_prop}
For every $d\in\left(\frac{5}{12}, \frac{1}{2}\right]$, there exists a~constant $e_d>0$ with the~following property.
Suppose that $\rho:\mathbb{N}\longrightarrow[0,1]$ satisfies $\rho \sim (2n)^{4(d-1)}$ and let $\Gamma = \lle S|R\rre$ be a random group in the model
$\mathcal{Q}(n,\rho)$. A.a.s.\ there exists a~set of triagrams $\mathcal{T}\subseteq R^3$, such that the triagram link graph $L(S, \mathcal{T})$ is a~two-sided $\eps$-spectral expander for $\eps=n^{-e_d}$.
\end{prop}

\begin{rem}
We will prove this for $e_d = \frac{\alpha^2}{32}$, where $\alpha = 12d-5$. We claim no optimality of this value.
\end{rem}

Let us see that Proposition~\ref{full_expander_prop} indeed suffices to show Theorem~\ref{binomial_model_main_thm}.

\begin{proof}[Proof of Theorem~\ref{binomial_model_main_thm}, using Proposition~\ref{full_expander_prop}]

Fix $d\in \left(\frac{5}{12},\frac{1}{2}\right]$ and let $e_d$ be a~number satisfying Proposition~\ref{full_expander_prop}.
For the first part, let $X$ be a~uniformly curved Banach space and let $\eps(X)$ be a~constant satisfying Theorem~\ref{crit_unif_curved}. 
For sufficiently large $n$, we have $n^{-e_d}<\eps(X)$, so $\mathcal{Q}(n,\rho)$ has property $(\mathrm{F}_X)$ a.a.s.

For the second part, let $X$ be an $L^p$--space. As noted in~the~introduction, properties of form $(\mathrm{F}L^p)$,
for $p\in[1,2]$, are equivalent to each other, so it suffices to consider only $p\geq 2$.

Theorem~\ref{crit_unif_curved} holds for $\eps(X)=2p^
{-\frac{p}{2}}2^{-\frac{p^2}{2}}$, so a.a.s.\ a~random group in the~model $\mathcal{Q}(n, \rho)$ has property $(\mathrm{F}L^p)$ for 
every $p\geq 2$, such that $n^{-e_d} < 2p^{-\frac{p}{2}}2^{-\frac{p^2}{2}}$. This inequality is equivalent to
\begin{equation*}
-e_d \log{n} < \log{2} -\frac{p}{2}\log{p}-\frac{p^2}{2}\log{2}.
\end{equation*}
For some universal constant $C>0$, this holds whenever $e_d \log{n} \geq Cp^2$.
Hence it suffices to choose
$f_d = \left(\frac{e_d}{C}\right)^{\frac{1}{2}}$.
	 
\end{proof}

We now turn to proving Proposition~\ref{full_expander_prop}. Hereafter, until the end of Section~\ref{asymptotic_regularity_sec}, we make
the following assumptions. Let us fix a~density $d\in\left(\frac{5}{12}, \frac{1}{2}\right]$ and $\rho : \mathbb{N} \longrightarrow [0,1]$, satisfying 
$\rho\sim (2n)^{4(d-1)}$. Let $\Gamma = \lle S|R\rre$ be a~random group in the model $\mathcal{Q}(n,\rho)$.
Define $\alpha\in(0,1]$ by $\alpha=12d-5$, so that $d=\frac{5+\alpha}{12}$.
For given $\rho$, let the~function $p=\rho_3$ and random sets $R_1, R_2, R_3\subseteq W_4(S)$ be defined as in Remark~\ref{three_round_relators_rem}, so that $p\sim \frac{1}{3}(2n)^{4(d-1)} \sim \frac{1}{3}(2n)^{\frac{\alpha}{3}-\frac{7}{3}}$
and w.l.o.g.\ we can assume that $R=R_1\cup R_2\cup R_3$.
Implicitly, $R_1, R_2, R_3$ are all defined on a~single probability space ---
we denote by $\mathbb{P}$ (resp.\ $\mathbb{E}$) the~probability function (resp.\ expected value operator)
corresponding to that space.

Recall that by $T(S)$ we denote the~set of all triagrams over $S$. We set
\begin{equation*}
\mathcal{T} = \left(R_1 \times R_2 \times R_3\right) \cap T(S)
\end{equation*}
and we are going to prove Proposition~\ref{full_expander_prop} for this choice of $\mathcal{T}$. Independence of 
$R_1, R_2$ and $R_3$ means that $\mathbb{P}\left(\tau\in\mathcal{T}\right)=p^3$ for every $\tau \in T(S)$. 
Finally, let $L=L(S,\mathcal{T})$ and $L_i=L_i(S, \mathcal{T})$, for $1\leq i\leq 3,$ be the triagram link graph and its parts, given by Definition~\ref{def_triagram_link_graph}. 
Note that each of $L_1, L_2, L_3$ has the same distribution
as a~random graph, because the sets $R_1, R_2, R_3$ are identically  independently distributed.

We aim to prove that $L$ is a.a.s.\ a~two-sided
$\eps$-spectral expander for $\eps=n^{-e_d}$, with $e_d>0$
depending only on $d$. 
As the next step, we show that it will follow once
we establish that each~$L_i$, for $1\leq i\leq3$, satisfies the same kind of property as we want for $L$ (Lemma~\ref{part_expander_lem}), together with a~certain asymptotic regularity condition (Lemma~\ref{regularity_lem}). The precise
formulations of the lemmas we use are as follows.
Recall that each of the graphs $L, L_1, L_2, L_3$ has $W_2(S)$ as its vertex set.
\begin{lem}\label{part_expander_lem}
Let $\eps=n^{-\frac{\alpha^2}{31}}$.
For $1\leq i\leq 3$, the graph $L_i$ is a.a.s.\ a~two-sided $\eps$-spectral expander.
\end{lem}

\begin{lem}\label{regularity_lem}
Let $d_0=2(2n)^7p^3$, $\delta=n^{-\frac{\alpha}{3}}$, and $1\leq i\leq 3$. Denote by $\mathcal{R}_{i, \delta}$ the~event that
\begin{equation*}
|d_{L_i}(w)-d_0|\leq \delta d_0\text{ for every }w\in W_2(S).
\end{equation*}

Then, for every $\beta>0$,
\begin{equation*}
\mathbb{P}(\mathcal{R}_{i, \delta})=1-o_{\beta, \rho}\left(n^{-\beta}\right).
\end{equation*}
\end{lem}

We prove Lemma~\ref{part_expander_lem} in 
Subsection~\ref{part_expander_proof_subsec} (relying on Lemma~\ref{regularity_lem}) and we postpone the proof of~Lemma~\ref{regularity_lem} until Section~\ref{asymptotic_regularity_sec}.
Assuming these results, Proposition~\ref{full_expander_prop}
can be proved as follows.

\begin{proof}[Proof of Proposition~\ref{full_expander_prop}, using Lemmas~\ref{part_expander_lem} and~\ref{regularity_lem}]
A.a.s.\ each of the graphs $L_1$, $L_2$, $L_3$ is a~two-sided 
$\eps$-spectral expander for $\eps=n^{-e_d'}$, where $e_d' = \frac{\alpha^2}{31}$,
and all the events $\mathcal{R}_{1,\delta}$, $\mathcal{R}_{2,\delta}$, $\mathcal{R}_{3,\delta}$ occur for
 $\delta=n^{-\frac{\alpha}{3}}$.
Since $e_d' < \frac{2}{3}\alpha$, we have $\delta^2<\eps$.
By~Lemma~\ref{regularity_lem}, we have
$d_{L_i}(w)=d_0\left(1+O_\text{abs}\left(\delta\right)\right)$ for $1\leq i\leq 3$ and $w\in W_2(S)$. Hence 
$d_L(w)=d_{L_1}(w)+d_{L_2}(w)+d_{L_3}(w)=3d_0\left(1+O_\text{abs}\left(\delta\right)\right)$. Aiming to apply Remark~\ref{expand_min_max_rem}, suppose that
$f\in L^2_0(W_2(S), d_L)$ is any function. Consider any $1\leq i\leq 3$.
Let $D_i=\sum_{w\in W_2(S)} d_{L_i}(w)$ and $u_i=D_i^{-1/2}\pmb{1}$,
so that $\|u_i\|_{L_i}=1$. Orthogonally decompose $f=c_iu_i+g_i$, where $c_i\in \mathbb{C}$ and 
$g_i\in L^2_0(W_2(S), d_{L_i})$. Define $h_i\in L^2(W_2(S),d_{L_i})$ 
by $h_i(w) = 1-\frac{d_L(w)}{3d_{L_i}(w)}$ for $w\in W_2(S)$.
We have 
\begin{equation*}
\begin{split}
c_i &= \lle f, u_i\rre_{L_i}=D_i^{-1/2}\sum_{w\in W_2(S)} f(w)d_{L_i}(w)\\
&=D_i^{-1/2}\sum_{w\in W_2(S)}f(w)\left(d_{L_i}(w)-\frac{1}{3}d_L(w)\right)
+\frac{1}{3}D_i^{-1/2}\lle f, \pmb{1}\rre_L\\
&=D_i^{-1/2}\sum_{w\in W_2(S)}f(w)\overline{h_i(w)}d_{L_i}(w)
=D_i^{-1/2}\lle f, h_i\rre_{L_i},
\end{split}
\end{equation*}
so that, by the Cauchy-Schwarz inequality, $|c_i|\leq D_i^{-1/2} \|f\|_{L_i} \|h_i\|_{L_i}$. 
Note that $h_i(w)=O_\alpha\left(\delta\right)$, so $\|h_i\|_{L_i} = O_\alpha\left(\delta D_i^{1/2}\right)$
and $|c_i|=O_\alpha\left(\delta\|f\|_{L_i}\right)$. By orthogonality, $\|f\|^2_{L_i}=|c_i|^2 + \|g_i\|_{L_i}^2$, hence $\|g_i\|_{L_i}\leq\|f\|_{L_i}$.

$A_{L_i}$ is self-adjoint and has $\pmb{1}$ as an eigenvector, so $L_0^2(W_2(S), d_{L_i})$ is its invariant
subspace and

\begin{equation*}
\lle A_{L_i}f, f \rre_{L_i} = \lle A_{L_i}(c_iu_i), c_iu_i \rre_{L_i} + \lle A_{L_i}g_i, g_i \rre_{L_i}
=|c_i|^2 + \lle A_{L_i}g_i, g_i \rre_{L_i}.
\end{equation*}
Since $L_i$ is a two-sided $\eps$-spectral expander, by Remark~\ref{expand_min_max_rem} we have $|\lle A_{L_i}g_i, g_i \rre_{L_i}| \leq \eps \|g_i\|_{L_i}^2$. 
Remembering that $\delta^2<\eps$, in summary we obtain
\begin{equation*}
|\lle A_{L_i}f, f \rre_{L_i}| \leq O_\alpha\left(\delta^2\|f\|^2_{L_i}\right)+\eps \|g_i\|_{L_i}^2
= O_\alpha\left(\eps\|f\|_{L_i}^2\right).
\end{equation*}
Combining this, for $1\leq i\leq 3$, yields
\begin{equation*}
\begin{split}
|\lle A_Lf, f\rre_L| &= \left|\sum_{w_1, w_2 \in W_2(S)} \omega_L(w_1,w_2)f(w_1)\overline{f(w_2)}\right|=\left|\sum_{i=1}^3\sum_{w_1, w_2 \in W_2(S)} \omega_{L_i}(w_1,w_2)f(w_1)\overline{f(w_2)}\right|\\
&=\left|\sum_{i=1}^3\lle A_{L_i}f, f\rre_{L_i}\right|=O_\alpha\left(\eps\sum_{i=1}^3\|f\|^2_{L_i}\right)
=O_\alpha\left(\eps\|f\|_L^2\right).
\end{split}
\end{equation*}
Now let $e_d=\frac{\alpha^2}{32}$, so that $e_d < e_d'$ and $\eps=n^{-e'_d}=o_\alpha\left(n^{-e_d}\right)$. This means that a.a.s.\ we have $|\lle A_Lf, f\rre_L| \leq n^{-e_d}\|f\|_L^2$
for every $f\in L_0^2(W_2(S), d_L)$. By Remark~\ref{expand_min_max_rem}, $L$ is hence a.a.s.\ a~two-sided $n^{-e_d}$-spectral expander.
\end{proof}

\subsection{The graphs $L_i$ as expanders}\label{part_expander_proof_subsec}
In this subsection we prove Lemma~\ref{part_expander_lem}, while assuming regularity given by
Lemma~\ref{regularity_lem}. W.l.o.g.\ let $i=1$, so our aim is to show that a.a.s.\ $L_1$
is a~two-sided $\eps$-spectral expander for $\eps=n^{-e_d'}$,
where $e'_d=\frac{\alpha^2}{31}$.
Let $\delta=n^{-\frac{\alpha}{3}}$ and $\mathcal{R} = \mathcal{R}_{1,\delta}$ be as in Lemma~\ref{regularity_lem}. 

Define the~random operator $A$
on $L^2(W_2(S), d_{L_1})$ as follows. If $\mathcal{R}$ holds, then $L_1$ does not have isolated vertices
and so let ${A=A_{L_1}}$ be the Markov operator of $L_1$ and if $\mathcal{R}$ does not hold, then let $A=I$ be the identity operator (as a~convenient placeholder). In both cases, $A$ is self-adjoint and has $N$ real eigenvalues
$\lambda_1(A)\geq\lambda_2(A)\geq\ldots\geq\lambda_N(A)$, where $N=|W_2(S)|$.
Denote
\begin{equation*}
r(A) = \max_{i=2,\ldots, N} |\lambda_i(A)| = \max\left(|\lambda_2(A)|, |\lambda_N(A)|\right).
\end{equation*}
If $\mathcal{R}$ holds, then $L_1$ is a~two-sided $\eps$-spectral expander if and only if $r(A)\leq \eps$. Following \cite{bro87}, we are going to bound $r(A)$ by using the~trace of a~power of $A$.
Let $k\geq 1$ be any natural number. As $\lambda_1(A)=1$, we have
\begin{equation*}
\Tr(A^{2k})=\sum_{i=1}^N|\lambda_i(A)|^{2k} \geq 1 + r(A)^{2k},
\end{equation*}
so that
\begin{equation}\label{trace_radius_ineq}
	r(A)^{2k}\leq \Tr(A^{2k})-1.
\end{equation}
Denote by $B_\eps$ the event that $L_1$ is not a~two-sided $\eps$-spectral expander. We have:
\begin{equation}\label{PBeps_starting_bound_ineq}
\begin{split}
    \mathbb{P}\left(B_\eps\right)
    &\leq\mathbb{P}\left(\mathcal{R}\text{ does not hold or }r(A)> \eps\right)\\
    &=\mathbb{P}\left(\mathcal{R}^c\right)+\mathbb{P}\left(r(A)\mathds{1}_\mathcal{R}>\eps\right)\\
    &=\mathbb{P}\left(\mathcal{R}^c\right)+\mathbb{P}\left(r(A)^{2k}\mathds{1}_\mathcal{R}>\eps^{2k}\right).
\end{split}
\end{equation}
By the Markov's inequality, 
$\mathbb{P}\left(r(A)^{2k}\mathds{1}_\mathcal{R}>\eps^{2k}\right)\leq 
\eps^{-2k}\mathbb{E}\left(r(A)^{2k}\mathds{1}_\mathcal{R}\right)$. Because of this 
and~(\ref{trace_radius_ineq}),
inequality (\ref{PBeps_starting_bound_ineq}) leads to:
\begin{equation*}
\begin{split}
    \mathbb{P}\left(B_\eps\right)
    &\leq\mathbb{P}\left(\mathcal{R}^c\right)+\eps^{-2k}\mathbb{E}\left(r(A)^{2k}\mathds{1}_\mathcal{R}\right)\\
    &\leq\mathbb{P}\left(\mathcal{R}^c\right)+\eps^{-2k}\mathbb{E}\big(\left(\Tr(A^{2k})-1\right)\mathds{1}_\mathcal{R}\big).
\end{split}
\end{equation*}
By rewriting $\mathbb{E}\left(\mathds{1}_\mathcal{R}\right)=\mathbb{P}\left(\mathcal{R}\right)=1-\mathbb{P}\left(\mathcal{R}^c\right)$ and bounding $1\leq \eps^{-2k}$, we get:
\begin{equation}\label{PBeps_final_bound_ineq}
\begin{split}
    \mathbb{P}\left(B_\eps\right)
    &\leq(1+\eps^{-2k})\mathbb{P}\left(\mathcal{R}^c\right)+\eps^{-2k}\mathbb{E}\left(\Tr(A^{2k})\mathds{1}_\mathcal{R}\right)-\eps^{-2k}\\
	&\leq 2\eps^{-2k}\mathbb{P}\left(\mathcal{R}^c\right)
	+\eps^{-2k}\mathbb{E}\left(\Tr(A^{2k})\mathds{1}_\mathcal{R}\right)-\eps^{-2k}.
\end{split}
\end{equation}

Let us see that now, to prove Lemma~\ref{part_expander_lem},
it is sufficient to show the following bound on $\mathbb{E}\left(\Tr(A^{2k})\mathds{1}_\mathcal{R}\right)$.

\begin{lem}\label{trace_of_power_lem}
If $k=\left\lfloor \frac{4}{\alpha}\right\rfloor+1$, then
\begin{equation*}
\mathbb{E}\left(\Tr(A^{2k})\mathds{1}_\mathcal{R}\right)\leq 1+O_\rho\left(n^{-\frac{\alpha}{3}}\right).
\end{equation*}
\end{lem}
\begin{proof}[Proof of Lemma~\ref{part_expander_lem}, by
using Lemma~\ref{regularity_lem} and Lemma~\ref{trace_of_power_lem}]

By~Lemma~\ref{regularity_lem}, $\mathbb{P}\left(\mathcal{R}^c\right)=o_\rho(n^{-1})$, so in particular also
$\mathbb{P}\left(\mathcal{R}^c\right)=O_\rho\left(n^{-\frac{\alpha}{3}}\right)$, since $\alpha\in(0,1]$.
From~(\ref{PBeps_final_bound_ineq}) and Lemma~\ref{trace_of_power_lem}, we get:

\begin{equation*}
\begin{split}
    \mathbb{P}\left(B_\eps\right)
    &\leq 2\eps^{-2k}\mathbb{P}\left(\mathcal{R}^c\right)
	+\eps^{-2k}\mathbb{E}\left(\Tr(A^{2k})\mathds{1}_\mathcal{R}\right)-\eps^{-2k}\\
    &\leq2\eps^{-2k}O_\rho\left(n^{-\frac{\alpha}{3}}\right)
	+\eps^{-2k}\left(1
	+O_\rho\left(n^{-\frac{\alpha}{3}}\right)\right)-\eps^{-2k}\\
    &\leq \eps^{-2k}O_\rho\left(n^{-\frac{\alpha}{3}}\right)
    =O_\rho\left(n^{\frac{2k\alpha^2}{31}-\frac{\alpha}{3}}\right).
\end{split}
\end{equation*}
As $k\leq \frac{5}{\alpha}$, so:
\begin{equation*}
\frac{2k\alpha^2}{31}-\frac{\alpha}{3} \leq \frac{10\alpha}{31}-\frac{\alpha}{3}<0.
\end{equation*}
Hence $\mathbb{P}\left(B_\eps\right)\rightarrow 0$ as $n\rightarrow \infty$.
\end{proof}

\subsection{Bounding the expected value of the trace}
In this section we prove Lemma~\ref{trace_of_power_lem}.
We fix $k=\left\lfloor \frac{4}{\alpha}\right\rfloor +1$.

Assume that $\mathcal{R}$ holds. Denote by $(e_u)_{u\in W_2(S)}$ 
the standard basis of $L^2(W_2(S), d_{L_1})$, given by $e_{u}(v) 
= \delta_{uv}$ for $u,v\in W_2(S)$. Operator $A$ in this basis 
is represented by the matrix $\left(a_{uv}\right)_{u, v\in 
W_2(S)}$, where $a_{uv}=\frac{\omega_{L_1}(u,v)}{d_{L_1}(u)}$. 
We have $d_{L_1}(u) \geq (1-\delta)d_0$, so ${a_{uv}\leq(1-\delta)^{-1}d_0^{-1}
\omega_{L_1}(u,v)}$ for $u,v\in W_2(S)$. Hence:
\begin{equation}\label{trace_method_eq}
\begin{split}
\Tr\left(A^{2k}\right)&=\sum_{u_1, u_2,\ldots, u_{2k} \in W_2(S)} a_{u_1u_2}a_{u_2u_3}\cdots a_{u_{2k}u_1}\\
&\leq (1-\delta)^{-2k}d_0^{-2k}\sum_{u_1, u_2,\ldots, u_{2k} \in W_2(S)} \omega_{L_1}(u_1, u_2)\omega_{L_1}(u_2, u_3)\cdots\omega_{L_1}(u_{2k},u_1).
\end{split}
\end{equation}
Consider any $\pmb{u}=(u_1,u_2,\ldots,u_{2k})\in W_2(S)^{2k}$ and denote $u_{2k+1}=u_1$. Denote by $E(L_1)$ the~set of~edges of $L_1$. The product
\begin{equation*}
	m=\omega_{L_1}(u_1, u_2)\omega_{L_1}(u_2, u_3)\cdots\omega_{L_1}(u_{2k},u_1)
\end{equation*}
is the number of cycles $(e_1,e_2,\ldots,e_{2k})\in E(L_1)^{2k}$, such that $e_i$ is an edge between $u_i$ and $u_{i+1}$, for $1\leq i \leq 2k$. Equivalently, by the construction of~$L_1$ (Definition~\ref{def_triagram_link_graph}), $m$ is the number of sequences of triagrams $\pmb{\tau}=(\tau_1,\tau_2,\ldots,\tau_{2k})\in\mathcal{T}^{2k}$, such that
\begin{equation}\label{compatibility_eq}
\left\{\partial_2(\tau_i)^{-1}, \partial_3(\tau_i)\right\}=\{u_i, u_{i+1}\}\quad\text{for }1\leq i\leq 2k.
\end{equation} 
If (\ref{compatibility_eq}) holds for given $\pmb{u}\in W_2(S)^{2k}$ and $\pmb{\tau}\in T(S)^{2k}$, then we say that
$\pmb{u}$ and $\pmb{\tau}$ are \emph{compatible}.
Denote $\Xi_{2k}=\left\{(\pmb{u},\pmb{\tau})\in W_2(S)^{2k}\times T(S)^{2k} : \pmb{u}\text{ and }\pmb{\tau}\text{ are compatible}\right\}$.
Call an element of $\Xi_{2k}$ a \emph{triagram cycle}.
It is not hard to see that for every $\pmb{\tau}\in T(S)^{2k}$ there exist at most 2 sequences $\pmb{u}\in W_2(S)^{2k}$ such that $(\pmb{u},\pmb{\tau})\in\Xi_{2k}$, but we will not need this observation.

We can now rewrite the inequality (\ref{trace_method_eq}) as
\begin{equation*}
\begin{split}
\Tr(A^{2k}) &\leq (1-\delta)^{-2k}d_0^{-2k}\sum_{\pmb{u}\in W_2(S)^{2k}}\big|\left\{\pmb{\tau}\in \mathcal{T}^{2k} : 
(\pmb{u},\pmb{\tau})\in\Xi_{2k}\right\}\big|\\
&= (1-\delta)^{-2k}d_0^{-2k}\sum_{(\pmb{u},\pmb{\tau})\in \Xi_{2k}} \mathds{1}_{\pmb{\tau} \in \mathcal{T}^{2k}}.
\end{split}
\end{equation*}

This is true, whenever $\mathcal{R}$ holds, so
\begin{equation}\label{sum_of_prob_eq}
\begin{split}
\mathbb{E}\left(\Tr(A^{2k})\mathds{1}_\mathcal{R}\right) &\leq (1-\delta)^{-2k}d_0^{-2k}\mathbb{E}\left(\sum_{(\pmb{u}, \pmb{\tau})\in \Xi_{2k}} \mathds{1}_{\pmb{\tau} \in \mathcal{T}^{2k}}\mathds{1}_\mathcal{R}\right)\\
&= (1-\delta)^{-2k}d_0^{-2k}\sum_{(\pmb{u}, \pmb{\tau})\in \Xi_{2k}} \mathbb{E}\big(\mathds{1}_{\pmb{\tau} \in \mathcal{T}^{2k}}\mathds{1}_\mathcal{R}\big)\\
&\leq (1-\delta)^{-2k}d_0^{-2k}\sum_{(\pmb{u}, \pmb{\tau})\in \Xi_{2k}} \mathbb{P}\left(\pmb{\tau}\in\mathcal{T}^{2k}\right).
\end{split}
\end{equation}
From now on, given $\xi\in\Xi_{2k}$, we fix the following notation. 
Write $\xi=(\pmb{u}, \pmb{\tau})$, where $\pmb{u}=(u_1,u_2,\ldots,u_
{2k})$ and $\pmb{\tau}=(\tau_1,\tau_2,\ldots, \tau_{2k})$, and 
denote $u_{2k+1}=u_1$. For $1\leq i\leq 2k$, write $\tau_i=(r_1^i,
r_2^i,r_3^i)=(x_1^ix_2^ix_3^i(y_1^i)^{-1}, y_1^iy_2^iy_3^i(z_1^i)^
{-1}, z_1^iz_2^iz_3^i(x_1^i)^{-1})$, where $r_l^i\in W_4(S)$ for $1\leq l\leq 3$, and $x_j^i, y_j^i, z_j^i\in S\cup S^{-1}$ for $1\leq j\leq 3$. For $1\leq l\leq 3$, let $F_\xi^l=\{r_l^i: 1\leq i\leq 2k\}$ be the set of relators appearing on the $l$-th position in the triagrams of $\pmb{\tau}$. As the random sets $R_1, R_2, R_3$ are independent, we have
\begin{equation}\label{probbysets_eq}
\mathbb{P}\left(\pmb{\tau}\in \mathcal{T}^{2k}\right)=p^{\left|F_\xi^1\right|}p^{\left|F_\xi^2\right|}p^{\left|F_\xi^3\right|}=p^{\left|F_\xi^1\right|+\left|F_\xi^2\right|+\left|F_\xi^3\right|}.
\end{equation} 

Let $|\cdot| : S\cup S^{-1} \longrightarrow S$ be the function given by $|s|=|s^{-1}|=s$ for $s\in S$. We note the following simple inequalities.

\begin{rem}\label{F_xi_ineq_rem}
Suppose $\xi\in\Xi_{2k}$. Then, the~following inequalities hold:
\begin{equation*}
\begin{split}
\left|\left\{\left|x_1^i\right| : 1\leq i\leq 2k\right\}\right|
&\leq\min\left(|F_\xi^3|, |F_\xi^1|\right),\\
\left|\left\{\left|y_1^i\right| : 1\leq i\leq 2k\right\}\right|
&\leq\min\left(|F_\xi^1|, |F_\xi^2|\right),\\
\left|\left\{\left|z_1^i\right| : 1\leq i\leq 2k\right\}\right|
&\leq\min\left(|F_\xi^2|, |F_\xi^3|\right).\\
\end{split}
\end{equation*}
Additionally, for $2\leq j\leq 3$, we have:
\begin{equation*}
\begin{split}
&\left|\left\{\left|x_j^i\right| : 1\leq i\leq 2k\right\}\right|\leq|F_\xi^1|,\\
&\left|\left\{\left|y_j^i\right| : 1\leq i\leq 2k\right\}\right|\leq|F_\xi^2|,\\
&\left|\left\{\left|z_j^i\right| : 1\leq i\leq 2k\right\}\right|\leq|F_\xi^3|.
\end{split}
\end{equation*}
\end{rem}

\begin{proof}
For every $1\leq j\leq 3$ and $1\leq i\leq 2k$, the~generator $x_j^i$ is uniquely determined by the~word $r_1^i
=x_1^ix_2^ix_3^i(y_1^i)^{-1}$, so $\left|\left\{\left|x_j^i\right| : 1\leq i\leq 2k\right\}\right|\leq|F_\xi^1|$.
At the same time, for every $1\leq i\leq 2k$, $x_1^i$ is uniquely determined by $r_3^i=z_1^iz_2^iz_3^i(x_1^i)^{-1}$,
hence for $j=1$ we have a stronger bound:
$\left|\left\{\left|x_1^i\right|:1\leq i\leq 2k\right\}\right|\leq \min\left(|F_\xi^3|, |F_\xi^1|\right)$.
The remaining inequalities can be proved analogously.
\end{proof}

We are going to group the terms of the last sum in (\ref{sum_of_prob_eq}) by using symmetries arising from permutations of the generators. Denote by $\Sigma_{2n}$ the symmetric group on the set $S\cup S^{-1}$ and let 
\begin{equation*}
\Theta=\left\{\theta\in \Sigma_{2n} : \theta\left(x^{-1}\right)=\theta(x)^{-1}\text{ for all }x\in S\cup S^{-1}\right\}
\end{equation*}
be the subgroup of permutations preserving the inverses. Every $\theta\in\Theta$ extends to a unique~automorphism of the 
free group $F(S)$, giving an~action $\Theta\curvearrowright F(S)$. This action preserves the sets $W_2(S)$ and~$W_4(S)$,
so we can define an action $\Theta\curvearrowright T(S)$, by setting $\theta(r_1, r_2, r_2)= (\theta(r_1),\theta(r_2),\theta(r_2))$ 
for $(r_1, r_2, r_3)\in T(S)$ and $\theta\in\Theta$. It is easy to see that the functions $\partial_1,\partial_2,\partial_3:T(S) \rightarrow W_2(S)$ 
are $\Theta$-equivariant. Finally, we have an~action $\Theta\curvearrowright \Xi_{2k}$, given by 
$\theta((u_1,u_2,\ldots,u_{2k}), (\tau_1, \tau_2, \ldots, \tau_{2k}))=
((\theta(u_1),\theta(u_2),\ldots,\theta(u_{2k})), (\theta(\tau_1), \theta(\tau_2), \ldots, \theta(\tau_{2k})))$
for $((u_1,u_2,\ldots,u_{2k}), (\tau_1, \tau_2, \ldots, \tau_{2k}))\in \Xi_{2k}$ and $\theta\in\Theta$. The fact that all these actions are well-defined is straightforward to verify. Given $\xi\in\Xi_{2k}$, denote by $\Theta(\xi)$ the orbit of $\xi$ under the action of $\Theta$.

If $\xi_1=(\pmb{u}_1, \pmb{\tau}_1), \xi_2=(\pmb{u}_2, \pmb{\tau}_2)\in\Xi_{2k}$ lie in the same $\Theta$-orbit, then $|F_{\xi_1}^l|=|F_{\xi_2}^l|$ for $1\leq l\leq 3$, so for any $\xi\in\Xi_{2k}$ the value of $p^{\left|F_{\xi}^1\right|+\left|F_{\xi}^2\right|+\left|F_{\xi}^3\right|}$ depends only on the $\Theta$-orbit of $\xi$.
Fix $\Xi_{2k}^0$ to be a set of representatives of the orbits of the action $\Theta\curvearrowright\Xi_{2k}$. Using (\ref{probbysets_eq}), we can rewrite 
the inequality~(\ref{sum_of_prob_eq}) as follows.

\begin{equation}\label{orbit_sum_eq}
\begin{split}
\mathbb{E}\left(\Tr(A^{2k})\mathds{1}_\mathcal{R}\right)
&\leq (1-\delta)^{-2k}d_0^{-2k}\sum_{\xi\in\Xi_{2k}} p^{\left|F_\xi^1\right|+\left|F_\xi^2\right|+\left|F_\xi^3\right|}\\
&= (1-\delta)^{-2k}d_0^{-2k}\sum_{\xi\in\Xi_{2k}^0} \left|\Theta(\xi)\right|p^{\left|F_{\xi}^1\right|+\left|F_{\xi}^2\right|+\left|F_{\xi}^3\right|}
\end{split}
\end{equation}

Let us introduce some more notation regarding a~given $\xi\in\Xi_{2k}$. Define the following sets:

\begin{equation*}
\begin{split}
A_\xi &= \left\{\left|x_2^i\right|:1\leq i\leq 2k\right\}\cup \left\{\left|x_3^i\right|:1\leq i\leq 2k\right\},\\
B_\xi &= \left\{\left|y_2^i\right|:1\leq i\leq 2k\right\}
\cup\left\{\left|z_3^i\right|:1\leq i\leq 2k\right\},\\
B'_\xi &= \left\{\left|z_1^i\right|:1\leq i\leq 2k\right\},\\
C_\xi &= \left\{\left|x_1^i\right|:1\leq i\leq 2k\right\}\cup\left\{\left|y_1^i\right|:1\leq i\leq 2k\right\},\\
D_\xi &= \left\{\left|y_3^i\right|:1\leq i\leq 2k\right\}\cup\left\{\left|z_2^i\right|:1\leq i\leq 2k\right\}.
\end{split}
\end{equation*}
Each of $A_\xi, B_\xi, B'_\xi, C_\xi, D_\xi$ is the~set of generators in $S$ which appear at specific positions (in either 
orientation) in words of triagrams $\tau_i$, where $1\leq i\leq 2k$, as seen in Figure~\ref{setsABCD_fig}. 
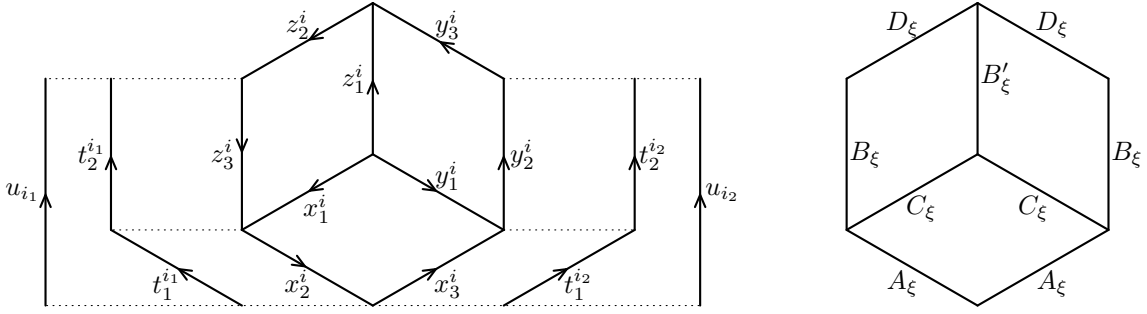
\begin{figure}[H]
	\centering
	\begin{tikzpicture}[line cap = round, line join = round]

\def\midarr{0.5}
\def\boundarr{0.5}
\tikzset{boundedgestyle/.style={
        thick,
        decoration={
            markings,
            mark=at position \boundarr with {\arrow{angle 45}}
        },
        postaction={decorate}
}}
\tikzset{midedgestyle/.style={
        thick,
        decoration={
            markings,
            mark=at position \midarr with {\arrow{angle 45}}
        },
        postaction={decorate}
}}
\tikzset{noarrowstyle/.style={
        thick,
        decoration={
            markings
        },
        postaction={decorate}
}}

\def\l{1}

\def\s{1}
\def\t{1.732}
\def\w{0.577}
\def\txt{0.25}

\def\q{0}
\def\p{0}
\def\r{0}

\def\toffset{1}
\def\uoffset{1.5}

\def\xa{\l*\p}
\def\ya{\l*\q-\l*2*\w*\r}

\draw[midedgestyle] (\xa, \ya) -- (\xa-\l*\t, \ya-\l*\s);
\node at (\xa-\boundarr*\l*\t+\txt*0.5*\l*\s, \ya-\boundarr*\l*\s-\txt*0.5*\t*\l) {\small $x_1^i$};

\draw[boundedgestyle] (\xa-\l*\t, \ya-\l*\s) -- (\xa, \ya-2*\l*\s);
\node at (\xa-\l*\t+\boundarr*\l*\t-\txt*0.5*\l*\s, \ya-\l*\s-\boundarr*\l*\s-\txt*0.5*\t*\l) {\small $x_2^i$};

\draw[boundedgestyle] (\xa-\toffset*\l*\t, \ya-2*\l*\s) -- 
(\xa-\l*\t-\toffset*\l*\t, \ya-\l*\s);
\node at (\xa-\l*\t-\toffset*\l*\t+\boundarr*\l*\t-\txt*0.5*\l*\s, 
\ya-\l*\s-\boundarr*\l*\s-\txt*0.5*\l*\t) {\small $t_1^{i_1}$};
\draw[boundedgestyle] (\xa-\l*\t-\toffset*\l*\t, \ya-\l*\s) --
(\xa-\l*\t-\toffset*\l*\t, \ya+\l*\s);
\node at (\xa-\l*\t-\toffset*\l*\t-\txt*\l*\s, 
\ya-\l*\s+\boundarr*2*\l*\s) {\small $t_2^{i_1}$};
\draw[dotted] (\xa, \ya-2*\l*\s) -- (\xa-\toffset*\l*\t, \ya-2*\l*\s) -- (\xa-\l*\t-\uoffset*\l*\t, \ya-2*\l*\s);
\draw[dotted] (\xa-\l*\t, \ya-\l*\s) -- 
(\xa-\l*\t-\toffset*\l*\t, \ya-\l*\s); 
\draw[dotted]  (\xa-\l*\t, \ya+\l*\s) -- (\xa-\l*\t-\toffset*\l*\t, \ya+\l*\s) -- (\xa-\l*\t-\uoffset*\l*\t, \ya+\l*\s);

\draw[boundedgestyle] (\xa-\l*\t-\uoffset*\l*\t, \ya-2*\l*\s) -- (\xa-\l*\t-\uoffset*\l*\t, \ya+\l*\s);
\node at (\xa-\l*\t-\uoffset*\l*\t-1.2*\txt*\l*\s, \ya-0.5*\l*\s) {\small $u_{i_1}$};

\draw[boundedgestyle] (\xa+\toffset*\l*\t, \ya-2*\l*\s) -- 
(\xa+\l*\t+\toffset*\l*\t, \ya-\l*\s);
\node at (\xa+\l*\t+\toffset*\l*\t-\boundarr*\l*\t+\txt*0.5*\l*\s, 
\ya-\l*\s-\boundarr*\l*\s-\txt*0.5*\l*\t) {\small $t_1^{i_2}$};
\draw[boundedgestyle] (\xa+\l*\t+\toffset*\l*\t, \ya-\l*\s) --
(\xa+\l*\t+\toffset*\l*\t, \ya+\l*\s);
\node at (\xa+\l*\t+\toffset*\l*\t+\txt*\l*\s, 
\ya-\l*\s+\boundarr*2*\l*\s) {\small $t_2^{i_2}$};

\draw[boundedgestyle] (\xa+\l*\t+\uoffset*\l*\t, \ya-2*\l*\s) -- (\xa+\l*\t+\uoffset*\l*\t, \ya+\l*\s);
\node at (\xa+\l*\t+\uoffset*\l*\t+1.2*\txt*\l*\s, \ya-0.5*\l*\s) {\small $u_{i_2}$};
\draw[dotted] (\xa, \ya-2*\l*\s) -- (\xa+\toffset*\l*\t, \ya-2*\l*\s) -- (\xa+\l*\t+\uoffset*\l*\t, \ya-2*\l*\s);
\draw[dotted] (\xa+\l*\t, \ya-\l*\s) -- 
(\xa+\l*\t+\toffset*\l*\t, \ya-\l*\s); 
\draw[dotted]  (\xa+\l*\t, \ya+\l*\s) -- (\xa+\l*\t+\toffset*\l*\t, \ya+\l*\s) -- (\xa+\l*\t+\uoffset*\l*\t, \ya+\l*\s);

\draw[boundedgestyle] (\xa, \ya-2*\l*\s) -- (\xa+\l*\t, \ya-\l*\s);
\node at (\xa+\boundarr*\l*\t+\txt*0.5*\l*\s, \ya-2*\l*\s+\boundarr*\l*\s-\txt*0.5*\t*\l) {\small $x_3^i$};

\def\xb{\l*\p+\l*\r}
\def\yb{\l*\q+\l*\w*\r}

\draw[midedgestyle] (\xb, \yb) -- (\xb+\l*\t, \yb-\l*\s);
\node at (\xb+\boundarr*\l*\t+\txt*0.5*\l*\s, \yb-\boundarr*\l*\s+\txt*0.5*\t*\l) {\small $y_1^i$};

\draw[boundedgestyle] (\xb+\l*\t, \yb-\l*\s) -- (\xb+\l*\t, \yb+\l*\s);
\node at (\xb+\l*\t+\txt*\l*\s, \yb-\l*\s+\boundarr*2*\l*\s) {\small $y_2^i$};

\draw[boundedgestyle] (\xb+\l*\t, \yb+\l*\s) -- (\xb, \yb+2*\l*\s);
\node at (\xb+\l*\t-\boundarr*\l*\t+\txt*0.5*\l*\s, \yb+\l*\s+\boundarr*\l*\s+\txt*0.5*\t*\l) {\small $y_3^i$};

\def\xc{\l*\p-\l*\r}
\def\yc{\l*\q+\l*\w*\r}

\draw[midedgestyle] (\xc, \yc) -- (\xc, \yc+2*\l*\s);
\node at (\xc-\txt*\l*\s, \yc+\boundarr*2*\l*\s) {\small $z_1^i$};

\draw[boundedgestyle] (\xc, \yc+2*\l*\s) -- (\xc-\l*\t, \yc+\l*\s);
\node at (\xc-\boundarr*\l*\t-\txt*0.5*\l*\s, \yc+2*\l*\s-\boundarr*\l*\s+\txt*0.5*\t*\l) {\small $z_2^i$};

\draw[boundedgestyle] (\xc-\l*\t, \yc+\l*\s) -- (\xc-\l*\t, \yc-\l*\s);
\node at (\xc-\l*\t-\txt*\l*\s, \yc+\l*\s-\boundarr*2*\l*\s) {\small $z_3^i$};

\def\q{0}
\def\p{8}
\def\r{0}

\def\xa{\l*\p}
\def\ya{\l*\q-\l*2*\w*\r}
\draw[noarrowstyle] (\xa, \ya) -- (\xa-\l*\t, \ya-\l*\s);
\node at (\xa-\boundarr*\l*\t+\txt*0.5*\l*\s, \ya-\boundarr*\l*\s-\txt*0.5*\t*\l) {\small $C_\xi$};
\draw[noarrowstyle] (\xa-\l*\t, \ya-\l*\s) -- (\xa, \ya-2*\l*\s);
\node at (\xa-\l*\t+\boundarr*\l*\t-\txt*0.5*\l*\s, \ya-\l*\s-\boundarr*\l*\s-\txt*0.5*\t*\l) {\small $A_\xi$};
\draw[noarrowstyle] (\xa, \ya-2*\l*\s) -- (\xa+\l*\t, \ya-\l*\s);
\node at (\xa+\boundarr*\l*\t+\txt*0.5*\l*\s, \ya-2*\l*\s+\boundarr*\l*\s-\txt*0.5*\t*\l) {\small $A_\xi$};

\def\xb{\l*\p+\l*\r}
\def\yb{\l*\q+\l*\w*\r}
\draw[noarrowstyle] (\xb, \yb) -- (\xb+\l*\t, \yb-\l*\s);
\node at (\xb+\boundarr*\l*\t-\txt*0.5*\l*\s, \yb-\boundarr*\l*\s-\txt*0.5*\t*\l) {\small $C_\xi$};
\draw[noarrowstyle] (\xb+\l*\t, \yb-\l*\s) -- (\xb+\l*\t, \yb+\l*\s);
\node at (\xb+\l*\t+\txt*\l*\s, \yb-\l*\s+\boundarr*2*\l*\s) {\small $B_\xi$};
\draw[noarrowstyle] (\xb+\l*\t, \yb+\l*\s) -- (\xb, \yb+2*\l*\s);
\node at (\xb+\l*\t-\boundarr*\l*\t+\txt*0.5*\l*\s, \yb+\l*\s+\boundarr*\l*\s+\txt*0.5*\t*\l) {\small $D_\xi$};

\def\xc{\l*\p-\l*\r}
\def\yc{\l*\q+\l*\w*\r}
\draw[noarrowstyle] (\xc, \yc) -- (\xc, \yc+2*\l*\s);
\node at (\xc+\txt*\l*\s, \yc+\boundarr*2*\l*\s) {\small $B'_\xi$};
\draw[noarrowstyle] (\xc, \yc+2*\l*\s) -- (\xc-\l*\t, \yc+\l*\s);
\node at (\xc-\boundarr*\l*\t-\txt*0.5*\l*\s, \yc+2*\l*\s-\boundarr*\l*\s+\txt*0.5*\t*\l) {\small $D_\xi$};
\draw[noarrowstyle] (\xc-\l*\t, \yc+\l*\s) -- (\xc-\l*\t, \yc-\l*\s);
\node at (\xc-\l*\t+\txt*\l*\s, \yc+\l*\s-\boundarr*2*\l*\s) {\small $B_\xi$};

\end{tikzpicture}
	\vspace{1em}
	\caption{On the left, generators defining $\tau_i$, $u_i$,
	and $u_{i+1}$ for $1\leq i\leq 2k$. Indices $i_1, i_2$ are such that $\{i_1, i_2\}=\{i,i+1\}$. On the right, definition of sets $A_\xi, B_\xi, B'_\xi, C_\xi, D_\xi$.}
\label{setsABCD_fig}
\end{figure}

For $1\leq i\leq 2k$, write $u_i=t_1^it_2^i$, where $t_1^i, t_2^i\in S\cup S^{-1}$. Additionally, let $t_1^{2k+1}=t_1^1$, 
$t_2^{2k+1}=t_2^1$, so that also $u_{2k+1}=t_1^{2k+1}t_2^{2k+1}$.
In this notation, the compatibility condition 
$\{u_i, u_{i+1}\}=\{\partial_2(\tau_i)^{-1}, \partial_3(\tau_i)\}$
reads as $\{t_1^it_2^i, t_1^{i+1}t_2^{i+1}\}=\{(x_2^i)^{-1}(z_3^i)^{-1}, x_3^iy_2^i\}$,
so that $\{t_1^i, t_1^{i+1}\}=\{(x_2^i)^{-1}, x_3^i\}$ and $\{t_2^i, t_2^{i+1}\}=\{(z_3^i)^{-1}, y_2^i\}$
for every $1\leq i\leq 2k$. This implies that $A_\xi =\{|t_1^i| : 1\leq i\leq 2k\}$ and $B_\xi = \{|t_2^i| : 
1\leq i\leq 2k\}$, so $|A_\xi|\leq 2k$ and $|B_\xi|\leq 2k$. We also have the following straightforward bounds: $|B'_\xi|\leq 
2k$, $|C_\xi|\leq 4k$, and $|D_\xi|\leq 4k$. Define $S_\xi = A_\xi\cup B_\xi\cup B'_\xi\cup C_\xi\cup D_\xi$ to be the set of all 
generators in $S$ appearing (possibly inverted) in the words of triagrams $\tau_1,\tau_2,\ldots,\tau_{2k}$.
Combining the inequalities, we get 
\begin{equation*}
|S_\xi|\leq 2k+2k+2k+4k+4k=14k. 
\end{equation*} 
If $|S_\xi|=14k$, then we say that the triagram cycle $\xi$ is \emph{generic}. Equivalently, $\xi$ is generic if and only if $A_\xi, B_\xi, B'_\xi, C_\xi, D_\xi$ are pairwise disjoint,
$|A_\xi|=|B_\xi|=|B'_\xi|=2k$ and $|C_\xi|=|D_\xi|=4k$. 
Additionally, if $\xi$ is generic, then also $|F_\xi^1|=|F_\xi^2|=|F_\xi^3|=2k$, by Remark~\ref{F_xi_ineq_rem}. If $\xi$ is generic, then so is any $\xi'\in\Theta(\xi)$.

To obtain our final bound on $\mathbb{E}\left(\Tr(A^{2k})\mathds{1}_\mathcal{R}\right)$, 
proving Lemma~\ref{trace_of_power_lem}, we analyze the last sum of (\ref
{orbit_sum_eq}) in the following steps. First, we bound the sizes of 
$\Theta$-orbits (Remark~\ref{orbit_size_rem}) and the size of $\Xi^0_{2k}$ 
(Remark~\ref{orbits_rep_rem}). Then we show that there 
are exactly $2^{2k}$ generic triagram cycles in $\Xi^0_{2k}$ 
(Remark~\ref{generic_num_rem}). Finally, we obtain bounds on the
terms of form $\left|\Theta(\xi)\right|p^{\left|F_{\xi}^1\right|+\left|F_{\xi}^2\right|+\left|F_{\xi}^3\right|}$, separately for generic and non-generic $\xi\in\Xi^0_{2k}$ (Lemma~\ref{term_bound_lem}).

\begin{rem}\label{orbit_size_rem}
For every $\xi\in\Xi_{2k}$, $|\Theta(\xi)|\leq (2n)^{|S_\xi|}$.
\end{rem}
\begin{proof}
Fix $\xi\in\Xi_{2k}$. When varying $\theta\in\Theta$, the element 
$\theta(\xi)$ is uniquely determined by
the restriction $\theta|_{S_\xi}:S_\xi \rightarrow S\cup S^{-1}$,
so that $|\Theta(\xi)|\leq (2n)^{|S_\xi|}$.
\end{proof}
\begin{rem}\label{orbits_rep_rem}
$|\Xi_{2k}^0|=O_k(1)$.
\end{rem}

\begin{proof}
Let $\xi\in\Xi_{2k}$. Write $S=\{s_1,s_2,\ldots,s_n\}$. There exists $\theta\in\Theta$ such that $\theta(S_\xi) =\{s_1, s_2, \ldots, 
s_{|S_\xi|}\}$. 
For $m=\min(14k,n)$ and denoting $S_m=\{s_1,s_2,\ldots,s_m\}$, we have $|S_\xi|\leq m$, so that $\xi'=\theta(\xi)$ is an element of the orbit $\Theta(\xi)$, satisfying $S_{\xi'}\subseteq S_m$. In other words, every $\Theta$-orbit has non-empty intersection with the set
\begin{equation*}
P_m=\{\xi\in\Xi_{2k} : S_\xi\subseteq S_m\},
\end{equation*}
implying that $|\Xi_{2k}^0|\leq |P_m|$.

Now it suffices to show that $|P_m|=O_k(1)$. For a~rough bound on $|P_m|$, we note that every $\xi=(\pmb{u},\pmb{\tau})\in P_m$ is uniquely determined by the~choice of $2k\cdot 2$ generators in $S_m\cup S_m^{-1}$ defining words $u_i$ for $1\leq i\leq 2k$, and $2k\cdot 3\cdot 3$ generators defining triagrams $\tau_i$ for $1\leq i\leq 2k$, so that $|P_m|\leq (2m)^{4k+18k}\leq(28k)^{22k}$.
\end{proof}

\begin{rem}\label{generic_num_rem}
If $n\geq 14k$, then there are exactly $2^{2k}$ generic triagram cycles in $\Xi_{2k}^0$.
\end{rem}
\begin{proof}
Consider any $\xi\in\Xi_{2k}$. As $u_i\neq u_{i+1}$ and $\{u_i, u_{i+1}\}=\{(x_2^i)^{-1}(z_3^i)^{-1}, x_3^iy_2^i\}$ for ${1\leq i\leq 2k}$,
we can define a~function $\sigma_\xi : \{1,2,\ldots,{2k}\}\longrightarrow \{0,1\}$ by:
\begin{equation}\label{sigma_def_eq}
\sigma_\xi(i) =
\begin{cases}
0,\text{ if }(u_i,u_{i+1})=((x_2^i)^{-1}(z_3^i)^{-1}, x_3^iy_2^i)\\
1,\text{ if }(u_i,u_{i+1})=(x_3^iy_2^i, (x_2^i)^{-1}(z_3^i)^{-1})
\end{cases}
\end{equation}
for every $1\leq i\leq 2k$. If $\xi_1, \xi_2\in \Xi_{2k}$ lie in the same $\Theta$-orbit, then $\sigma_{\xi_1}=\sigma_{\xi_2}$.

To prove that there are at least $2^{2k}$ generic triagram cycles $\xi\in\Xi_{2k}^0$, it suffices to show that
for any of the $2^{2k}$ distinct functions $\sigma : \{1,2,\ldots,2k\}\longrightarrow \{0,1\}$, there exists a generic $\xi\in\Xi_{2k}$, such that $\sigma_\xi=\sigma$. Given $\sigma$, we can construct such $\xi$ as follows. First, write
$S=\{s_1,s_2,\ldots,s_n\}$ and let ${u_i=s_{2i-1}s_{2i}}$ for ${1\leq i\leq 2k}$. 
Then let the values of $x_2^i, x_3^i, y_2^i, z_3^i$, for $1\leq i\leq 2k$, be such that $\sigma$ satisfies the condition
(\ref{sigma_def_eq}). Choose $x_1^i, y_1^i, y_3^i, z_1^i, z_2^i$, for $1\leq i\leq 2k$, to be the generators
$s_{4k+1},s_{4k+2},\ldots,s_{14k}$ in any order (here we need the assumption that $n\geq 14k$). 
The non-reducibility conditions 
coming from the definition of a triagram are satisfied, as any two generators $a, b\in S\cup S^{-1}$, defining $\xi$, for which it is required that $a\neq b^{-1}$, satisfy a stronger condition $|a|\neq |b|$.

On the other hand, the fact that there are at most $2^{2k}$ generic triagram cycles $\xi\in\Xi_{2k}^0$, is a~consequence of the following observation: if $\xi, \xi'\in\Xi_{2k}$ are generic triagram cycles and $\sigma_{\xi}=\sigma_{\xi'}$,
then $\xi$ and $\xi'$ lie in the same $\Theta$-orbit. 
To prove this, let us denote, for $1\leq i\leq 2k$, $1\leq j\leq 3$, and $1\leq q\leq 2$, by $x_j^{i\prime}, 
y_j^{i\prime},z_j^{i\prime},t_q^{i\prime},u_i^{\prime}$,
the generators and words defined for $\xi'$ 
analogously to how $x_j^i,y_j^i,z_j^i,t_q^i,u_i$ are defined for $\xi$. 
As $\left\{ |t_1^i|, |t_2^i|, |t_1^{i+1}|, |t_2^{i+1}| \right\} = \left\{|x_2^i|, |x_3^i|, |y_2^i|, |z_3^i|\right\}$ for~$1\leq i\leq 2k$, we have:
\begin{equation*}
S_\xi 
=\bigcup_{i=1}^{2k}\left\{|t_1^i|, |t_2^i|, |x_1^i|, |y_1^i|, |y_3^i|, |z_1^i|, |z_2^i|\right\}.
\end{equation*}
The set $S_\xi$ has $14k=2k\cdot 7$ elements, so the sets $\left\{|t_1^i|, |t_2^i|, |x_1^i|, |y_1^i|, |y_3^i|, |z_1^i|, |z_2^i|\right\}\subseteq S$
are pairwise disjoint for distinct 
$1\leq i\leq 2k$, and each of them contains 7 elements.
The analogous statement is true for $\xi'$, so there exists 
$\theta\in\Theta$, such that $\theta(t_1^i)=t_1^{i\prime},\ \theta(t_2^i)=t_2^{i\prime},\ \theta(x_1^i)=x_1^{i\prime},\ \theta(y_1^i)=y_1^{i\prime},\ \theta(y_3^i)=y_3^{i\prime},\ \theta(z_1^i)=z_1^{i\prime}$, and $\theta(z_2^i)=z_2^{i\prime}$ for $1\leq i\leq 2k$. The conditions $\theta(t_1^i)=t_1^{i\prime}$, and $\theta(t_2^i)=t_2^{i\prime}$ mean that $\theta(u_i)=u_i'$ for $1\leq i\leq 2k$. As $\sigma_\xi=\sigma_{\xi'}$, it implies
that also $\theta(x_2^i)=x_2^{i\prime},\ \theta(x_3^i)=x_3^{i\prime},\ \theta(y_2^i)=y_2^{i\prime}$ and $\theta(z_3^i)=z_3^{i\prime}$ for $1\leq i\leq 2k$, so $\theta(\xi)=\xi'$.

\end{proof}
\begin{lem}\label{term_bound_lem}
Let $\xi\in\Xi_{2k}$. If $\xi$ is generic, then $\left|\Theta(\xi)\right|p^{\left|F_{\xi}^1\right|+\left|F_{\xi}^2\right|+
\left|F_{\xi}^3\right|}\leq (2n)^{14k}p^{6k}$. If $\xi$ is not generic, then
$\left|\Theta(\xi)\right|p^{\left|F_{\xi}^1\right|+\left|F_{\xi}^2\right|+
\left|F_{\xi}^3\right|}=O_\rho\left((2n)^{14k-\frac{1}{3}\alpha}p^{6k}\right)$.
\end{lem}
\begin{proof}

If $\xi$ is generic, then, as we noted before, $|F_\xi^1|=|F_\xi^2|=|F_\xi^3|=2k$. Hence,
by Remark~\ref{orbit_size_rem}, we have $\left|\Theta(\xi)\right|p^{\left|F_{\xi}^1\right|+\left|F_{\xi}^2
\right|+\left|F_{\xi}^3\right|}\leq (2n)^{|S_\xi|} p^{6k} = (2n)^{14k}p^{6k}$.

Now let us assume that $\xi$ is not generic. Since $p\sim\frac{1}{3}(2n)^{\frac{-7+\alpha}{3}}$, 
there exists a constant $B_p > 0$, such that $p\geq B_p(2n)^{\frac{-7+\alpha}{3}}$ for all $n\geq 1$. 
It can be chosen so that $B_p<1$. 
Clearly, $|F_\xi^1|, |F_\xi^2|, |F_\xi^3|\leq 2k$, so the number
$f_\xi=|F_\xi^1|+|F_\xi^2|+|F_\xi^3|-6k$ satisfies $f_\xi \in [-6k, 0]$ and so
\begin{equation*}
p^{f_\xi} \leq \left(B_p(2n)^{\frac{-7+\alpha}{3}}\right)^{f_\xi}
\leq B_p^{-6k}(2n)^{\frac{-7+\alpha}{3}f_\xi}
=O_\rho\left((2n)^{\frac{-7+\alpha}{3}f_\xi}\right).
\end{equation*}
By using this and Remark~\ref{orbit_size_rem}, we get
\begin{equation*}
\begin{split}
\left|\Theta(\xi)
\right|p^{\left|F_{\xi}^1\right|+\left|F_{\xi}^2\right|+
\left|F_{\xi}^3\right|} &\leq 
(2n)^{\left|S_\xi\right|}p^{\left|F_{\xi}^1\right|+\left|F_{\xi}^2\right|+
\left|F_{\xi}^3\right|}\\ 
&= (2n)^{\left|S_\xi\right|}p^{f_\xi}p^{6k}\\
&= O_\rho\left((2n)^{e_\xi}p^{6k}\right),
\end{split}
\end{equation*}
where we denote
\begin{equation*}
e_\xi=\left|S_\xi\right|+\frac{-7+\alpha}{3}f_\xi
=\left|S_\xi\right|+\frac{-7+\alpha}{3}\left(|F_{\xi}^1|+|F_{\xi}^2|+|F_{\xi}^3|-6k\right).
\end{equation*}
Now it suffices to show that $e_\xi\leq 14k-\frac{\alpha}
{3}$. To achieve this, we will first establish the following inequalities:
\begin{enumerate}[label=(\Roman*)]
    \item\label{ineq_B1} $|S_\xi|\leq |A_\xi|+|B_\xi\cup B_\xi'|+|C_\xi|+|D_\xi|$,
    \item\label{ineq_B2} $|C_\xi|\leq \frac{4}{3}|F_\xi^1|+\frac{1}{3}|F_\xi^2|+\frac{1}{3}|F_\xi^3|$,
    \item\label{ineq_B3} $|D_\xi|\leq |F_\xi^2|+|F_\xi^3|$,
    \item\label{ineq_B4} $|A_\xi| \leq \frac{3-\alpha}{3}|F_\xi^1|+\frac{2}{3}\alpha k$,
    \item\label{ineq_B5} $|B_\xi\cup B_\xi'|\leq \frac{3-\alpha}{3}(|F_\xi^2|+|F_\xi^3|)+\frac{4}{3}\alpha k$.
\end{enumerate}
Additionally, it will be clear from our arguments, that if any of the~inequalities \ref{ineq_B1}-\ref{ineq_B5} is strict, then its sides differ by at least $\min\left(1, \frac{1}{3}, \frac{\alpha}{3}\right)=\frac{\alpha}{3}$. We prove the inequalities as follows.

\textbf{\ref{ineq_B1}}. This follows straightforwardly from the definition of $S_\xi$. 

\textbf{\ref{ineq_B2}}. 
By Remark~\ref{F_xi_ineq_rem}, we have:
\begin{equation*}
|C_\xi| \leq \left|\left\{\left|x_1^i\right| : 1\leq i\leq 2k\right\}\right|
+ \left|\left\{\left|y_1^i\right| : 1\leq i\leq 2k\right\}\right|
\leq 2|F_\xi^1|,
\end{equation*}
and, by using the same remark to differently bound the~middle sum, $|C_\xi|\leq |F_\xi^3|+|F_\xi^2|$. By combining the two bounds on $|C_\xi|$, we obtain:
\begin{equation*}
|C_\xi| = \frac{2}{3}|C_\xi|+\frac{1}{3}|C_\xi| \leq\frac{2}{3}\left(2|F_\xi^1|\right)+\frac{1}{3}\left(|F_\xi^3|+|F_\xi^2|\right)=\frac{4}{3}|F_\xi^1|+\frac{1}{3}|F_\xi^2|+\frac{1}{3}|F_\xi^3|.
\end{equation*}
If equality holds in \ref{ineq_B2}, then 
$|C_\xi| = 2|F_\xi^1|=|F_\xi^2|+|F_\xi^3|$.

\textbf{\ref{ineq_B3}}. By Remark~\ref{F_xi_ineq_rem}, we have
\begin{equation*}
|D_\xi|\leq \left|\left\{\left|y_3^i\right| : 1\leq i\leq 2k\right\}\right|
+ \left|\left\{\left|z_2^i\right| : 1\leq i\leq 2k\right\}\right|
\leq \left|F_\xi^2\right|+\left|F_\xi^3\right|.
\end{equation*}

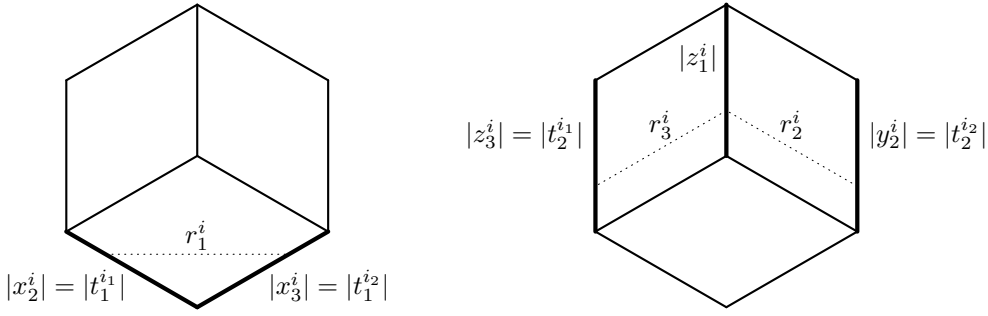
\begin{figure}[H]
	\centering
	\begin{tikzpicture}[line cap = round, line join = round]

\def\midarr{0.5}
\def\boundarr{0.5}
\def\Hratio{0.7}
\def\Kratio{0.3}
\def\Kvertextxty{0.3}
\def\Kedgetxty{1.25}
\tikzset{boundedgestyle/.style={
        thick,
        decoration={
            markings,
            mark=at position \boundarr with {\arrow{angle 45}}
        },
        postaction={decorate}
}}
\tikzset{midedgestyle/.style={
        thick,
        decoration={
            markings,
            mark=at position \midarr with {\arrow{angle 45}}
        },
        postaction={decorate}
}}
\tikzset{noarrowstyle/.style={
        thick,
        decoration={
            markings
        },
        postaction={decorate}
}}
\def\l{1}
\def\s{1}
\def\t{1.732}
\def\w{0.577}
\def\txt{0.25}

\def\q{0}
\def\p{0}
\def\r{0}
\def\xa{\l*\p}
\def\ya{\l*\q-\l*2*\w*\r}
\draw[noarrowstyle] (\xa, \ya) -- (\xa-\l*\t, \ya-\l*\s);
\draw[noarrowstyle, ultra thick] (\xa-\l*\t, \ya-\l*\s) -- (\xa, \ya-2*\l*\s);
\node at (\xa-\l*\t+\boundarr*\l*\t-\txt*3.5*\l*\s, \ya-\l*\s-\boundarr*\l*\s-\txt*0.5*\t*\l) {\small $|x_2^i|=|t_1^{i_1}|$};
\draw[noarrowstyle, ultra thick] (\xa, \ya-2*\l*\s) -- (\xa+\l*\t, \ya-\l*\s);
\node at (\xa+\boundarr*\l*\t+\txt*3.5*\l*\s, \ya-2*\l*\s+\boundarr*\l*\s-\txt*0.5*\t*\l) {\small $|x_3^i|=|t_1^{i_2}|$};
\draw[dotted] (\xa-\Hratio*\l*\t, \ya-2*\l*\s+\Hratio*\l*\s) -- 
(\xa+\Hratio*\l*\t, \ya-2*\l*\s+\Hratio*\l*\s);
\node at (\xa, \ya-2*\l*\s+\Hratio*\l*\s+\txt*\l*\s) {\small $r_1^i$};
\def\xb{\l*\p+\l*\r}
\def\yb{\l*\q+\l*\w*\r}
\draw[noarrowstyle] (\xb, \yb) -- (\xb+\l*\t, \yb-\l*\s);
\draw[noarrowstyle] (\xb+\l*\t, \yb-\l*\s) -- (\xb+\l*\t, \yb+\l*\s);
\draw[noarrowstyle] (\xb+\l*\t, \yb+\l*\s) -- (\xb, \yb+2*\l*\s);

\def\xc{\l*\p-\l*\r}
\def\yc{\l*\q+\l*\w*\r}
\draw[noarrowstyle] (\xc, \yc) -- (\xc, \yc+2*\l*\s);

\draw[noarrowstyle] (\xc, \yc+2*\l*\s) -- (\xc-\l*\t, \yc+\l*\s);

\draw[noarrowstyle] (\xc-\l*\t, \yc+\l*\s) -- (\xc-\l*\t, \yc-\l*\s);

\def\q{0}
\def\p{7}
\def\r{0}

\def\xa{\l*\p}
\def\ya{\l*\q-\l*2*\w*\r}
\draw[noarrowstyle] (\xa, \ya) -- (\xa-\l*\t, \ya-\l*\s);

\draw[noarrowstyle] (\xa-\l*\t, \ya-\l*\s) -- (\xa, \ya-2*\l*\s);

\draw[noarrowstyle] (\xa, \ya-2*\l*\s) -- (\xa+\l*\t, \ya-\l*\s);

\def\xb{\l*\p+\l*\r}
\def\yb{\l*\q+\l*\w*\r}
\draw[noarrowstyle] (\xb, \yb) -- (\xb+\l*\t, \yb-\l*\s);

\draw[noarrowstyle, ultra thick] (\xb+\l*\t, \yb-\l*\s) -- (\xb+\l*\t, \yb+\l*\s);

\draw[dotted] (\xb+\l*\t, \yb-\l*\s+\Kratio*2*\l*\s) -- (\xc, \yc+\Kratio*2*\l*\s);

\node at (\xb+0.5*\l*\t, \yb-\l*\s+\Kratio*2*\l*\s+0.5*\l*\s+\Kedgetxty*\txt*\l*\s) {\small $r_2^i$};
\node at (\xb+\l*\t+3.75*\txt*\l*\s, \yb-\l*\s+\boundarr*2*\l*\s+\Kvertextxty*\l*\s) {\small $|y_2^i|=|t_2^{i_2}|$};
\draw[noarrowstyle] (\xb+\l*\t, \yb+\l*\s) -- (\xb, \yb+2*\l*\s);

\def\xc{\l*\p-\l*\r}
\def\yc{\l*\q+\l*\w*\r}
\draw[noarrowstyle, ultra thick] (\xc, \yc) -- (\xc, \yc+2*\l*\s);
\node at (\xc-\txt*1.5*\l*\s, \yc+\boundarr*2*\l*\s+\Kvertextxty*\l*\s) {\small $|z_1^i|$};
\draw[noarrowstyle] (\xc, \yc+2*\l*\s) -- (\xc-\l*\t, \yc+\l*\s);

\draw[noarrowstyle, ultra thick] (\xc-\l*\t, \yc+\l*\s) -- (\xc-\l*\t, \yc-\l*\s);

\draw[dotted] (\xc, \yc + \Kratio*2*\l*\s) -- (\xc-\l*\t, \yc-\l*\s + \Kratio*2*\l*\s);

\node at (\xc - 0.5*\l*\t, \yc - 0.5*\l*\s + \Kratio*2*\l*\s + \Kedgetxty*\txt*\l*\s) {\small $r_3^i$};
\node at (\xc-\l*\t-3.75*\txt*\l*\s, \yc+\l*\s-\boundarr*2*\l*\s+\Kvertextxty*\l*\s) {\small $|z_3^i|=|t_2^{i_1}|$};

\end{tikzpicture}
	\vspace{1em}
	\caption{Definition of auxilliary graphs. On the left, in the graph $H$, the~edge 
	$r_1^i\in F_\xi^1$ joining vertices $|x_2^i|,|x_3^i|\in A_\xi$.
	On the right, in the~graph $K$, the~edge $r_2^i\in \widetilde{F_\xi^2}$ joining vertices $|y_2^i|$ and $|z_1^i|$, and
	the~edge $r_3^i\in \widetilde{F_\xi^3}$ joining vertices $|z_1^i$| and $|z_3^i|$.}
\end{figure}

\textbf{\ref{ineq_B4}}. Let us construct an auxiliary graph $H$ with
vertex set $A_\xi$ and edge set $F_\xi^1$. Given an edge $r=x_1x_2x_3y_1^{-1}\in F_\xi^1$, declare the endpoints of $r$ to be $|x_2|$ and $|x_3|$.
Under this definition, for every $1\leq i\leq 2k$, $r_1^i$ is an~edge joining 
$|x_2^i|$ and $|x_3^i|$. The condition: $\{|t_1^i|, |t_1^{i+1}|\}=\{|x_2^i|, |x_3^i|\}$ 
for $1\leq i\leq 2k$, means that the sequence of edges $c_H=(r_1^1, r_1^2,\ldots,r_1^{2k})$ is a~closed path in $H$,
passing consecutively through the vertices $|t_1^1|, |t_1^2|, \ldots, |t_1^{2k}|, |t_1^1|$, hence through every vertex of $H$ at least once. In particular, the graph $H$ is connected. Let us consider two cases, depending on whether $H$ is a tree, where by a~tree we mean a~simple graph (without multiple edges and loops) that is connected and has no cycles.

\textbf{(i)}. Suppose first that $H$ is not a tree. Then $H$ has at least as many edges
as vertices, i.e.\ $|F_\xi^1|\geq |A_\xi|$. Using also the trivial bound $|F_\xi^1|\leq 2k$, we obtain
\begin{equation*}
\frac{3-\alpha}{3}|F_\xi^1|+\frac{2}{3}\alpha k
\geq \frac{3-\alpha}{3}|F_\xi^1|+\frac{\alpha}{3}|F_\xi^1|
=|F_\xi^1|\geq |A_\xi|,
\end{equation*}
which is \ref{ineq_B4}. In case (i), if equality holds in \ref{ineq_B4}, then $|F_\xi^1|=|A_\xi|=2k$.

\textbf{(ii)}. Now suppose that $H$ is a tree, so $|F_\xi^1|=|A_\xi|-1$. We claim that every edge $e$ of $H$ appears in the sequence $c_H$ at least twice. To see that, let $u$ be the starting vertex
of $c_H$ and denote by $v_1,v_2$ the endpoints of $e$, so 
that $v_2$ is at a~greater distance from $u$, than $v_1$. 
The edge $e$ is traversed by $c_H$ the first time $v_2$ is visited, and has to be later 
backtracked at some point for $c_H$ to be able to return to $u$. This observation
implies that $2k\geq 2|F_\xi^1|$, which simplifies to $k\geq |F_\xi^1|$.
Bounding also $k>\frac{4}{\alpha}$, we get
\begin{equation*}
\begin{split}
\frac{3-\alpha}{3}|F_\xi^1|+\frac{2}{3}\alpha k
&\geq \frac{3-\alpha}{3}|F_\xi^1|+\frac{\alpha}{3}|F_\xi^1|+\frac{1}{3}\alpha k=|F_\xi^1|+\frac{1}{3}\alpha k\\
&=|A_\xi|+\frac{1}{3}\alpha k-1> |A_\xi|+\frac{1}{3}> |A_\xi|.
\end{split}
\end{equation*}
In case (ii), \ref{ineq_B4} holds as a strict inequality.

\textbf{\ref{ineq_B5}}. Similarly as in the proof of \ref{ineq_B4}, we construct an~auxiliary graph $K$. Let the vertex set of $K$ be $B_\xi \cup B'_\xi$ and let the edge set of $K$ be
$\widetilde{F_\xi^2}\cup\widetilde{F_\xi^3}$, where $\widetilde{F_\xi^2}$ and $\widetilde{F_\xi^3}$ are disjoint copies of 
$F_\xi^2$ and $F_\xi^3$. Slightly abusing the notation, given an edge $r=y_1y_2y_3z_1^{-1}\in\widetilde{F_\xi^2}$ (resp. $r=z_1z_2z_3x_1^{-1}\in\widetilde{F_\xi^3}$), declare the endpoints of $r$ to be $|y_2|$ and $|z_1|$ (resp.\ $|z_1|$ and $|z_3|$).
This means that, for every $1\leq i\leq 2k$, $r_2^i\in\widetilde{F_\xi^2}$ is an edge between $|y_2^i|$ and $|z_1^i|$, and $r_3^i\in\widetilde{F_\xi^3}$ is an edge
between $|z_1^i|$ and $|z_3^i|$. As $\{|t_2^i|, |t_2^{i+1}|\}=\{|z_3^i|, |y_2^i|\}$,
this means that the edges $r_2^i$ and $r_3^i$, in some order, form a~path in $K$, passing consecutively through the vertices $|t_2^i|, |z_1^i|$ and $|t_2^{i+1}|$.
By combining these paths for all $1\leq i\leq 2k$, we obtain a closed path $c_K$ of length $2k\cdot 2=4k$, containing every edge of $K$ and passing consecutively through the 
vertices $|t_2^1|, |z_1^1|, |t_2^2|, |z_1^2|, |t_2^3|,\ldots, |t_2^{2k}|, 
|z_1^{2k}|, |t_2^1|$. As every element of $B_\xi\cup B'_\xi$ is of form $|t_2^i|$ or $|z_1^i|$ for some $1\leq i\leq 2k$, $c_K$ passes through all vertices of $K$. In particular, the graph $K$ is connected. As in the~proof of \ref{ineq_B4}, let us consider two cases, depending on whether $K$ is a tree.

\textbf{(i)}. Suppose first that $K$ is not a tree, so $|\widetilde{F_\xi^2}\cup\widetilde{F_\xi^3}|\geq |B_\xi \cup B'_\xi|$. Combining this with the trivial bound
\begin{equation*}
|\widetilde{F_\xi^2}\cup\widetilde{F_\xi^3}|=|F_\xi^2|+|F_\xi^3| \leq 2k+2k=4k,
\end{equation*} 
we get
\begin{equation*}
\begin{split}
    \frac{3-\alpha}{3}(|F_\xi^2|+|F_\xi^3|)+\frac{4}{3}\alpha k
    &=\frac{3-\alpha}{3}|\widetilde{F_\xi^2}\cup\widetilde{F_\xi^3}|+\frac{4}{3}\alpha k\\
    &\geq \frac{3-\alpha}{3}|\widetilde{F_\xi^2}\cup\widetilde{F_\xi^3}|+\frac{\alpha}{3} |\widetilde{F_\xi^2}\cup\widetilde{F_\xi^3}|\\
    &=|\widetilde{F_\xi^2}\cup\widetilde{F_\xi^3}|\geq |B_\xi\cup B'_\xi|,
\end{split}
\end{equation*}
showing \ref{ineq_B5}. If equality holds in \ref{ineq_B5} in case (i),
then $|B_\xi\cup B'_\xi|=|\widetilde{F_\xi^2}\cup\widetilde{F_\xi^3}|=4k$, so 
also $|F_\xi^2|=|F_\xi^3|=2k$.

\textbf{(ii)}. If $K$ is a tree, then $|\widetilde{F_\xi^2}\cup\widetilde{F_\xi^3}|= |B_\xi\cup B'_\xi|-1$. Analogously as in the proof of \ref{ineq_B4}, we see that 
every edge of $K$ appears in the sequence $c_K$ at least twice, so that 
$4k\geq 2|\widetilde{F_\xi^2}\cup\widetilde{F_\xi^3}|$. Equivalently, $2k\geq |\widetilde{F_\xi^2}\cup\widetilde{F_\xi^3}|$. Applying also the fact that
$k>\frac{4}{\alpha}$, we get

\begin{equation*}
\begin{split}
    \frac{3-\alpha}{3}(|F_\xi^2|+|F_\xi^3|)+\frac{4}{3}\alpha k
    &=\frac{3-\alpha}{3}|\widetilde{F_\xi^2}\cup\widetilde{F_\xi^3}|+\frac{4}{3}\alpha k\\
    &\geq \frac{3-\alpha}{3}|\widetilde{F_\xi^2}\cup\widetilde{F_\xi^3}|+\frac{\alpha}{3} |\widetilde{F_\xi^2}\cup\widetilde{F_\xi^3}|+\frac{2}{3}\alpha k\\
    &=|\widetilde{F_\xi^2}\cup\widetilde{F_\xi^3}|+\frac{2}{3}\alpha k
    = |B_\xi\cup B'_\xi|+\frac{2}{3}\alpha k-1\\
    &> |B_\xi\cup B'_\xi|+\frac{5}{3}> |B_\xi\cup B'_\xi|.
\end{split}
\end{equation*}
Hence the inequality \ref{ineq_B5} is always strict when $K$ is a tree.

We have thus proved inequalities \ref{ineq_B1}-\ref{ineq_B5}. Let us now see that 
the assumption that $\xi$ is not generic implies that at least one of these inequalities
is strict. Assume to the contrary, that equality holds in all the inequalities
\ref{ineq_B1}-\ref{ineq_B5}. As shown, equalities in \ref{ineq_B4} and \ref{ineq_B5} imply that $|F_\xi^1|=|F_\xi^2|=|F_\xi^3|=2k$, $|A_\xi|=2k$, and $|B_\xi\cup B'_\xi|=4k$.
Equalities in \ref{ineq_B2} and \ref{ineq_B3} hence give $|C_\xi|=4k$
and $|D_\xi|=4k$, so equality in \ref{ineq_B1} now means that $|S_\xi|=14k$, contradicting the assumption that $\xi$ is not generic.

By combining the inequalities \ref{ineq_B1}-\ref{ineq_B5}, we obtain
\begin{equation*}
\begin{split}
|S_\xi|&\leq |A_\xi|+|B_\xi\cup B_\xi'|+|C_\xi|+|D_\xi|\\
&\leq\left(\frac{3-\alpha}{3}|F_\xi^1|+\frac{2}{3}\alpha k\right)
+ \left(\frac{3-\alpha}{3}(|F_\xi^2|+|F_\xi^3|)+\frac{4}{3}\alpha k\right)
+ \left(\frac{4}{3}|F_\xi^1|+\frac{1}{3}|F_\xi^2|+\frac{1}{3}|F_\xi^3|\right)\\
&\quad + \left(|F_\xi^2|+|F_\xi^3|\right)\\
&= \frac{7-\alpha}{3}(|F_\xi^1|+|F_\xi^2|+|F_\xi^3|)+2\alpha k.
\end{split}
\end{equation*}

Since at least one of the inequalities \ref{ineq_B1}-\ref{ineq_B5} is strict,
meaning that its sides differ by at least $\frac{\alpha}{3}$, we have in fact the following stronger bound:

\begin{equation*}
|S_\xi|\leq \frac{7-\alpha}{3}(|F_\xi^1|+|F_\xi^2|+|F_\xi^3|)+2\alpha k-\frac{\alpha}{3},
\end{equation*}
which can be easily rearranged as $e_\xi\leq 14k-\frac{\alpha}{3}$.
\end{proof}

We are now in the position to prove Lemma~\ref{trace_of_power_lem}.
\begin{proof}[Proof of Lemma~\ref{trace_of_power_lem}]
Let $n\geq 14k$. Let
$\Xi_{2k}^{0,g}=\{\xi\in \Xi_{2k}^0 : \xi\text{ is generic}\}$.
Recall that $d_0=2(2n)^7p^3$ and $\delta=n^{-\frac{\alpha}{3}}$.
Using inequality~(\ref{orbit_sum_eq}), Remarks \ref{orbits_rep_rem} and \ref{generic_num_rem}, and Lemma~\ref{term_bound_lem}, we obtain

\begin{equation}\label{final_trace_ineq}
\begin{split}
\mathbb{E}\left(\Tr(A^{2k})\mathds{1}_\mathcal{R}\right)
&\leq (1-\delta)^{-2k}d_0^{-2k}\sum_{\xi\in\Xi_{2k}^0} \left|\Theta(\xi)\right|p^{\left|F_{\xi}^1\right|+\left|F_{\xi}^2\right|+\left|F_{\xi}^3\right|}\\
&\leq (1-\delta)^{-2k}d_0^{-2k}\left(\left|\Xi_{2k}^{0,g}\right|(2n)^{14k}p^{6k}+
\left|\Xi_{2k}^0\setminus\Xi_{2k}^{0,g}\right|O_\rho\left((2n)^{14k-\frac{1}{3}\alpha}p^{6k}\right)\right)\\
&= (1-\delta)^{-2k}d_0^{-2k}\left(2^{2k}(2n)^{14k}p^{6k}+
O_{\rho,k}\left((2n)^{14k-\frac{1}{3}\alpha}p^{6k}\right)\right)\\
&= (1-\delta)^{-2k}\left(d_0^{-1}2(2n)^7p^3\right)^{2k}\left(1+
2^{-2k}O_{\rho,k}\left((2n)^{-\frac{\alpha}{3}}\right)\right)\\
&=\left(1-n^{-\frac{\alpha}{3}}\right)^{-2k}\left(1+
O_\rho\left(n^{-\frac{\alpha}{3}}\right)\right),
\end{split}
\end{equation}
where in the last step we note that $k$ is defined as a~function of $\alpha$,
and we remember that $\alpha$ is determined by $\rho$,
so the notation $O_{\rho, k}(\cdot)$ is equivalent to $O_\rho(\cdot)$.
Consider $n$ large enough, so that $n^{-\frac{\alpha}{3}}<\frac{1}{2}$. By the mean value theorem applied to the
function $f(x)=x^{-2k}$ and interval $(1-n^{-\frac{\alpha}{3}},1)$, we have
$(1-n^{-\frac{\alpha}{3}})^{-2k}-1=
2kt^{-2k-1}n^{-\frac{\alpha}{3}}$
for some $t\in(1-n^{-\frac{\alpha}{3}},1)\subset(\frac{1}{2},1)$.
In particular, $(1-n^{-\frac{\alpha}{3}})^{-2k}\leq
1+2^{2k+2}kn^{-\frac{\alpha}{3}}=1+O_\rho\left(n^{-\frac{\alpha}{3}}\right)$.
Using this with (\ref{final_trace_ineq}), we get:
\begin{equation*}
\begin{split}
\mathbb{E}\left(\Tr(A^{2k})\mathds{1}_\mathcal{R}\right)
&\leq\left(1-n^{-\frac{\alpha}{3}}\right)^{-2k}\left(1+
O_\rho\left(n^{-\frac{\alpha}{3}}\right)\right)\\
&\leq\left(1+
O_\rho\left(n^{-\frac{\alpha}{3}}\right)\right)^2\\
&= 1+O_\rho\left(n^{-\frac{\alpha}{3}}\right).
\end{split}
\end{equation*}
We thus proved Lemma~\ref{trace_of_power_lem} for $n$ large enough, given $k$, 
which is sufficient for full generality.
\end{proof}

\section{Asymptotic regularity of $L_i$}\label{asymptotic_regularity_sec}
In this section we prove Lemma~\ref{regularity_lem}. We recall that 
$d_0=2(2n)^7p^3$. For general $\tau\in T(S)$, we fix the notation: $\tau = (r_1, r_2, r_3) = (x_1x_2x_3y_1^{-1}, y_1y_2y_3z_1^{-1}, z_1z_2z_3x_1^{-1})$, where $r_l\in W_4(S)$ for $1\leq l\leq 3$ and $x_j, y_j, z_j\in S\cup S^{-1}$ for $1\leq j\leq 3$.

\subsection{Counting triagrams and expected degree of a vertex}

In Section~\ref{asymptotic_regularity_sec} we will frequently use 
the following observation that the number of triagrams with $m$ 
fixed generators is asymptotically equal to $(2n)^{9-m}$ or 
exactly equal to 0. 

\begin{lem}\label{counting_triagrams_lem}
Suppose $0\leq m\leq 9$ and $C$ is a set of $m$ conditions, each of one of the three following forms:
$x_i=a$, $y_i=a$ or $z_i=a$, where $1\leq i\leq 3$, symbols $x_i,y_i,z_i$ are variables, and $a \in S\cup S^{-1}$.
Suppose that every variable appears in at most one condition in $C$. Let $T(S,C)$ be the set of all triagrams $\tau$
over $S$, satisfying all conditions in $C$.
If $T(S,C)$ is non-empty,
then $|T(S,C)|=(2n)^{9-m}\left(1-O_\text{abs}\left(n^{-1}\right)\right)
\leq (2n)^{9-m}$.
\end{lem}

\begin{proof}
First, let $V=\{x_1, x_2, x_3, y_1, y_2, y_3, z_1, z_2, z_3\}$ be a set of 9 variables and 
let us see that triagrams over $S$ are in one-to-one correspondence with 
assignments $\phi:V\longrightarrow S\cup S^{-1}$, satisfying a~set $D$ of 15 inequalities of form $u\neq v$, where $u, v\in V\cup V^{-1}$ are distinct expressions. This follows by 
noticing that each of the conditions required by the definition of a triagram: $x_1x_2x_3y_1^{-1}, y_1y_2y_3z_1^{-1}, z_1z_2z_3x_1^{-1}\in W_4(S)$, is equivalent to 4 inequalities of the
given form. These, together with $x_3\neq y_2^{-1}$, $y_3\neq z_2^{-1}$, and $z_3\neq x_2^{-1}$, define
the~desired set $D$ of size $3\cdot4 + 3=15$.

Now, $|T(S,C)|$ is the number of assignments $\phi:V\longrightarrow S\cup S^{-1}$,
which satisfy all the conditions in $C\cup D$. The equations in $C$ amount to fixing $m$ values of $\phi$, so ${|T(S,C)|\leq (2n)^{9-m}}$. 
If these values violate any of the inequalities in $D$, then $T(S,C)=\emptyset$. Otherwise, without contradicting $D$, we can choose the~remaining $(9-m)$ values of $\phi$ one by one, having for each subsequent variable $t\in V$ at least $(2n-15)$ valid choices of $\phi(t)$, since at 
every stage of this construction every condition in $D$ excludes at most one value of $\phi(t)$. 
Hence, $|T(S,C)|\geq (2n-15)^{9-m}=(2n)^{9-m}\left(1-O_\text{abs}\left(n^{-1}\right)\right)$, which together with 
$|T(S,C)|\leq (2n)^{9-m}$ concludes the argument.
\end{proof}

As its first application, let us use Lemma~\ref
{counting_triagrams_lem} to show that every vertex of each $L_i$, for $1\leq i\leq 3$, has expected degree asymptotically equal to $d_0$.

\begin{rem}\label{degree_rem}
For every $1\leq i\leq 3$ and $w\in W_2$,
\begin{equation*}
\mathbb{E}\left(d_{L_i}(w)\right)=d_0\left(1-O_\text{abs}\left(n^{-1}\right)\right).
\end{equation*}
\end{rem}

\begin{proof}
W.l.o.g.\ assume that $i=1$. Let $w=st\in W_2$, where $s,t\in S\cup S^{-1}$ and $s\neq t^{-1}$.
By Definition~\ref{def_triagram_link_graph}, edges of $L_1$ correspond bijectively to triagrams in $\mathcal{T}$, and
the edge corresponding to any $\tau\in\mathcal{T}$ has distinct endpoints $\partial_2(\tau)^{-1}=x_2^{-1}z_3^{-1}$ and $\partial_3(\tau)=x_3y_2$. This means that 
$\mathbb{E}\left(d_{L_1}(w)\right)=(m_1+m_2)p^3$, where $m_1$ (resp.\ $m_2$) is the number of triagrams $\tau\in T_S$, such that 
$x_2^{-1}z_3^{-1}=w$ (resp.\ $x_3y_2=w$),
i.e.\ $x_2=s^{-1}$ and $z_3=t^{-1}$ (resp.\ $x_3=s$ and $y_2=t$). 
If $n\geq 9$, then it is easy to see that $m_1>0$: for pairwise distinct generators $g_1, g_2, \ldots, g_7\in S\setminus\{s, s^{-1}, t, t^{-1}\}$, the triple $\tau = (g_1s^{-1}g_2g_3^{-1}, g_3g_4g_5g_6^{-1}, g_6g_7t^{-1}g_1^{-1})$ is a~triagram satisfying $x_2=s^{-1}$ and $z_3=t^{-1}$. Analogously, we also have $m_2>0$ for $n\geq 9$. Hence, by~Lemma~\ref{counting_triagrams_lem}, $m_1=(2n)^{7}\left(1-O_\text{abs}\left(n^{-1}\right)\right)$, 
$m_2=(2n)^{7}\left(1-O_\text{abs}\left(n^{-1}\right)\right)$, and so
\begin{equation*}
\mathbb{E}\left(d_{L_1}(w)\right)=(m_1+m_2)p^3=2(2n)^7\left(1-O_\text{abs}\left(n^{-1}\right)\right)p^3
=d_0\left(1-O_\text{abs}\left(n^{-1}\right)\right).
\end{equation*}

\end{proof}

\subsection{Concentration of 3rd degree homogeneous polynomials}
\label{kim-vu_subsec}

The main tool we use to prove Lemma~\ref{regularity_lem}
is the~concentration inequality of Kim--Vu (\cite[Main Theorem]{kim00}), which we now formulate for a~special case of our interest. 

Let $I$ be a finite non-empty set and let $\{X_i:i\in I\}\cup\{Y_i:i\in I\}
\cup\{Z_i:i\in I\}$ be a mutually independent set of $3|I|$ identically 
distributed random variables taking values in the set $\{0,1\}$.
Let $F\subseteq I^3$ be any subset and consider
the random variable 
\begin{equation}\label{V_form_eq}
V=\sum_{(i,j,k)\in F}X_iY_jZ_k.
\end{equation}

Given $i, j, k\in I$, define the following variables, depending on the expression defining $V$, and 
which can be seen as formal partial derivatives of~$V$ with respect to the variables $X_i$, $Y_j$, 
and $Z_k$:

\begin{equation*}
\begin{split}
&V_{X_i}=\sum_{\substack{j',k' \in I:\\ (i,j',k')\in F}} Y_{j'}Z_{k'},\quad
V_{Y_j}=\sum_{\substack{i',k' \in I:\\ (i',j,k')\in F}} X_{i'}Z_{k'},\quad
V_{Z_k}=\sum_{\substack{i',j' \in I:\\ (i',j',k)\in F}} X_{i'}Y_{j'},\\
&V_{Y_jZ_k}=\sum_{\substack{i'\in I:\\ (i',j,k)\in F}} X_{i'}, \quad
V_{X_iZ_k}=\sum_{\substack{j'\in I:\\ (i,j',k)\in F}} Y_{j'}, \quad
V_{X_iY_j}=\sum_{\substack{k'\in I:\\ (i,j,k')\in F}} Z_{k'},
\end{split}
\end{equation*}
and
\begin{equation*}
V_{X_iY_jZ_k}=\begin{cases}
1, & \text{ if }(i,j,k)\in F,\\
0, & \text{ otherwise}.
\end{cases}
\end{equation*}

Denote also the following quantities: 
\begin{equation*}
\begin{split}
E_0(V)&=\mathbb{E}V,\\
E_1(V)&=\max\left(\max_{i\in I}\mathbb{E}V_{X_i},
\max_{j\in I}\mathbb{E}V_{Y_j}, \max_{k\in I}\mathbb{E}V_{Z_k}\right),\\
E_2(V)&=\max\left(\max_{j,k\in I}\mathbb{E}V_{Y_jZ_k},
\max_{i,k\in I}\mathbb{E}V_{X_iZ_k}, \max_{i,j\in I}\mathbb{E}V_{X_iY_j}\right),\\
E_3(V)&=\max_{i,j,k\in I}\mathbb{E}V_{X_iY_jZ_k},\\
E(V)&=\max_{0\leq i\leq 3} E_i(V),\\
E'(V)&=\max_{1\leq i\leq 3}E_i(V).
\end{split}
\end{equation*}

In this notation we have the following inequality.
\begin{thm}[Special case of {\cite[Main Theorem]{kim00}}]
\label{kim_vu_thm}
There exist absolute constants ${a,b>0}$, such that, for any variable
$V$ of form (\ref{V_form_eq}) and any $\lambda>1$,
\begin{equation*}
    \mathbb{P}\left(|V-\mathbb{E}V|>a\left(E(V)E'(V)\right)^{\frac{1}{2}}\lambda^3\right) \leq b\exp\left(-\lambda+2\log (3|I|)\right).
\end{equation*}
\end{thm}
To see that Theorem~\ref{kim_vu_thm} is in fact a special case
of \cite[Main Theorem]{kim00}, assume w.l.o.g.\ that $I=\{1,\ldots,m\}$
for some $m\geq 1$. Let $n=3m$ and $k=3$. Choose $(t_1,t_2,\ldots,t_n)=(X_1,\ldots,X_m,Y_1,\ldots,Y_m,Z_1,\ldots,Z_m)$ and let $H$ be the~hypergraph
with $V(H)=\{1,2,\ldots,n\}$ and $\mathcal{E}(H)=\{\{i,m+j,2m+k\}:(i,j,k)\in F\}$. Declare the weight $w(e)$ of every edge $e\in \mathcal{E}(H)$ to be equal to $1$. Now, in the notation
of \cite{kim00}, $V=Y_H$, $E(V)=E(H)$ and $E'(V)=E'(H)$, so the inequality in Theorem~\ref{kim_vu_thm} holds for $a=a_3$ and an implicit universal constant $b>0$, implied by the~$O(\cdot)$ 
notation on the right-hand side of \cite[Main Theorem]{kim00}.

\subsection{Proof of asymptotic regularity}

In this subsection we prove Lemma~\ref{regularity_lem} by using Theorem~\ref{kim_vu_thm}.
W.l.o.g.\ assume that $i=1$. For every $r\in W_4(S)$
define the random variables: $X_r=\mathds{1}_{r\in R_1}$, $Y_r=\mathds{1}_{r\in R_2}$, and $Z_r=\mathds{1}_{r\in R_3}$. Consider any $w\in W_2(S)$. Let $F_{w,1}=\{\tau\in T(S):\partial_2^{-1}(\tau)=w\text\}=\{\tau\in T(S):z_3x_2=w^{-1}\text\}$, $F_{w,2}=\{\tau\in T(S):\partial_3(\tau)=w\}=\{\tau\in T(S):x_3y_2=w\}$, and $F_w=F_{w,1}\cup F_{w,2}$ Recall that $\partial_2^{-1}(\tau)\neq\partial_3(\tau)$ for every
$\tau\in T(S)$, so the sets $F_{w,1}$ and $F_{w,2}$ are disjoint. Since $L_1$ is constructed from $\mathcal{T}=(R_1\times R_2\times R_3)\cap T(S)$, by Definition~\ref{def_triagram_link_graph} we have
\begin{equation}\label{deg_form_eq}
    d_{L_1}(w)=|F_w\cap\mathcal{T}|=|F_w\cap (R_1\times R_2\times R_3)|=\sum_{(r_1,r_2,r_3)\in F_w} X_{r_1}Y_{r_2}Z_{r_3}.
\end{equation}
As the sets $R_1, R_2, R_3$ are mutually independent, so is the set of $3|W_4(S)|$ variables of form $X_r$, $Y_r$, or $Z_r$, for $r\in W_4(S)$. Hence the random variable $V=d_{L_1}(w)$ is of form (\ref{V_form_eq}) for $I=W_4(S)$ and $F=F_w$. Aiming to use Theorem~\ref{kim_vu_thm}, let us compute $E(V)$ and $E'(V)$ for every $w\in W_2(S)$.

\begin{lem}\label{E_Eprim_lem}
Let $w\in W_2(S)$. Consider the random variable $V=d_{L_1}(w)$ and its presentation in the form~(\ref{deg_form_eq}). Then $E(V)=\mathbb{E}(d_{L_1}(w))$ and $E'(V)=1$ for $n$ sufficiently large, depending only on $\rho$.
\end{lem}
\begin{proof}
We compute $E_0(V)$ and $E_3(V)$, and bound $E_1(V)$ and $E_2(V)$. 

$\boldsymbol{E_0(V)}$. By Remark~\ref{degree_rem}, we have $E_0(V)=\mathbb{E}(d_{L_1}(w))=d_0(1-O_\text{abs}(n^{-1}))=2(2n)^7p^3(1-O_\text{abs}(n^{-1}))\sim\frac{2}{27}(2n)^\alpha$. In particular, $E_0(V)>1$ for $n$ sufficiently large.

$\boldsymbol{E_1(V)}$. We first bound $\max_{r\in W_4}\mathbb{E}V_{X_r}$. 
Fix $r\in W_4(S)$. Since $\mathbb{E}Y_{r_2}Z_{r_3}=\left(\mathbb{E}Y_{r_2}\right)\left(\mathbb{E}Z_{r_3}\right)=p^2$ for any $r_2, r_3\in W_4(S)$, we have
\begin{equation*}
    \mathbb{E}V_{X_r} = \sum_{\substack{r_2,r_3 \in W_4(S):\\ (r,r_2,r_3)\in F_w}} \mathbb{E}Y_{r_2}Z_{r_3}=p^2(m_1+m_2),
\end{equation*}
where, for $1\leq i\leq 2$, $m_i$ is the number of triagrams $\tau=(r_1,r_2,r_3)\in F_{w,i}$, such that $r_1=r$. Let us bound $m_1$. The conditions on $\tau$: $\tau\in F_{w,1}$ and $r_1=r$ are either
contradictory, or together amount to fixing 5 generators defining $\tau$, as shown in
Figure~\ref{bounding_EV_X_fig}. In either case, by Lemma~\ref{counting_triagrams_lem},
$m_1\leq (2n)^{9-5}=(2n)^4$. Analogously, $m_2\leq (2n)^4$.
As these bounds do not depend on the choice of $r$, we have $\max_{r\in W_4(S)}\mathbb{E}V_{X_r}\leq 2p^2(2n)^4$.
\begin{figure}[H]
	\centering
	\begin{tikzpicture}[line cap = round, line join = round]

\def\midarr{0.5}
\def\boundarr{0.5}
\tikzset{boundedgestyle/.style={
        thin,
        decoration={
            markings,
            mark=at position \boundarr with {\arrow{angle 45}}
        },
        postaction={decorate}
}}
\tikzset{midedgestyle/.style={
        thin,
        decoration={
            markings,
            mark=at position \midarr with {\arrow{angle 45}}
        },
        postaction={decorate}
}}
\tikzset{noarrowstyle/.style={
        thick,
        decoration={
            markings
        },
        postaction={decorate}
}}

\def\l{1}

\def\s{1}
\def\t{1.732}
\def\w{0.577}
\def\txt{0.25}

\def\q{0}
\def\p{0}
\def\r{0}

\def\xa{\l*\p}
\def\ya{\l*\q-\l*2*\w*\r}
\draw[midedgestyle, ultra thick] (\xa, \ya) -- (\xa-\l*\t, \ya-\l*\s);
\node at (\xa-\boundarr*\l*\t+\txt*0.5*\l*\s, \ya-\boundarr*\l*\s-\txt*0.5*\t*\l) {$x_1$};
\draw[boundedgestyle, ultra thick] (\xa-\l*\t, \ya-\l*\s) -- (\xa, \ya-2*\l*\s);
\node at (\xa-\l*\t+\boundarr*\l*\t-\txt*0.5*\l*\s, \ya-\l*\s-\boundarr*\l*\s-\txt*0.5*\t*\l) {$x_2$};
\draw[boundedgestyle, ultra thick] (\xa, \ya-2*\l*\s) -- (\xa+\l*\t, \ya-\l*\s);
\node at (\xa+\boundarr*\l*\t+\txt*0.5*\l*\s, \ya-2*\l*\s+\boundarr*\l*\s-\txt*0.5*\t*\l) {$x_3$};

\def\xb{\l*\p+\l*\r}
\def\yb{\l*\q+\l*\w*\r}
\draw[midedgestyle, ultra thick] (\xb, \yb) -- (\xb+\l*\t, \yb-\l*\s);
\node at (\xb+\boundarr*\l*\t+\txt*0.5*\l*\s, \yb-\boundarr*\l*\s+\txt*0.5*\t*\l) {$y_1$};
\draw[boundedgestyle] (\xb+\l*\t, \yb-\l*\s) -- (\xb+\l*\t, \yb+\l*\s);
\node at (\xb+\l*\t+\txt*\l*\s, \yb-\l*\s+\boundarr*2*\l*\s) {$y_2$};
\draw[boundedgestyle] (\xb+\l*\t, \yb+\l*\s) -- (\xb, \yb+2*\l*\s);
\node at (\xb+\l*\t-\boundarr*\l*\t+\txt*0.5*\l*\s, \yb+\l*\s+\boundarr*\l*\s+\txt*0.5*\t*\l) {$y_3$};

\def\xc{\l*\p-\l*\r}
\def\yc{\l*\q+\l*\w*\r}
\draw[midedgestyle] (\xc, \yc) -- (\xc, \yc+2*\l*\s);
\node at (\xc-\txt*\l*\s, \yc+\boundarr*2*\l*\s) {$z_1$};
\draw[boundedgestyle] (\xc, \yc+2*\l*\s) -- (\xc-\l*\t, \yc+\l*\s);
\node at (\xc-\boundarr*\l*\t-\txt*0.5*\l*\s, \yc+2*\l*\s-\boundarr*\l*\s+\txt*0.5*\t*\l) {$z_2$};
\draw[boundedgestyle, ultra thick] (\xc-\l*\t, \yc+\l*\s) -- (\xc-\l*\t, \yc-\l*\s);
\node at (\xc-\l*\t-\txt*\l*\s, \yc+\l*\s-\boundarr*2*\l*\s) {$z_3$};

\def\q{0}
\def\p{5}
\def\r{0}

\def\xa{\l*\p}
\def\ya{\l*\q-\l*2*\w*\r}
\draw[midedgestyle, ultra thick] (\xa, \ya) -- (\xa-\l*\t, \ya-\l*\s);
\node at (\xa-\boundarr*\l*\t+\txt*0.5*\l*\s, \ya-\boundarr*\l*\s-\txt*0.5*\t*\l) {$x_1$};
\draw[boundedgestyle, ultra thick] (\xa-\l*\t, \ya-\l*\s) -- (\xa, \ya-2*\l*\s);
\node at (\xa-\l*\t+\boundarr*\l*\t-\txt*0.5*\l*\s, \ya-\l*\s-\boundarr*\l*\s-\txt*0.5*\t*\l) {$x_2$};
\draw[boundedgestyle, ultra thick] (\xa, \ya-2*\l*\s) -- (\xa+\l*\t, \ya-\l*\s);
\node at (\xa+\boundarr*\l*\t+\txt*0.5*\l*\s, \ya-2*\l*\s+\boundarr*\l*\s-\txt*0.5*\t*\l) {$x_3$};

\def\xb{\l*\p+\l*\r}
\def\yb{\l*\q+\l*\w*\r}
\draw[midedgestyle, ultra thick] (\xb, \yb) -- (\xb+\l*\t, \yb-\l*\s);
\node at (\xb+\boundarr*\l*\t+\txt*0.5*\l*\s, \yb-\boundarr*\l*\s+\txt*0.5*\t*\l) {$y_1$};
\draw[boundedgestyle, ultra thick] (\xb+\l*\t, \yb-\l*\s) -- (\xb+\l*\t, \yb+\l*\s);
\node at (\xb+\l*\t+\txt*\l*\s, \yb-\l*\s+\boundarr*2*\l*\s) {$y_2$};
\draw[boundedgestyle] (\xb+\l*\t, \yb+\l*\s) -- (\xb, \yb+2*\l*\s);
\node at (\xb+\l*\t-\boundarr*\l*\t+\txt*0.5*\l*\s, \yb+\l*\s+\boundarr*\l*\s+\txt*0.5*\t*\l) {$y_3$};

\def\xc{\l*\p-\l*\r}
\def\yc{\l*\q+\l*\w*\r}
\draw[midedgestyle] (\xc, \yc) -- (\xc, \yc+2*\l*\s);
\node at (\xc-\txt*\l*\s, \yc+\boundarr*2*\l*\s) {$z_1$};
\draw[boundedgestyle] (\xc, \yc+2*\l*\s) -- (\xc-\l*\t, \yc+\l*\s);
\node at (\xc-\boundarr*\l*\t-\txt*0.5*\l*\s, \yc+2*\l*\s-\boundarr*\l*\s+\txt*0.5*\t*\l) {$z_2$};
\draw[boundedgestyle] (\xc-\l*\t, \yc+\l*\s) -- (\xc-\l*\t, \yc-\l*\s);
\node at (\xc-\l*\t-\txt*\l*\s, \yc+\l*\s-\boundarr*2*\l*\s) {$z_3$};

\end{tikzpicture}
	\vspace{1em}
	\caption{Bounding $\mathbb{E}V_{X_r}$. 
        On the left, thickened segments represent generators fixed by the conditions $r_1=r$ and $\tau\in F_{w,1}$. On the right, generators fixed by the conditions $r_1=r$ and $\tau\in F_{w,2}$.}
	\label{bounding_EV_X_fig}
\end{figure}
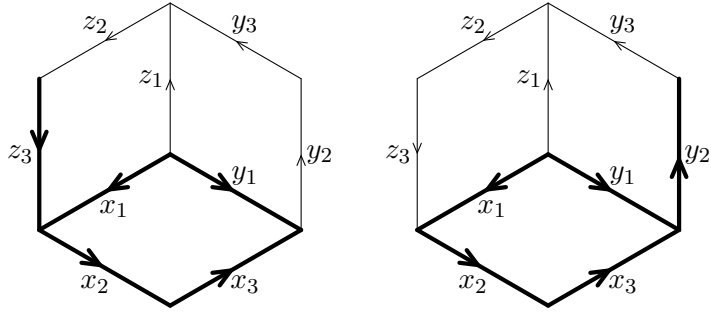

Next we bound $\max_{r\in W_4(S)}\mathbb{E}V_{Y_r}$. Once again, fix $r\in W_4(S)$. We have
\begin{equation*}
    \mathbb{E}V_{Y_r} = \sum_{\substack{r_1,r_3 \in W_4(S):\\ (r_1,r,r_3)\in F_w}} \mathbb{E}X_{r_1}Z_{r_3}=p^2(m_1+m_2),
\end{equation*}
where, for $1\leq i\leq 2$, $m_i$ is the number of triagrams $\tau=
(r_1,r_2,r_3)\in F_{w,i}$, such that $r_2=r$. 
The conditions on $\tau$: $\tau\in F_{w,1}$ and $r_2=r$ (resp.\ $\tau\in F_{w,2}$ and $r_2=r$) are contradictory or fix together 6 (resp.\ 5) generators defining $\tau$, as shown in Figure~\ref{bounding_EV_Y_fig}. Hence, 
by Lemma~\ref{counting_triagrams_lem}, $m_1\leq (2n)^3$ and $m_2\leq (2n)^4$.
As these bounds do not depend on $r$, we get $\max_{r\in W_4(S)}\mathbb{E}V_{Y_r}\leq 2p^2(2n)^4$.
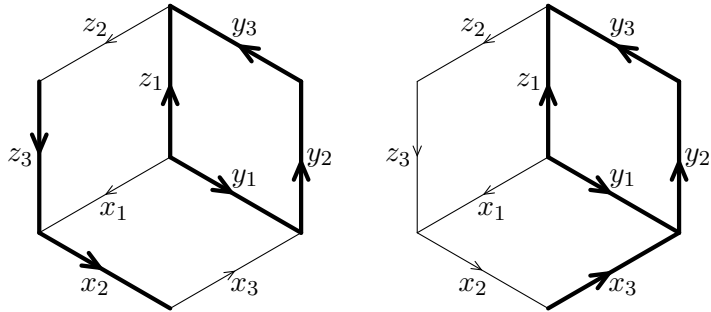
\begin{figure}[H]
	\centering
	\begin{tikzpicture}[line cap = round, line join = round]

\def\midarr{0.5}
\def\boundarr{0.5}
\tikzset{boundedgestyle/.style={
        thin,
        decoration={
            markings,
            mark=at position \boundarr with {\arrow{angle 45}}
        },
        postaction={decorate}
}}
\tikzset{midedgestyle/.style={
        thin,
        decoration={
            markings,
            mark=at position \midarr with {\arrow{angle 45}}
        },
        postaction={decorate}
}}
\tikzset{noarrowstyle/.style={
        thick,
        decoration={
            markings
        },
        postaction={decorate}
}}

\def\l{1}

\def\s{1}
\def\t{1.732}
\def\w{0.577}
\def\txt{0.25}

\def\q{0}
\def\p{0}
\def\r{0}

\def\xa{\l*\p}
\def\ya{\l*\q-\l*2*\w*\r}
\draw[midedgestyle] (\xa, \ya) -- (\xa-\l*\t, \ya-\l*\s);
\node at (\xa-\boundarr*\l*\t+\txt*0.5*\l*\s, \ya-\boundarr*\l*\s-\txt*0.5*\t*\l) {$x_1$};
\draw[boundedgestyle, ultra thick] (\xa-\l*\t, \ya-\l*\s) -- (\xa, \ya-2*\l*\s);
\node at (\xa-\l*\t+\boundarr*\l*\t-\txt*0.5*\l*\s, \ya-\l*\s-\boundarr*\l*\s-\txt*0.5*\t*\l) {$x_2$};
\draw[boundedgestyle] (\xa, \ya-2*\l*\s) -- (\xa+\l*\t, \ya-\l*\s);
\node at (\xa+\boundarr*\l*\t+\txt*0.5*\l*\s, \ya-2*\l*\s+\boundarr*\l*\s-\txt*0.5*\t*\l) {$x_3$};

\def\xb{\l*\p+\l*\r}
\def\yb{\l*\q+\l*\w*\r}
\draw[midedgestyle, ultra thick] (\xb, \yb) -- (\xb+\l*\t, \yb-\l*\s);
\node at (\xb+\boundarr*\l*\t+\txt*0.5*\l*\s, \yb-\boundarr*\l*\s+\txt*0.5*\t*\l) {$y_1$};
\draw[boundedgestyle, ultra thick] (\xb+\l*\t, \yb-\l*\s) -- (\xb+\l*\t, \yb+\l*\s);
\node at (\xb+\l*\t+\txt*\l*\s, \yb-\l*\s+\boundarr*2*\l*\s) {$y_2$};
\draw[boundedgestyle, ultra thick] (\xb+\l*\t, \yb+\l*\s) -- (\xb, \yb+2*\l*\s);
\node at (\xb+\l*\t-\boundarr*\l*\t+\txt*0.5*\l*\s, \yb+\l*\s+\boundarr*\l*\s+\txt*0.5*\t*\l) {$y_3$};

\def\xc{\l*\p-\l*\r}
\def\yc{\l*\q+\l*\w*\r}
\draw[midedgestyle, ultra thick] (\xc, \yc) -- (\xc, \yc+2*\l*\s);
\node at (\xc-\txt*\l*\s, \yc+\boundarr*2*\l*\s) {$z_1$};
\draw[boundedgestyle] (\xc, \yc+2*\l*\s) -- (\xc-\l*\t, \yc+\l*\s);
\node at (\xc-\boundarr*\l*\t-\txt*0.5*\l*\s, \yc+2*\l*\s-\boundarr*\l*\s+\txt*0.5*\t*\l) {$z_2$};
\draw[boundedgestyle, ultra thick] (\xc-\l*\t, \yc+\l*\s) -- (\xc-\l*\t, \yc-\l*\s);
\node at (\xc-\l*\t-\txt*\l*\s, \yc+\l*\s-\boundarr*2*\l*\s) {$z_3$};

\def\q{0}
\def\p{5}
\def\r{0}

\def\xa{\l*\p}
\def\ya{\l*\q-\l*2*\w*\r}
\draw[midedgestyle] (\xa, \ya) -- (\xa-\l*\t, \ya-\l*\s);
\node at (\xa-\boundarr*\l*\t+\txt*0.5*\l*\s, \ya-\boundarr*\l*\s-\txt*0.5*\t*\l) {$x_1$};
\draw[boundedgestyle] (\xa-\l*\t, \ya-\l*\s) -- (\xa, \ya-2*\l*\s);
\node at (\xa-\l*\t+\boundarr*\l*\t-\txt*0.5*\l*\s, \ya-\l*\s-\boundarr*\l*\s-\txt*0.5*\t*\l) {$x_2$};
\draw[boundedgestyle, ultra thick] (\xa, \ya-2*\l*\s) -- (\xa+\l*\t, \ya-\l*\s);
\node at (\xa+\boundarr*\l*\t+\txt*0.5*\l*\s, \ya-2*\l*\s+\boundarr*\l*\s-\txt*0.5*\t*\l) {$x_3$};

\def\xb{\l*\p+\l*\r}
\def\yb{\l*\q+\l*\w*\r}
\draw[midedgestyle, ultra thick] (\xb, \yb) -- (\xb+\l*\t, \yb-\l*\s);
\node at (\xb+\boundarr*\l*\t+\txt*0.5*\l*\s, \yb-\boundarr*\l*\s+\txt*0.5*\t*\l) {$y_1$};
\draw[boundedgestyle, ultra thick] (\xb+\l*\t, \yb-\l*\s) -- (\xb+\l*\t, \yb+\l*\s);
\node at (\xb+\l*\t+\txt*\l*\s, \yb-\l*\s+\boundarr*2*\l*\s) {$y_2$};
\draw[boundedgestyle, ultra thick] (\xb+\l*\t, \yb+\l*\s) -- (\xb, \yb+2*\l*\s);
\node at (\xb+\l*\t-\boundarr*\l*\t+\txt*0.5*\l*\s, \yb+\l*\s+\boundarr*\l*\s+\txt*0.5*\t*\l) {$y_3$};

\def\xc{\l*\p-\l*\r}
\def\yc{\l*\q+\l*\w*\r}
\draw[midedgestyle, ultra thick] (\xc, \yc) -- (\xc, \yc+2*\l*\s);
\node at (\xc-\txt*\l*\s, \yc+\boundarr*2*\l*\s) {$z_1$};
\draw[boundedgestyle] (\xc, \yc+2*\l*\s) -- (\xc-\l*\t, \yc+\l*\s);
\node at (\xc-\boundarr*\l*\t-\txt*0.5*\l*\s, \yc+2*\l*\s-\boundarr*\l*\s+\txt*0.5*\t*\l) {$z_2$};
\draw[boundedgestyle] (\xc-\l*\t, \yc+\l*\s) -- (\xc-\l*\t, \yc-\l*\s);
\node at (\xc-\l*\t-\txt*\l*\s, \yc+\l*\s-\boundarr*2*\l*\s) {$z_3$};

\end{tikzpicture}
	\vspace{1em}
	\caption{Bounding $\mathbb{E}V_{Y_r}$. 
        On the left, generators fixed by the conditions $r_2=r$ and $\tau\in F_{w,1}$. On the right, generators fixed by the conditions $r_2=r$ and $\tau\in F_{w,2}$.}
		\label{bounding_EV_Y_fig}
\end{figure}

Finally, reasoning entirely analogously to how we have bounded 
$\max_{r\in W_4(S)}\mathbb{E}V_{Y_r}$, we obtain
that $\max_{r\in W_4(S)}\mathbb{E}V_{Z_r}\leq 2p^2(2n)^4$.

In summary, since $\alpha\leq 1$, 
\begin{equation*}
\begin{split}
E_1(V)&=\max\left(\max_{r\in W_4(S)}\mathbb{E}V_{X_r},
\max_{r\in W_4(S)}\mathbb{E}V_{Y_r}, \max_{r\in W_4(S)}\mathbb{E}V_{Z_r}\right)\\
&\leq 2p^2(2n)^4=\frac{2}{9}(2n)^{-\frac{2}{3}+\frac{2}{3}\alpha}(1+o_\rho(1))\leq\frac{2}{9}+o_\rho(1),
\end{split}
\end{equation*}
so $E_1(V)<1$ for $n$ sufficiently large.

$\boldsymbol{E_2(V)}$. We first bound $\max_{r,r'\in W_4(S)}\mathbb{E}V_{Y_rZ_{r'}}$. Fix $r, r'\in W_4(S)$. We have

\begin{equation*}
    \mathbb{E}V_{Y_rZ_{r'}}=\sum_{\substack{r_1 \in W_4(S):\\ (r_1,r,r')\in F_w}} \mathbb{E}X_{r_1}=p(m_1+m_2),
\end{equation*}
where $m_i$, for $1\leq i\leq 2$, is the number of triagrams $\tau=(r_1, r_2, r_3)
\in F_{w,i}$, such that $r_2=r$ and $r_3=r'$. Let us first bound $m_1$. The conditions
on $\tau$: $\tau\in F_{w,1}$, $r_2=r$, and $r_3=r'$, are contradictory or 
fix together 8 generators defining $\tau$, as shown in Figure \ref{bounding_EV_YZ_fig}. Hence, by Lemma~\ref{counting_triagrams_lem}, $m_1\leq 2n$. Analogously, $m_2\leq 2n$. 
As these bounds do not depend on $r$ and $r'$, we have $\max_{r,r'\in W_4(S)}\mathbb{E}V_{Y_rZ_{r'}}
\leq 4pn$.
\begin{figure}[H]
	\centering
	\begin{tikzpicture}[line cap = round, line join = round]

\def\midarr{0.5}
\def\boundarr{0.5}
\tikzset{boundedgestyle/.style={
        thin,
        decoration={
            markings,
            mark=at position \boundarr with {\arrow{angle 45}}
        },
        postaction={decorate}
}}
\tikzset{midedgestyle/.style={
        thin,
        decoration={
            markings,
            mark=at position \midarr with {\arrow{angle 45}}
        },
        postaction={decorate}
}}
\tikzset{noarrowstyle/.style={
        thick,
        decoration={
            markings
        },
        postaction={decorate}
}}

\def\l{1}

\def\s{1}
\def\t{1.732}
\def\w{0.577}
\def\txt{0.25}

\def\q{0}
\def\p{0}
\def\r{0}

\def\xa{\l*\p}
\def\ya{\l*\q-\l*2*\w*\r}
\draw[midedgestyle, ultra thick] (\xa, \ya) -- (\xa-\l*\t, \ya-\l*\s);
\node at (\xa-\boundarr*\l*\t+\txt*0.5*\l*\s, \ya-\boundarr*\l*\s-\txt*0.5*\t*\l) {$x_1$};
\draw[boundedgestyle, ultra thick] (\xa-\l*\t, \ya-\l*\s) -- (\xa, \ya-2*\l*\s);
\node at (\xa-\l*\t+\boundarr*\l*\t-\txt*0.5*\l*\s, \ya-\l*\s-\boundarr*\l*\s-\txt*0.5*\t*\l) {$x_2$};
\draw[boundedgestyle] (\xa, \ya-2*\l*\s) -- (\xa+\l*\t, \ya-\l*\s);
\node at (\xa+\boundarr*\l*\t+\txt*0.5*\l*\s, \ya-2*\l*\s+\boundarr*\l*\s-\txt*0.5*\t*\l) {$x_3$};

\def\xb{\l*\p+\l*\r}
\def\yb{\l*\q+\l*\w*\r}
\draw[midedgestyle, ultra thick] (\xb, \yb) -- (\xb+\l*\t, \yb-\l*\s);
\node at (\xb+\boundarr*\l*\t+\txt*0.5*\l*\s, \yb-\boundarr*\l*\s+\txt*0.5*\t*\l) {$y_1$};
\draw[boundedgestyle, ultra thick] (\xb+\l*\t, \yb-\l*\s) -- (\xb+\l*\t, \yb+\l*\s);
\node at (\xb+\l*\t+\txt*\l*\s, \yb-\l*\s+\boundarr*2*\l*\s) {$y_2$};
\draw[boundedgestyle, ultra thick] (\xb+\l*\t, \yb+\l*\s) -- (\xb, \yb+2*\l*\s);
\node at (\xb+\l*\t-\boundarr*\l*\t+\txt*0.5*\l*\s, \yb+\l*\s+\boundarr*\l*\s+\txt*0.5*\t*\l) {$y_3$};

\def\xc{\l*\p-\l*\r}
\def\yc{\l*\q+\l*\w*\r}
\draw[midedgestyle, ultra thick] (\xc, \yc) -- (\xc, \yc+2*\l*\s);
\node at (\xc-\txt*\l*\s, \yc+\boundarr*2*\l*\s) {$z_1$};
\draw[boundedgestyle, ultra thick] (\xc, \yc+2*\l*\s) -- (\xc-\l*\t, \yc+\l*\s);
\node at (\xc-\boundarr*\l*\t-\txt*0.5*\l*\s, \yc+2*\l*\s-\boundarr*\l*\s+\txt*0.5*\t*\l) {$z_2$};
\draw[boundedgestyle, ultra thick] (\xc-\l*\t, \yc+\l*\s) -- (\xc-\l*\t, \yc-\l*\s);
\node at (\xc-\l*\t-\txt*\l*\s, \yc+\l*\s-\boundarr*2*\l*\s) {$z_3$};

\def\q{0}
\def\p{5}
\def\r{0}

\def\xa{\l*\p}
\def\ya{\l*\q-\l*2*\w*\r}
\draw[midedgestyle, ultra thick] (\xa, \ya) -- (\xa-\l*\t, \ya-\l*\s);
\node at (\xa-\boundarr*\l*\t+\txt*0.5*\l*\s, \ya-\boundarr*\l*\s-\txt*0.5*\t*\l) {$x_1$};
\draw[boundedgestyle] (\xa-\l*\t, \ya-\l*\s) -- (\xa, \ya-2*\l*\s);
\node at (\xa-\l*\t+\boundarr*\l*\t-\txt*0.5*\l*\s, \ya-\l*\s-\boundarr*\l*\s-\txt*0.5*\t*\l) {$x_2$};
\draw[boundedgestyle, ultra thick] (\xa, \ya-2*\l*\s) -- (\xa+\l*\t, \ya-\l*\s);
\node at (\xa+\boundarr*\l*\t+\txt*0.5*\l*\s, \ya-2*\l*\s+\boundarr*\l*\s-\txt*0.5*\t*\l) {$x_3$};

\def\xb{\l*\p+\l*\r}
\def\yb{\l*\q+\l*\w*\r}
\draw[midedgestyle, ultra thick] (\xb, \yb) -- (\xb+\l*\t, \yb-\l*\s);
\node at (\xb+\boundarr*\l*\t+\txt*0.5*\l*\s, \yb-\boundarr*\l*\s+\txt*0.5*\t*\l) {$y_1$};
\draw[boundedgestyle, ultra thick] (\xb+\l*\t, \yb-\l*\s) -- (\xb+\l*\t, \yb+\l*\s);
\node at (\xb+\l*\t+\txt*\l*\s, \yb-\l*\s+\boundarr*2*\l*\s) {$y_2$};
\draw[boundedgestyle, ultra thick] (\xb+\l*\t, \yb+\l*\s) -- (\xb, \yb+2*\l*\s);
\node at (\xb+\l*\t-\boundarr*\l*\t+\txt*0.5*\l*\s, \yb+\l*\s+\boundarr*\l*\s+\txt*0.5*\t*\l) {$y_3$};

\def\xc{\l*\p-\l*\r}
\def\yc{\l*\q+\l*\w*\r}
\draw[midedgestyle, ultra thick] (\xc, \yc) -- (\xc, \yc+2*\l*\s);
\node at (\xc-\txt*\l*\s, \yc+\boundarr*2*\l*\s) {$z_1$};
\draw[boundedgestyle, ultra thick] (\xc, \yc+2*\l*\s) -- (\xc-\l*\t, \yc+\l*\s);
\node at (\xc-\boundarr*\l*\t-\txt*0.5*\l*\s, \yc+2*\l*\s-\boundarr*\l*\s+\txt*0.5*\t*\l) {$z_2$};
\draw[boundedgestyle, ultra thick] (\xc-\l*\t, \yc+\l*\s) -- (\xc-\l*\t, \yc-\l*\s);
\node at (\xc-\l*\t-\txt*\l*\s, \yc+\l*\s-\boundarr*2*\l*\s) {$z_3$};

\end{tikzpicture}
	\vspace{1em}
	\caption{Bounding $\mathbb{E}V_{Y_rZ_{r'}}$. 
        On the left, generators fixed by $r_2=r$, $r_3=r'$ and $\tau\in F_{w,1}$. On the right, generators fixed by $r_2=r$, $r_3=r'$ and $\tau\in F_{w,2}$.}
	\label{bounding_EV_YZ_fig}
\end{figure}
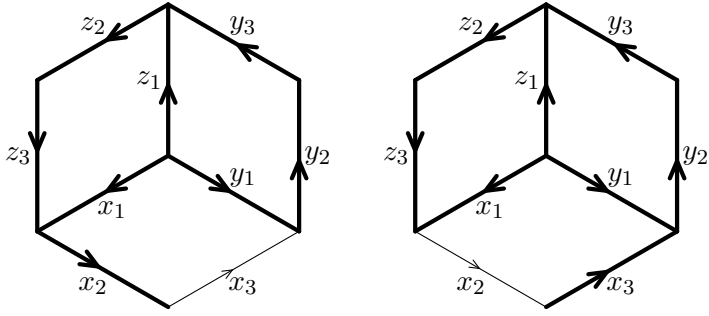

Next we bound $\max_{r,r'\in W_4(S)}\mathbb{E}V_{X_rZ_{r'}}$. Fix $r,r'\in W_4(S)$. We have
\begin{equation*}
    \mathbb{E}V_{X_rZ_{r'}}=\sum_{\substack{r_2 \in W_4(S):\\ (r,r_2,r')\in F_w}} \mathbb{E}Y_{r_2}=p(m_1+m_2),
\end{equation*}
where, for $1\leq i\leq 2$, $m_i$ is the number of triagrams $\tau=(r_1,r_2,r_3)\in F_{w,i}$, such that $r_1=r$ and $r_3=r'$. The conditions on $\tau$: $\tau\in F_{w,1},\ r_1=r$, and $r_3=r'$ 
(resp.\ $\tau\in F_{w,2},\ r_1=r$, and $r_3=r'$) are contradictory or together fix 7 (resp.\ 8)
generators defining $\tau$, as shown in Figure~\ref{bounding_EV_XZ_fig}. Hence,
by Lemma~\ref{counting_triagrams_lem}, $m_1\leq (2n)^2$ and $m_2\leq 2n$.
As these bounds do not depend on $r$ and $r'$, we have $\max_{r,r'\in W_4(S)}\mathbb{E}V_{X_rZ_{r'}}\leq 2p(2n)^2$.
\begin{figure}[H]
	\centering
	\begin{tikzpicture}[line cap = round, line join = round]

\def\midarr{0.5}
\def\boundarr{0.5}
\tikzset{boundedgestyle/.style={
        thin,
        decoration={
            markings,
            mark=at position \boundarr with {\arrow{angle 45}}
        },
        postaction={decorate}
}}
\tikzset{midedgestyle/.style={
        thin,
        decoration={
            markings,
            mark=at position \midarr with {\arrow{angle 45}}
        },
        postaction={decorate}
}}
\tikzset{noarrowstyle/.style={
        thick,
        decoration={
            markings
        },
        postaction={decorate}
}}

\def\l{1}

\def\s{1}
\def\t{1.732}
\def\w{0.577}
\def\txt{0.25}

\def\q{0}
\def\p{0}
\def\r{0}

\def\xa{\l*\p}
\def\ya{\l*\q-\l*2*\w*\r}
\draw[midedgestyle, ultra thick] (\xa, \ya) -- (\xa-\l*\t, \ya-\l*\s);
\node at (\xa-\boundarr*\l*\t+\txt*0.5*\l*\s, \ya-\boundarr*\l*\s-\txt*0.5*\t*\l) {$x_1$};
\draw[boundedgestyle, ultra thick] (\xa-\l*\t, \ya-\l*\s) -- (\xa, \ya-2*\l*\s);
\node at (\xa-\l*\t+\boundarr*\l*\t-\txt*0.5*\l*\s, \ya-\l*\s-\boundarr*\l*\s-\txt*0.5*\t*\l) {$x_2$};
\draw[boundedgestyle, ultra thick] (\xa, \ya-2*\l*\s) -- (\xa+\l*\t, \ya-\l*\s);
\node at (\xa+\boundarr*\l*\t+\txt*0.5*\l*\s, \ya-2*\l*\s+\boundarr*\l*\s-\txt*0.5*\t*\l) {$x_3$};

\def\xb{\l*\p+\l*\r}
\def\yb{\l*\q+\l*\w*\r}
\draw[midedgestyle, ultra thick] (\xb, \yb) -- (\xb+\l*\t, \yb-\l*\s);
\node at (\xb+\boundarr*\l*\t+\txt*0.5*\l*\s, \yb-\boundarr*\l*\s+\txt*0.5*\t*\l) {$y_1$};
\draw[boundedgestyle] (\xb+\l*\t, \yb-\l*\s) -- (\xb+\l*\t, \yb+\l*\s);
\node at (\xb+\l*\t+\txt*\l*\s, \yb-\l*\s+\boundarr*2*\l*\s) {$y_2$};
\draw[boundedgestyle] (\xb+\l*\t, \yb+\l*\s) -- (\xb, \yb+2*\l*\s);
\node at (\xb+\l*\t-\boundarr*\l*\t+\txt*0.5*\l*\s, \yb+\l*\s+\boundarr*\l*\s+\txt*0.5*\t*\l) {$y_3$};

\def\xc{\l*\p-\l*\r}
\def\yc{\l*\q+\l*\w*\r}
\draw[midedgestyle, ultra thick] (\xc, \yc) -- (\xc, \yc+2*\l*\s);
\node at (\xc-\txt*\l*\s, \yc+\boundarr*2*\l*\s) {$z_1$};
\draw[boundedgestyle, ultra thick] (\xc, \yc+2*\l*\s) -- (\xc-\l*\t, \yc+\l*\s);
\node at (\xc-\boundarr*\l*\t-\txt*0.5*\l*\s, \yc+2*\l*\s-\boundarr*\l*\s+\txt*0.5*\t*\l) {$z_2$};
\draw[boundedgestyle, ultra thick] (\xc-\l*\t, \yc+\l*\s) -- (\xc-\l*\t, \yc-\l*\s);
\node at (\xc-\l*\t-\txt*\l*\s, \yc+\l*\s-\boundarr*2*\l*\s) {$z_3$};

\def\q{0}
\def\p{5}
\def\r{0}

\def\xa{\l*\p}
\def\ya{\l*\q-\l*2*\w*\r}
\draw[midedgestyle, ultra thick] (\xa, \ya) -- (\xa-\l*\t, \ya-\l*\s);
\node at (\xa-\boundarr*\l*\t+\txt*0.5*\l*\s, \ya-\boundarr*\l*\s-\txt*0.5*\t*\l) {$x_1$};
\draw[boundedgestyle, ultra thick] (\xa-\l*\t, \ya-\l*\s) -- (\xa, \ya-2*\l*\s);
\node at (\xa-\l*\t+\boundarr*\l*\t-\txt*0.5*\l*\s, \ya-\l*\s-\boundarr*\l*\s-\txt*0.5*\t*\l) {$x_2$};
\draw[boundedgestyle, ultra thick] (\xa, \ya-2*\l*\s) -- (\xa+\l*\t, \ya-\l*\s);
\node at (\xa+\boundarr*\l*\t+\txt*0.5*\l*\s, \ya-2*\l*\s+\boundarr*\l*\s-\txt*0.5*\t*\l) {$x_3$};

\def\xb{\l*\p+\l*\r}
\def\yb{\l*\q+\l*\w*\r}
\draw[midedgestyle, ultra thick] (\xb, \yb) -- (\xb+\l*\t, \yb-\l*\s);
\node at (\xb+\boundarr*\l*\t+\txt*0.5*\l*\s, \yb-\boundarr*\l*\s+\txt*0.5*\t*\l) {$y_1$};
\draw[boundedgestyle, ultra thick] (\xb+\l*\t, \yb-\l*\s) -- (\xb+\l*\t, \yb+\l*\s);
\node at (\xb+\l*\t+\txt*\l*\s, \yb-\l*\s+\boundarr*2*\l*\s) {$y_2$};
\draw[boundedgestyle] (\xb+\l*\t, \yb+\l*\s) -- (\xb, \yb+2*\l*\s);
\node at (\xb+\l*\t-\boundarr*\l*\t+\txt*0.5*\l*\s, \yb+\l*\s+\boundarr*\l*\s+\txt*0.5*\t*\l) {$y_3$};

\def\xc{\l*\p-\l*\r}
\def\yc{\l*\q+\l*\w*\r}
\draw[midedgestyle, ultra thick] (\xc, \yc) -- (\xc, \yc+2*\l*\s);
\node at (\xc-\txt*\l*\s, \yc+\boundarr*2*\l*\s) {$z_1$};
\draw[boundedgestyle, ultra thick] (\xc, \yc+2*\l*\s) -- (\xc-\l*\t, \yc+\l*\s);
\node at (\xc-\boundarr*\l*\t-\txt*0.5*\l*\s, \yc+2*\l*\s-\boundarr*\l*\s+\txt*0.5*\t*\l) {$z_2$};
\draw[boundedgestyle, ultra thick] (\xc-\l*\t, \yc+\l*\s) -- (\xc-\l*\t, \yc-\l*\s);
\node at (\xc-\l*\t-\txt*\l*\s, \yc+\l*\s-\boundarr*2*\l*\s) {$z_3$};

\end{tikzpicture}
	\vspace{1em}
	\caption{Bounding $\mathbb{E}V_{X_rZ_{r'}}$. 
        On the left, generators fixed by $r_1=r$, $r_3=r'$ and $\tau\in F_{w,1}$. On the right, generators fixed by $r_1=r$, $r_3=r'$ and $\tau\in F_{w,2}$.}
	\label{bounding_EV_XZ_fig}
\end{figure}
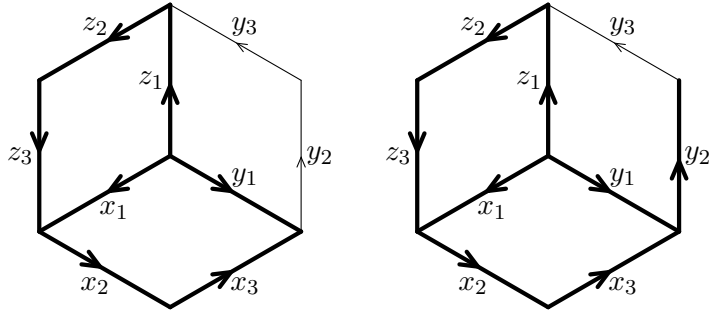

Finally, reasoning entirely symmetrically to how we have bounded $\max_{r,r'\in 
W_4(S)}\mathbb{E}V_{X_rZ_{r'}}$, we obtain that 
$\max_{r,r'\in W_4(S)}\mathbb{E}V_{X_rY_{r'}}\leq 2p(2n)^2$.

In summary, since $\alpha \leq 1$:
\begin{equation*}
\begin{split}
E_2(V)&=\max\left(\max_{r,r'\in W_4(S)}\mathbb{E}V_{Y_rZ_{r'}},
\max_{r,r'\in W_4(S)}\mathbb{E}V_{X_rZ_{r'}}, \max_{r,r'\in W_4(S)}\mathbb{E}V_{X_rY_{r'}}\right)\\
&\leq\max(4pn, 2p(2n)^2)=2p(2n)^2\\
&=\frac{2}{3}(2n)^{-\frac{1}{3}+\frac{\alpha}{3}}(1+o_\rho(1))\leq\frac{2}{3}+o_\rho(1),
\end{split}
\end{equation*}
so $E_2(V)<1$ for $n$ sufficiently large.

$\boldsymbol{E_3(V)}$. Fix $r_1, r_2, r_3\in W_4(S)$. The random
variable $V_{X_{r_1}Y_{r_2}Z_{r_3}}$ is constant, so
\begin{equation*}
\mathbb{E}V_{X_{r_1}Y_{r_2}Z_{r_3}}=V_{X_{r_1}Y_{r_2}Z_{r_3}}=\begin{cases}
1, & \text{ if }(r_1,r_2,r_3)\in F_w,\\
0, & \text{ otherwise.}
\end{cases}
\end{equation*}
As in the proof of Remark~\ref{degree_rem}, since $w\in W_2(S)$, 
so $F_w\neq\emptyset$ for $n\geq 9$.
In particular,
\begin{equation*}
E_3(V)=\max_{r_1,r_2,r_3\in W_4(S)}\mathbb{E}V_{X_{r_1}Y_{r_2}Z_{r_3}}=1.
\end{equation*}

We have thus proved, for $n$ sufficiently large, that $E_0(V)>1$, $E_1(V)<1$, $E_2(V)<1$, and $E_3(V)=1$. Note
that our arguments show in fact that these inequalities hold as soon as $n\geq n_0$ for some $n_0$, depending only on $\rho$.
Hence, for $n\geq n_0$ we have
\begin{equation*}
E(V)=\max_{0\leq i\leq 3}E_i(V)=E_0(V)=\mathbb{E}(d_{L_1}(w))
\end{equation*}
and
\begin{equation*}
E'(V)=\max_{1\leq i\leq 3}E_i(V)=1.
\end{equation*}
\end{proof}

We are now in the position to show Lemma~\ref{regularity_lem}. 

\begin{proof}[Proof of~Theorem~\ref{regularity_lem}]
Let $\delta=n^{-\frac{\alpha}{3}}$ and
\begin{equation*}
\eta=\max_{v\in W_2(S)} \left| \frac{\mathbb{E}\left(d_{L_1}(v)\right)-d_0}{d_0} \right|.
\end{equation*}
By Remark~\ref{degree_rem}, we have $\eta=O_\text{abs}\left(n^{-1}\right)$. (It is not difficult to see that $\mathbb{E}\left(d_{L_1}(v)\right)$
can, for given $n$ and $\rho$, take at most two different values, depending on whether
$v=st$ for $s=t$ or $s\neq t$, where $s,t\in S\cup S^{-1}$. Hence $\eta$
is a~maximum of at most two distinct numbers, but we do not need this fact in our argument.)

Fix $w\in W_2(S)$. We aim to use Theorem~\ref{kim_vu_thm} to bound from above the probability that $|d_{L_1}(w)-d_0|>\delta d_0$.
By definition of $\eta$, we have $|\mathbb{E}\left(d_{L_1}(w)\right)-d_0|\leq \eta d_0$, so if it occurs that $|d_{L_1}(w)-d_0|>\delta d_0$,
then also $|d_{L_1}(w)-\mathbb{E}\left(d_{L_1}(w)\right)|>(\delta-\eta)d_0$. Hence, 
\begin{equation}\label{proof_of_reg_lem_ineq}
\mathbb{P}\Big(|d_{L_1}(w)-d_0|>\delta d_0\Big)
\leq \mathbb{P}\Big(|d_{L_1}(w)-\mathbb{E}\left(d_{L_1}(w)\right)|>(\delta-\eta)d_0\Big).
\end{equation}

Consider the~presentation of $V=d_{L_1}(w)$ in the form (\ref{deg_form_eq}).
By Lemma~\ref{E_Eprim_lem}, there exists $n_0$, depending only on $\rho$,
such that $E(V)=\mathbb{E}\left(d_{L_1}(w)\right)$ and $E'(V)=1$ if $n\geq n_0$.
We can bound $\mathbb{E}(d_{L_1}(w))\leq(1+\eta)d_0$, so if $n\geq n_0$ and $\delta>\eta$, then
\begin{equation}\label{rewriting_eta_lam_ineq}
(\delta-\eta)d_0\geq (\delta-\eta)d_0^{\frac{1}{2}}\left(\frac{\mathbb{E}(d_{L_1}(w))}{1+\eta}\right)^{\frac{1}{2}}=a(E(V)E'(V))^{\frac{1}{2}}\lambda^3,
\end{equation}
where $a$ is the~constant from Theorem~\ref{kim_vu_thm} and we denote
\begin{equation*}
\lambda = \frac{d_0^{\frac{1}{6}}(\delta-\eta)^{\frac{1}{3}}}{a^{\frac{1}{3}}(1+\eta)^\frac{1}{6}}.
\end{equation*}
Since $\alpha\leq 1$, we have $\delta-\eta \sim n^{-\frac{\alpha}{3}}$.
Additionally $d_0=2(2n)^7p^3\sim \frac{2}{27}(2n)^\alpha$, so
\begin{equation}\label{lambda_as_sim}
\lambda \sim \frac{\left(\frac{2}{27}(2n)^\alpha\right)^{\frac{1}{6}}
\left(n^{-\frac{\alpha}{3}}\right)^\frac{1}{3}}{a^{\frac{1}{3}}}
\sim C_\alpha n^{\frac{\alpha}{18}},
\end{equation}
where $C_\alpha>0$ is a constant, depending only on $\alpha$. 
Hence, there exists $n_1\geq n_0$, depending only on $\rho$, 
such that $\delta>\eta$ and $\lambda>1$ for $n\geq n_1$.

By combining (\ref{proof_of_reg_lem_ineq}), (\ref{rewriting_eta_lam_ineq}), and 
using Theorem~\ref{kim_vu_thm}, we obtain the following inequalities:

\begin{equation*}
\begin{split}
\mathbb{P}\Big(|d_{L_1}(w)-d_0|>\delta d_0\Big)
&\leq \mathbb{P}\Big(|V-\mathbb{E}V|>(\delta-\eta)d_0\Big)\\
&\leq \mathbb{P}\Big(|V-\mathbb{E}V|>a(E(V)E'(V))^\frac{1}{2}\lambda^3\Big)\\
&\leq b\exp\left(-\lambda+2\log(3|W_4|)\right),
\end{split}
\end{equation*}
which hold for every $w\in W_2(S)$ and $n\geq n_1$. Recall that $\mathcal{R}_{1,\delta}$ is the event that $|d_{L_1}(w)-d_0|\leq\delta d_0$ for every $w\in W_2(S)$, so
\begin{equation*}
\mathbb{P}\left(\mathcal{R}_{1,\delta}^c\right)\leq \sum_{w\in W_2} \mathbb{P}\Big(|d_{L_1}(w)-d_0|>\delta d_0\Big)\leq b|W_2(S)|\exp\left(-\lambda+2\log(3|W_4(S)|)\right)
\end{equation*}
for $n\geq n_1$. Using (\ref{lambda_as_sim}) and the inequalities $|W_2(S)|\leq (2n)^2$, $|W_4(S)|\leq(2n)^4$, we obtain further that
\begin{equation*}
\begin{split}
\mathbb{P}\left(\mathcal{R}_{1,\delta}^c\right)
&\leq b(2n)^2\exp\left(-C_\alpha n^{\frac{\alpha}{18}}(1+o_\rho(1)) +2\log\left(3(2n)^4\right)\right)\\
&= \exp\left(-C_\alpha n^{\frac{\alpha}{18}}(1+o_\rho(1)) + D\log{n}+E\right)\\
&\leq \exp\left(-C'_\alpha n^{\frac{\alpha}{18}}\right)
\end{split}
\end{equation*}
for some absolute constants $D,E$, a~constant $C'_\alpha>0$, depending only on $\alpha$,
and $n$ sufficiently large, depending only on $\rho$. In particular, for every $\beta>0$, we have $\mathbb{P}\left(\mathcal{R}_{1,\delta}\right)=1-o_{\beta, \rho}\left(n^{-\beta}\right)$.

\end{proof}

\section{Bounding the conformal dimension}\label{bound_conf_dim_sec}

In this section we prove Theorem~\ref{confdim_bounds_square_thm}.
The proof we present essentially repeats the~arguments given in \cite[Subsection 10.2]{dru19} for random 
triangular groups. Our contribution here is showing property $(\mathrm{F}L^p)$ for a~random group in the square model, for $p$ increasing with $n$ (Theorem~\ref{main_thm_FLp}).
We will also use the fact that a~random group in the square model is a.a.s.\ non-elementary hyperbolic at any density $d<\frac{1}{2}$.
We separately prove lower and upper bounds of~Theorem~\ref{confdim_bounds_square_thm} in Subsections~\ref{confdim_lower_bound_subsec}
and \ref{confdim_upper_bound_subsec}, respectively.

\subsection{Lower bound}\label{confdim_lower_bound_subsec}

We want to show that, for $d\in\left(\frac{5}{12}, \frac{1}{2}\right)$, a~random group $\Gamma$ in the model $\mathcal{Q}\left(n,d\right)$ a.a.s.\ satisfies
\begin{equation*}
	f_d(\log{n})^{\frac{1}{2}}\leq\operatorname{Confdim}(\partial\Gamma),
\end{equation*}
where $f_d>0$ is a~constant, depending only on $d$. 
If $f_d$ is chosen so that is satisfies Theorem~\ref{main_thm_FLp}, 
then the inequality is an immediate consequence of the following 
theorem, which is a~corollary to \cite[th\'eor\`eme 0.1]{bou16}.
\begin{thm}\label{conformal_FLp_lower_bound_thm}
	If $\Gamma$ is a~non-elementary hyperbolic group with property $(\mathrm{F}L^p)$,
	then \newline ${\operatorname{Confdim(\partial\Gamma)}\geq p}$.
\end{thm}
Note that, in the original notation of \cite{bou16}, $\operatorname{Confdim}\left(\partial\Gamma\right)$
refers actually to the so-called \emph{Ahlfors regular conformal dimension} of $\partial\Gamma$, which in general 
can be larger than the~conformal dimension we defined in Subsection~\ref
{intro_confdim_subsec}. However, these definitions coincide when $\Gamma$ is a~finitely-generated non-elementary 
hyperbolic group. This follows from \cite[Proposition~2.2.6]{mac10}, where the~required uniform perfectness assumption
can be straighforwardly justified by using an~appropriate measure on $\partial\Gamma$, which is described 
by \cite[th\'eor\`eme~5.4, corollaire~5.5 and proposition~7.4]{coo93}.

\subsection{Upper bound}\label{confdim_upper_bound_subsec}
In this subsection we will assume only that $d\in\left(0,\frac{1}{2}\right)$ and prove that
a~random group in the model $\mathcal{Q}(n,d)$ a.a.s.\ satisfies
\begin{equation}\label{confdim_lower_bound_ineq}
\operatorname{Confdim}\left(\partial\Gamma\right)\leq F_d\log{n},
\end{equation}
where $F_d>0$ is a~constant depending only on $d$. 

We first recall that random groups in the square model satisfy the 
following linear isoperimetric inequality.
\begin{thm}[{\cite[Theorem 3.14]{odr16a}}]\label{isoperim_square_thm}
	Let $d\in\left(0,\frac{1}{2}\right)$ and $\eps>0$. A.a.s.\ all reduced van Kampen diagrams $D$ associated to the~presentation of a~random group in the~square model at density $d$, satisfy
	\begin{equation}\label{isoperim_square_eq}
	|\partial D| \geq 4(1-2d-\eps)|D|.
	\end{equation}
\end{thm}
This inequality, by \cite[Proposition~15]{oll07}, implies a~quantitative version of \cite[Corollary 3.15]{odr16a}:

\begin{lem}\label{hyperbolicity_lem}
	For $d\in \left(0,\frac{1}{2}\right)$, a random group $\Gamma = \lle S|R\rre$ in the square model at density $d$ is a.a.s.\ $\delta$-hyperbolic w.r.t.\ $S$ for
	$\delta=\frac{17}{1-2d}$.
\end{lem}
\begin{proof}
	Let $\eps \in (0, 1-2d)$.
	By rewriting $4|D|=\mathcal{A}(D)$ in Theorem~\ref{isoperim_square_thm}, we see that a.a.s.\ a~random presentation $\lle S|R\rre$ in the square model at density $d$ satisfies the assumptions of \cite[Proposition~15]{oll07} for $C=1-2d-\eps$ and $\lambda=4$, so the group $\Gamma = \lle S|R\rre$ is $\delta$-hyperbolic w.r.t.~$S$ for $\delta=\frac{16}{1-2d-\eps}$. Choosing sufficiently small $\eps$, we can obtain $\delta \leq \frac{17}{1-2d}$.
\end{proof}

Now, using Lemma \ref{hyperbolicity_lem}, we can retrace the proof of~\cite[Proposition~1.7]{mac12} to prove (\ref{confdim_lower_bound_ineq})
for $F_d=\frac{200}{1-2d}$ (see also the proof of \cite[Proposition 10.6]{dru19}).
\begin{lem}
Suppose $d\in\left(0,\frac{1}{2}\right)$. A.a.s.\ a random group $\Gamma$ in the model $\mathcal{Q}(n,d)$ satisfies

\begin{equation*}
\operatorname{Confdim}(\partial\Gamma)\leq \frac{200}{1-2d}\log{n}.
\end{equation*}
\end{lem}
\begin{proof}
Let $\delta = \frac{17}{1-2d}$. By Lemma~\ref{hyperbolicity_lem} and using \cite[III.H.3.21]{bri99}, 
$\Gamma$ is a.a.s.\ $\delta$-hyperbolic and the~boundary $\partial\Gamma$ admits a~visual metric with visual 
exponent $\eps = \frac{\log{2}}{4\delta} = \frac{(1-2d)\log{2}}{68}\geq \frac{1-2d}{100}$. 
Hausdorff dimension of $\partial\Gamma$ with this metric is equal to
$\frac{1}{\eps} h(\Gamma)$, where $h(\Gamma)$ is the volume
entropy of $\Gamma$ \cite[th\'eor\`eme 5.4]{coo93}. Since $\Gamma$ has $n$ generators,
$h(\Gamma)\leq \log{(2n-1)}$, and
\begin{equation*}
\operatorname{Confdim}(\partial\Gamma)\leq \frac{1}{\eps}h(\Gamma) \leq \frac{100}{1-2d}\log{(2n-1)}
\leq \frac{200}{1-2d}\log{n},
\end{equation*}
where we used the fact that $2n-1\leq n^2$ for $n\geq 1$.
\end{proof}

\end{document}